\documentclass[11pt]{amsart}
\usepackage{amsmath,amssymb,amsthm,mathrsfs,enumerate,bm,xcolor,multirow,pbox,setspace}
\usepackage[colorlinks,linkcolor=red,citecolor=blue,urlcolor=blue]{hyperref}
\allowdisplaybreaks[4]

\numberwithin{equation}{section}
\newcommand{\N}{\mathbb{N}}
\newcommand{\R}{\mathbb{R}}

\newcommand{\E}{\mathbb{E}}
\newcommand{\Prob}{\mathbb{P}}

\newcommand{\pnorm}[2]{\lVert#1\rVert_{#2}}

\newcommand{\biggpnorm}[2]{\bigg\lVert#1\bigg\rVert_{#2}}
\newcommand{\abs}[1]{\lvert#1\rvert}

\newcommand{\biggabs}[1]{\bigg\lvert#1\bigg\rvert}

\renewcommand{\epsilon}{\varepsilon}

\renewcommand{\d}[1]{\mathrm{d}#1}

\newcommand{\floor}[1]{\left\lfloor #1 \right\rfloor}

\DeclareMathOperator{\rank}{rank}
\DeclareMathOperator{\trace}{tr}

\AtBeginDocument{%
	\def\MR#1{}
}

\newcommand{\beq}{\begin{equation}}
\newcommand{\eeq}{\end{equation}}
\newcommand{\beqa}{\begin{equation} \begin{aligned}}
\newcommand{\eeqa}{\end{aligned} \end{equation}}
\newcommand{\beqas}{\begin{equation*} \begin{aligned}}
\newcommand{\eeqas}{\end{aligned} \end{equation*}}

\newcommand{\bit}{\begin{itemize}}
	\newcommand{\eit}{\end{itemize}}
\newcommand{\bmat}{\begin{bmatrix}}
	\newcommand{\emat}{\end{bmatrix}}

\theoremstyle{definition}\newtheorem{problem}{Problem}[section]
\theoremstyle{definition}\newtheorem{definition}[problem]{Definition}
\theoremstyle{remark}\newtheorem{assumption}{Assumption}

\theoremstyle{remark}\newtheorem{remark}[problem]{Remark}
\theoremstyle{definition}\newtheorem{example}[problem]{Example}
\theoremstyle{plain}\newtheorem{theorem}[problem]{Theorem}
\theoremstyle{plain}\newtheorem{lemma}[problem]{Lemma}
\theoremstyle{plain}\newtheorem{proposition}[problem]{Proposition}
\theoremstyle{plain}\newtheorem{corollary}[problem]{Corollary}
\theoremstyle{plain}

\begin{document}

\title{Oracle posterior contraction rates under hierarchical priors}
\author[Q. Han]{Qiyang Han}

\address[Q. Han]{
	Department of Statistics, Rutgers University, Piscataway, NJ 08854, USA.
}
\email{qh85@stat.rutgers.edu}

\thanks{Supported in part by NSF Grant DMS-1566514. }

\date{\today}

\maketitle

\begin{abstract}
We offer a general Bayes theoretic framework to derive posterior contraction rates under a hierarchical prior design: the first-step prior serves to assess the model selection uncertainty, and the second-step prior quantifies the prior belief on the strength of the signals within the model chosen from the first step. In particular, we establish non-asymptotic oracle posterior contraction rates under (i) a local Gaussianity condition on the log likelihood ratio of the statistical experiment, (ii) a local entropy condition on the dimensionality of the models, and  (iii) a sufficient mass condition on the second-step prior near the best approximating signal for each model. The first-step prior can be designed generically. The posterior distribution enjoys Gaussian tail behavior and therefore the resulting posterior mean also satisfies an oracle inequality, automatically serving as an adaptive point estimator in a frequentist sense. Model mis-specification is allowed in these oracle rates. 

The local Gaussianity condition serves as a unified attempt of non-asymptotic Gaussian quantification of the experiments, and can be easily verified in various experiments considered in \cite{ghosal2007convergence} and beyond. The general results are applied in various problems including: (i) trace regression, (ii) shape-restricted isotonic/convex regression, (iii) high-dimensional partially linear regression, (iv) covariance matrix estimation in the sparse factor model, (v) detection of non-smooth polytopal image boundary, and (vi) intensity estimation in a Poisson point process model. These new results serve either as theoretical justification of practical prior proposals in the literature, or as an illustration of the generic construction scheme of a (nearly) minimax adaptive estimator for a complicated experiment.
\end{abstract}


\section{Introduction}

\subsection{Overview}

Suppose we observe $X^{(n)}$ from a statistical experiment $(\mathfrak{X}^{(n)}, \mathcal{A}^{(n)}, P_{f}^{(n)})$, where $f $ belongs to a statistical model $ \mathcal{F}$ and $\{P_f^{(n)}\}_{f  \in \mathcal{F}}$ is dominated by a $\sigma$-finite measure $\mu$. In many cases, instead of using a single `big' model $\mathcal{F}$, a collection of suitably nested (sub-)models $\{\mathcal{F}_m\}_{m \in \mathcal{I}}\subset \mathcal{F}$ are available to statisticians. A hierarchical Bayesian approach assigns a first-step prior $\Lambda_n$ assessing the uncertainty in which model to use, followed by a second-step prior $\Pi_{n,m}$ quantifying the prior belief in the strength of the signals within the specific chosen model $\mathcal{F}_m$ from the first step.

Such a hierarchical prior design is intrinsic in many proposals for different problems, including the canonical Gaussian white noise/regression and density estimation \cite{arbel2013bayesian,belitser2003adaptive,deJonge2010adaptive,ghosal2008nonparametric,kruijer2010adaptive,lember2007universal,rousseau2017asymptotic,scricciolo2006convergence}, and the more recent sparse linear regression \cite{castillo2015bayesian,castillo2012needles}, trace regression \cite{alquier2014bayesian}, shape restricted regression \cite{hannah2011bayesian,holmes2003generalized}, covariance matrix estimation \cite{gao2015ratepca,pati2014posterior}, etc. Despite many contraction rates available for different models (see e.g. \cite{castillo2014bayesian,castillo2015bayesian,castillo2012needles,ghosal2000convergence,ghosal2007convergence,ghosal2017fundamentals,hoffmann2015adaptive,rousseau2010rates,shen2001rates,van2008rates,van2009adaptive} for some key contributions), a unified theoretical understanding towards the behavior of posterior distributions under the hierarchical prior design has been limited. \cite{ghosal2008nonparametric} focused on designing adaptive Bayes procedures with models primarily indexed by the smoothness level of function classes in the context of density estimation. Their conditions are complicated and seem not directly applicable to other settings. \cite{deJonge2010adaptive} uses a specific location mixture prior for regression/density estimation/classification. \cite{arbel2013bayesian} considered a more general setting where the models are indexed by functions that admit a linear $\ell_2$-basis structure (e.g. Sobolev/Besov type); see also \cite{rousseau2017asymptotic}. \cite{gao2015general} designed a prior specific to structured linear problems in the Gaussian regression model, with their main focus on high-dimensional (linear) and network problems. As such, all these results apriori require certain specific form of the prior, the model structure, or the statistical experiments.

The goal of this paper aims at giving a unified theoretical treatment of deriving posterior contraction rates under the common hierarchical prior design, without specifying particular forms for the prior, the model structure, or the experiments. More specifically, we aim at identifying common \emph{structural assumptions} on the statistical experiments $(\mathfrak{X}^{(n)}, \mathcal{A}^{(n)}, P_{f}^{(n)})$, the collection of models $\{\mathcal{F}_m\}$ and the priors $\{\Lambda_n\}$ and $\{\Pi_{n,m}\}$ such that
the posterior distribution both
\begin{enumerate}
	\item[(G1)] contracts at an \emph{oracle} rate with respect to some metric\footnote{The requirement of being a metric can be weakened.} $d_n$:
	\begin{align}\label{eqn:post_cont_rate_generic}
	\inf_{ m \in \mathcal{I}} \bigg(\inf_{g \in \mathcal{F}_m} d_n^2(f_0,g) + \mathrm{pen}(m)\bigg),
	\end{align}
	where $\mathrm{pen}(m)$\footnote{$\mathrm{pen}(m)$ may depend on $n$ but we suppress this dependence for notational convenience.} is related to the `dimension' of $\mathcal{F}_m$, and
	\item[(G2)] puts little mass on models that are substantially larger than the oracle one balancing the bias-variance tradeoff in (\ref{eqn:post_cont_rate_generic}).
\end{enumerate} 
The oracle formulation (\ref{eqn:post_cont_rate_generic}) follows the convention in the frequentist literature on model selection \cite{barron1991minimum,yang1998asymptotic,barron1999risk,massart2007concentration,tsybakov2014aggregation}, and has several advantages: (i) (\emph{minimaxity}) if the true signal $f_0$ can be well-approximated by the models $\{\mathcal{F}_m\}$, the contraction rate in (\ref{eqn:post_cont_rate_generic}) is usually (nearly) minimax optimal, (ii) (\emph{adaptivity}) if $f_0$ lies in certain low-dimensional model $\mathcal{F}_m$, the contraction rate adapts to this unknown information, and (iii) (\emph{mis-specification}) if the models $\mathcal{F}_m$ are mis-specified while $d_n^2(f_0,\cup_{m \in \mathcal{I}}\mathcal{F}_m)$ remains `small', then the contraction rate  should still be rescued by this relatively `small' bias. 

As the main abstract result of this paper (cf. Theorem \ref{thm:general_ms}), we show that our goals (G1)-(G2) can be accomplished under:
\begin{enumerate}
	\item[(i)] (\textbf{Experiment}) a local Gaussianity condition on the log likelihood ratio for the statistical experiment with respect to $d_n$;
	\item[(ii)] (\textbf{Models}) a dimensionality condition of the model $\mathcal{F}_m$ measured in terms of local entropy with respect to the metric $d_n$;
	\item[(iii)] (\textbf{Priors}) exponential weighting for the first-step prior $\Lambda_n$, and sufficient mass of the second-step prior $\Pi_{n,m}$ near the `best' approximating signal $f_{0,m}$ within the model $\mathcal{F}_m$ for the true signal $f_0$.
\end{enumerate} 
The local Gaussianity condition is rooted in the frequentist theory of the convergence rates of $M$-estimators (i.e. estimators maximizing certain likelihood) via the theory of Gaussian and empirical processes. In fact, the local Gaussianity serves as an essential ingredient for various (by-now standard) techniques, including the Gaussian concentration and the chaining with bracketing, that give a unification to the theory for, e.g. regression and density estimation \cite{birge1993rates,van2000empirical,van1996weak} (see Appendix \ref{section:MLE} for more discussions). 

From the Bayesian theoretic side, one important convention in studying posterior contraction rates in the literature has been the construction of appropriate tests with exponentially small type I and II errors with respect to certain metric, the Gaussian behavior of type II error being particularly crucial \cite{ghosal2000convergence,ghosal2007convergence}. It is rather curious if the frequentist local Gaussianity can also be useful in the Bayes theory. Our formulation in (i) can be viewed as an attempt in this regard, and seems useful in that, local Gaussianity with respect to the intrinsic metric is a rather universal property in various statistical experiments including the ones considered in \cite{ghosal2007convergence} and beyond: Gaussian/Laplace/binary/Poisson regression, density estimation, Gaussian autoregression, Gaussian time series, covariance matrix estimation, image boundary detection, and support boundary recovery in a Poisson point process model, etc. Moreover, such local Gaussianity naturally entails the Gaussian tail behavior of the posterior distribution, thereby complementing a recent result of \cite{hoffmann2015adaptive} who showed that such a Gaussian tail behavior cannot be uniformly improved under uniform posterior consistency.

Conditions (ii) and (iii) are familiar in Bayes nonparametrics literature. In particular, the first-step prior can be designed generically (cf. Proposition \ref{prop:prior_generic}). Sufficient mass of the second-step prior $\Pi_{n,m}$ is a minimal condition in the sense that using $\Pi_{n,m}$ alone should lead to a (nearly) optimal posterior contraction rate on the model $\mathcal{F}_m$.

As an illustration of the scope of our general results in concrete applications, we justify the prior proposals in (i) \cite{alquier2014bayesian,alquier2015bayesian} for the trace regression problem, and in (ii) \cite{hannah2011bayesian,holmes2003generalized} for the shape-restricted regression problems. Despite many theoretical results for Bayes high-dimensional models (cf. \cite{banerjee2014posterior,castillo2015bayesian,castillo2012needles,gao2015general,gao2015ratepca,pati2014posterior}), it seems that the important low-rank trace regression problem has not yet been successfully addressed. Our result here fills in this gap. Furthermore, to the best knowledge of the author, the theoretical results concerning shape-restricted regression problems provide the first systematic approach that bridges the gap between Bayesian nonparametrics and shape-restricted nonparametric function estimation literature in the context of adaptive estimation\footnote{Almost completed at the same time, \cite{mariucci2017bayesian} considered a Bayes approach for univariate log-concave density estimation, where they derived contraction rates without addressing the adaptation issue. }.

Several other applications are considered, including: (iii) high-dimensional partially linear regression model, (iv) covariance matrix estimation in the sparse factor model, (v) detection of polytopal image boundary, and (vi) estimation of piecewise constant intensity in a Poisson point process model. These results serve as an illustration of the generic construction scheme of a (nearly) minimax adaptive estimator in multi-structured experiments, or in experiments that seem far from Gaussian. We also revisit some density estimation problems, in particular in the location mixture models. The purpose of this is to provide some guidance of how the local Gaussianity can be applied via appropriate localization of the parameter space, when such Gaussianity may fail to hold at a global scale. 

During the preparation of this paper, we become aware of a very recent paper \cite{yang2017bayesian} who independently considered a similar problem. Both our approach and \cite{yang2017bayesian} shed light on the behavior of Bayes procedures under hierarchical priors, while differing in several important aspects (cf. Remark \ref{rmk:comparison_yang}). Moreover, our work here applies to a wide range of applications that are not covered by \cite{yang2017bayesian}.

\subsection{Notation}

Let $(\mathcal{F},\pnorm{\cdot}{})$ be a subset of the normed space of real functions $f:\mathcal{X}\to \R$. Let $\mathcal{N}(\epsilon,\mathcal{F},\pnorm{\cdot}{})$ be the $\epsilon$-covering number; see page 83 of \cite{van1996weak} for more details. For a real-valued measurable function $f$ defined on $(\mathcal{X},\mathcal{A},P)$, $\pnorm{f}{L_p(P)}\equiv \big(P\abs{f}^p)^{1/p}$ denotes the usual $L_p$-norm under $P$ (where $p\geq 1$), and will be simplified as $\pnorm{f}{p}$ when there is no potential confusion. $\pnorm{f}{\infty}\equiv\pnorm{f}{L_\infty}\equiv  \sup_{x \in \mathcal{X}} \abs{f(x)}$ denotes the supremum norm.

For any $v \in \R^d$, we use $\pnorm{v}{p}$ to denote the usual Euclidean $p$-norm. For any $\epsilon>0$, denote $B_d(v,\epsilon)\equiv \{u \in \R^d: \pnorm{u-v}{2}\leq \epsilon\}$ the Euclidean ball in $\R^d$ centered at $v$ with radius $\epsilon$.

$C_{x}$ denotes a generic constant that depends only on $x$, whose numeric value may change from line to line. $a\lesssim_{x} b$ and $a\gtrsim_x b$ mean $a\leq C_x b$ and $a\geq C_x b$ respectively, and $a\asymp_x b$ means $a\lesssim_{x} b$ and $a\gtrsim_x b$. 
For $a,b \in \R$, $a\vee b:=\max\{a,b\}$ and $a\wedge b:=\min\{a,b\}$. $P_{f}^{(n)} T$ denotes the expectation of a random variable $T=T(X^{(n)})$ under the experiment $(\mathfrak{X}^{(n)},\mathcal{A}^{(n)}, P_f^{(n)})$.

\subsection{Organization}
Section \ref{section:gen_frame} is devoted to the general results on oracle posterior contraction rates. We work out a wide range of experiments and some concrete applications that fit into our general theory in Section \ref{section:models}. Detailed proofs are deferred to the Appendix.

\section{General results}\label{section:gen_frame}
In the hierarchical prior design framework, we first put a prior $\Lambda_n$ on the model index $\mathcal{I}$, followed by a prior $\Pi_{n,m}$ on the model $\mathcal{F}_m$ chosen from the first step. The overall prior is a probability measure on $\mathcal{F}$ given by $
\Pi_n\equiv \sum_{m \in \mathcal{I}} \lambda_{n}(m)\Pi_{n,m}$.
The posterior distribution is then a random measure on $\mathcal{F}$: for a measurable subset $B \subset \mathcal{F}$, 
\begin{align}\label{eqn:post_dist}
\Pi_n(B|X^{(n)})& = {\int_B  p_f^{(n)}(X^{(n)})\ \d{\Pi_n(f)}}\bigg/{\int p_f^{(n)}(X^{(n)})\ \d{\Pi_n(f)}}
\end{align}
where $p_f^{(n)}(\cdot)$ denotes the probability density function of $P_f^{(n)}$ with respect to the dominating measure $\mu$.

\subsection{Assumptions}

For some $v>0,c\in [0,\infty)$ let
\begin{align}
\psi_{v,c}(\lambda)=v\lambda^2\cdot \bm{1}_{\abs{\lambda}\leq 1/c}+ \infty\cdot \bm{1}_{\abs{\lambda}>1/c}
\end{align}
denote the local quadratic function.

\begin{assumption}[\textbf{Experiment: Local Gaussianity condition}]\label{assump:laplace_cond_kl}
	There exist some constants $c_1>0$ and $\kappa=(\kappa_g,\kappa_\Gamma) \in (0,\infty)\times [0,\infty)$ such that for all $n \in \N, \lambda \in \R$, and $f_0,f_1 \in \mathcal{F}$, 
	\begin{align*}
	P_{f_0}^{(n)} e^{\lambda \left(\log ({p_{f_0}^{(n)}}/{p_{f_1}^{(n)}})-P_{f_0}^{(n)}\log ({p_{f_0}^{(n)}}/{p_{f_1}^{(n)}})\right)}\leq c_1 e^{\psi_{\kappa_g nd_n^2(f_0,f_1),\kappa_\Gamma }(\lambda)}.
	\end{align*}
	Here  $d_n:\mathcal{F}\times \mathcal{F}\to \R_{\geq 0}$ is a symmetric function satisfying 
	\begin{align}\label{cond:d_n_KL}
	\big(c_2\cdot  d_n^2(f_0,f_1)-d_0^2\big)_+\leq n^{-1}P_{f_0}^{(n)}\log ({p_{f_0}^{(n)}}/{p_{f_1}^{(n)}})\leq c_3 \cdot d_n^2(f_0,f_1)+d_0^2,
	\end{align}
	for some constants $c_2,c_3>0$ and $d_0\geq 0$ (possibly depending on $n$).
\end{assumption}

In Assumption \ref{assump:laplace_cond_kl}, we require the log likelihood ratio to have local Gaussian behavior with respect to the intrinsic `metric' $d_n$ in the sense of (\ref{cond:d_n_KL}). If $\kappa_\Gamma$ can be chosen to be $0$, then the log likelihood ratio exhibits global Gaussian behavior. In Section \ref{section:models}, many statistical experiments, beyond the apparent Gaussian ones, will be shown to satisfy this local Gaussianity condition in their respective intrinsic metrics. In some cases the local Gaussianity by itself may entail certain apriori compactness constraints on the parameter space, for instance boundedness requirements for the parameter space in binary/Poisson regression and density estimation. These constraints can be removed, in a technical way, by working with appropriately localized subsets of the parameter space on which the local Gaussianity holds. See Section \ref{section:localization} and Appendix \ref{section:more_examples} for more details and examples in this regard. 

As already mentioned in the Introduction, this local Gaussianity point of view has its root in the unified treatment of deriving convergence rates of $M$-estimators---a formal connection to the theory of sieved MLE under local Gaussianity will be given in Appendix \ref{section:MLE}.

A direct consequence of the local Gaussianity of the statistical experiment is the following.

\begin{lemma}\label{lem:local_test_general}
	Let Assumption \ref{assump:laplace_cond_kl} hold. For any $f_0,f_1 \in \mathcal{F}$ such that $d_n(f_0,f_1)\geq \sqrt{2/(c_2\wedge c_3)} \cdot d_0$, there exists some test $\phi_n$ such that
	\begin{align*}
	\sup_{f \in \mathcal{F}: d_n^2(f,f_1)\leq c_5d_n^2(f_0,f_1)}\big(P_{f_0}^{(n)} \phi_n+P_f^{(n)}(1-\phi_n)\big)\leq c_6e^{-c_7 nd_n^2(f_0,f_1)}
	\end{align*}
	where $c_5\leq 1/4,c_6\in [2,\infty)$ and $c_7\in (0,1)$ only depends on the constants in Assumption \ref{assump:laplace_cond_kl}.
\end{lemma}

Next we state the assumption on the complexity of the models $\{\mathcal{F}_m\}_{m \in \mathcal{I}}$. Let $\mathcal{I}=\N^q$ be a $q$-dimensional lattice with the natural order $(\mathcal{I},\leq)$\footnote{For any $a,b \in \mathcal{I}$, $a\leq b$ iff $a_i\leq b_i$ for all $1\leq i\leq q$. Similar definition applies to $<,\geq,>$. }. Here the dimension $q$ is understood as \emph{the number of different structures} in the models $\{\mathcal{F}_m\}_{m \in \mathcal{I}}$. For instance, in the trace regression problem (cf. Section \ref{section:trace_reg}), there is only one rank structure so $q=1$; in the covariance matrix estimation problem in the sparse factor model (cf. Section \ref{section:cov_matrix_est}), there are both rank and sparsity structures so $q=2$. In the sequel we will not explicitly mention $q$ unless otherwise specified. We require the models to be nested in the sense that $\mathcal{F}_m\subset \mathcal{F}_{m'}$ if and only if $m\leq m'$ \footnote{Nesting requirement is for simplicity; see Appendix \ref{section:more_examples} for examples of non-nesting models.}.

Let $f_{0,m}$ denote the `best' approximation of $f_0$ within the model $\mathcal{F}_m$ in the sense that $f_{0,m} \in \arg\inf_{g \in \mathcal{F}_m} d_n(f_0,g)$\footnote{We assume that $f_{0,m}$ is well-defined without loss of generality.}. Our assumption on the model complexity below, at a heuristic level, says that $\mathcal{F}_m$ has dimension $n\delta_{n,m}^2$ measured in a local entropy sense, for some $\delta_{n,m}>0$. In typical cases, $\mathcal{F}_m$ has `dimension' $m$, and $\delta_{n,m}^2\approx \frac{m}{n}\times \textrm{poly-log}$ is regarded as the contraction rate on $\mathcal{F}_m$ (up to logarithmic factors).

\begin{assumption}[\textbf{Models: Local entropy condition}]\label{assump:local_ent_general}
	Let $\{\delta_{n,m}\}_{m \in \mathcal{I}}\subset \R_{>0}$ be such that each $\delta_{n,m}$ depends on $n,m$ only, and:
	\begin{itemize}
		\item For each $m \in \mathcal{I}$, 
		\begin{align}\label{cond:entropy_essen}
		1+\sup_{\epsilon>\delta_{n,m}}\log \mathcal{N}\big(c_5 \epsilon , \{f \in \mathcal{F}_m: d_n(f,g)\leq 2\epsilon\}, d_n\big)\leq (c_7/2)n \delta_{n,m}^2
		\end{align}
		holds for all $g \in \{f_{0,m'}\}_{m'\leq m}$. 
		\item 	Furthermore there exist some constants $\mathfrak{c} \in [1,\infty), \gamma \in [1, \infty), \mathfrak{h}_0\in[1,\infty]$ such that 
		for any $m \in \mathcal{I}$, $\alpha\geq c_7/2$ and any $1\leq h\leq \mathfrak{h}_0$,
		\begin{align}\label{cond:suplinearity_delta}
		\sum_{m'\geq h m} e^{-\alpha n\delta_{n,m'}^2}\leq 2e^{-\alpha n h\delta_{n,m}^2/\mathfrak{c}^2},
		\quad \mathfrak{c}^{-2}\delta_{n,hm}^2\leq h^{\gamma}\delta_{n,m}^2.
		\end{align}
	\end{itemize}
	
\end{assumption}

Using $\delta_{n,m}$'s, the models can be divided into over-fitting or under-fitting ones according to whether $\delta_{n,m}^2\geq \inf_{g\in \mathcal{F}_m} d_n^2(f_0,g)$ or $\delta_{n,m}^2< \inf_{g\in \mathcal{F}_m} d_n^2(f_0,g)$. 

Note that if we choose all models $\mathcal{F}_m=\mathcal{F}$, then (\ref{cond:entropy_essen}) reduces to the local entropy condition in \cite{ghosal2000convergence,ghosal2007convergence}. When $\mathcal{F}_m$ is finite-dimensional, typically we can check (\ref{cond:entropy_essen}) for all $g \in \mathcal{F}_m$. Now we comment on (\ref{cond:suplinearity_delta}). The left side of (\ref{cond:suplinearity_delta}) essentially requires super linearity of the map $m \mapsto \delta_{n,m}^2$, while the right side of (\ref{cond:suplinearity_delta}) controls the degree of this super linearity. As a leading example, (\ref{cond:suplinearity_delta}) will be trivially satisfied with $\mathfrak{c}=\gamma=1, \mathfrak{h}_0=\infty$ when $n\delta_{n,m}^2= c\cdot m\log (en)$ for some absolute constant $c>2/c_7$.

Finally we state assumptions on the priors.  

\begin{assumption}[\textbf{Priors: Mass condition}]\label{assump:prior_mass_general}
	For all $m$,
	\begin{enumerate}
		\item[(P1)] (First-step prior) There exists some $\mathfrak{h}\geq 1$ such that 
		\begin{align}\label{cond:suff_prior_mass_general}
		\lambda_{n}(m)\geq e^{-2n\delta_{n,m}^2}/2,\quad
		\sum_{k>\mathfrak{h} m} \lambda_{n}(k) \leq 2e^{- n\delta_{n,m}^2}.
		\end{align}
		\item[(P2)] (Second-step prior)
		\begin{align}\label{cond:sufficient_prior_mass_local_ball_general}
		\Pi_{n,m}\left(\left\{f \in \mathcal{F}_m: d_n^2(f,f_{0,m})\leq \delta_{n,m}^2/c_3  \right\}\right)\geq e^{-2n\delta_{n,m}^2}.
		\end{align}
	\end{enumerate}
\end{assumption}

Condition (P1) can be verified by using the following generic prior $\Lambda_n$:
\begin{align}\label{eqn:prior_generic}
\lambda_{n}(m)\propto \exp(-2 n\delta_{n,m}^2).
\end{align}
\begin{proposition}\label{prop:prior_generic}
	Suppose the first condition of (\ref{cond:suplinearity_delta}) holds. Then (P1) in Assumption \ref{assump:prior_mass_general} holds for the prior (\ref{eqn:prior_generic}) with $\mathfrak{h}_0\geq \mathfrak{h}\geq 2\mathfrak{c}^2$.
\end{proposition}
(\ref{eqn:prior_generic}) will be the model selection (first-step) prior on the model index $\mathcal{I}$ in all examples in Section \ref{section:models}.

Condition (P2) is reminiscent of the classical prior mass condition considered in \cite{ghosal2000convergence,ghosal2007convergence}. Since $\delta_{n,m}^2$ is understood as the `posterior contraction rate' for the model $\mathcal{F}_m$, (P2) can also be viewed as a \emph{solvability condition} imposed on each model. Note that (\ref{cond:sufficient_prior_mass_local_ball_general}) only requires a sufficient prior mass on a Kullback-Leibler ball near $f_{0,m}$, where \cite{ghosal2000convergence,ghosal2007convergence} use more complicated metric balls induced by higher moments of the Kullback-Leibler divergence.

\subsection{Main abstract results}


We say an index set $\mathcal{M}\subset \mathcal{I}$ \emph{rectangular} if and only if there exist some integers $1\leq a_k\leq b_k\leq \infty (k=1,\ldots,q)$ such that $\mathcal{M}= \prod_{k=1}^q\{a_k,\ldots,b_k\}$.

\begin{theorem}\label{thm:general_ms}
	Suppose Assumptions \ref{assump:laplace_cond_kl}-\ref{assump:prior_mass_general} hold for some rectangular $\mathcal{M}\subset \mathcal{I}$ with $\mathfrak{h}\geq C_0 \mathfrak{c}^2$ and $\mathfrak{h}_0\geq C_0', d_0^2\leq \inf_{m \in \mathcal{M}} \epsilon_{n,m}^2/C_0'$, where $\epsilon_{n,m}^2 \equiv \inf_{g \in \mathcal{F}_m}d_n^2(f_0,g)\vee \delta_{n,m}^2$. Suppose $d_n$ satisfies the triangle inequality. Then:
	\begin{enumerate}
		\item For any $m \in \mathcal{M}$,
		\begin{align}\label{ineq:post_rate_general}
		& P_{f_0}^{(n)} \Pi_n\big(f \in \mathcal{F}: d_n^2(f,f_{0})> C_1 \epsilon_{n,m}^2\big\lvert X^{(n)}\big) \leq C_2e^{-n\epsilon_{n,m}^2/C_2}.
		\end{align}
		\item 
		For any $m \in \mathcal{M}$ such that $\delta_{n,m}^2\geq \inf_{g \in \mathcal{F}_m}d_n^2(f_0,g)$\footnote{We use the convention that $\mathcal{F}_m\equiv \mathcal{F}_{m\wedge b}$ where $b=(b_1,\ldots,b_q)$ where $\mathcal{M}= \prod_{k=1}^q\{a_k,\ldots,b_k\}$.},
		\begin{align}\label{ineq:post_model_selection_general}
		P_{f_0}^{(n)} \Pi_n\big(f \notin \mathcal{F}_{C_3m}\lvert X^{(n)}\big) \leq C_2 e^{-n\epsilon_{n,m}^2/C_2}.
		\end{align}
		\item 
		Let $\hat{f}_n\equiv \Pi_n(f| X^{(n)})$ be the posterior mean. If $\mathfrak{h}_0=\infty$ and $d_n(\cdot,\cdot)$ is convex in each of its arguments, then
		\begin{align}\label{ineq:post_mean_general}
		P_{f_0}^{(n)} d_n^2(\hat{f}_n,f_0)\leq C_4 \inf_{m \in \mathcal{M}} \epsilon_{n,m}^2.
		\end{align}
	\end{enumerate}
	Here the constant $C_0$ depends on $\{c_i\}_{i=1}^3,\kappa$ and $C_0', \{C_i\}_{i=1}^4$ depend on the $\{c_i\}_{i=1}^3, \kappa, \mathfrak{c},\mathfrak{h}$ and $\gamma$.
\end{theorem}

\begin{remark}
	Some technical comments:
	\begin{enumerate}
		\item $f_{0,m}$ in Assumptions \ref{assump:local_ent_general} and \ref{assump:prior_mass_general} may be taken other than the minimizer of $f \mapsto d_n^2(f_0,f)$ over $\mathcal{F}_m$. In this case, the conclusion of the above theorems is valid by using $\epsilon_{n,m}^2\equiv d_n^2(f_0,f_{0,m})\vee \delta_{n,m}^2$.
		\item The constants $\{C_i\}_{i=0}^4$ do not depend on $m \in \mathcal{M}$, so the conclusions in (1)-(2) hold simultaneously for all $m\in \mathcal{M}$.
	\end{enumerate}	
\end{remark}

Theorem \ref{thm:general_ms} shows that the task of constructing Bayes procedures \emph{adaptive} to a collection of models in the intrinsic metric of a given statistical experiment, can be essentially reduced to that of designing a suitable \emph{non-adaptive} prior for each model, provided the model selection prior is chosen according to (P1). Furthermore, the resulting posterior mean serves as an automatic adaptive point estimator in a frequentist sense. Besides being rate-adaptive to the collection of models, (\ref{ineq:post_model_selection_general}) shows that the posterior distribution does not spread too much mass on overly large models. Results of this type have been derived primarily in the Gaussian regression model (cf. \cite{castillo2015bayesian,castillo2012needles,gao2015general}) and in density estimation \cite{ghosal2008nonparametric}; here our result shows that this is a general phenomenon for the hierarchical prior design. 

As mentioned in the Introduction, previous results \cite{arbel2013bayesian,deJonge2010adaptive,gao2015general,ghosal2008nonparametric,rousseau2017asymptotic} require certain specific form of the prior, model structure, or the experiments. Our Theorem \ref{thm:general_ms} can thus be viewed as a generalization of these results without such apriori requirements under a hierarchical prior design. As will be clear from concrete applications in Section \ref{section:models}, another advantage of the formulation of Theorem \ref{thm:general_ms} is that Assumptions \ref{assump:local_ent_general}-\ref{assump:prior_mass_general} typically concern \emph{finite-dimensional} models $\mathcal{F}_m$ so verification is easy and routine. 

Note that $f_0$ is arbitrary and hence our oracle inequalities (\ref{ineq:post_rate_general}) and (\ref{ineq:post_mean_general}) account for model mis-specification errors. Previous work allowing model mis-specification includes \cite{gao2015general} who mainly focuses on structured linear models in the Gaussian regression setting, and \cite{kleijn2006misspecification} who pursued generality at the cost of harder-to-check conditions.

\begin{remark}\label{rmk:main_thm}
	We make some technical remarks.
	\begin{enumerate}
		\item The probability estimate in (\ref{ineq:post_rate_general}) is of Gaussian type and is therefore sharp (up to constants) in view of the lower bound result Theorem 2.1 in \cite{hoffmann2015adaptive}. Such sharp estimates have been derived separately in the Hellinger metric \cite{ghosal2000convergence}, or in individual settings, e.g. the sparse normal mean model \cite{castillo2012needles}, the sparse PCA model \cite{gao2015ratepca}, and the structured linear model \cite{gao2015general}, to name a few. The Gaussian estimate naturally implies good behavior of the posterior mean under bounded metrics (cf. page 507 of \cite{ghosal2000convergence}). In the leading case $\mathfrak{c}=\gamma=1, \mathfrak{h}_0=\infty$ in Assumption \ref{assump:local_ent_general}, the posterior mean $\hat{f}_n$ satisfies an oracle inequality with a Gaussian tail\footnote{This can be seen by a simple modification of the proof by calculating the moment generating function.}.

		\item (\ref{ineq:post_model_selection_general}) asserts that the posterior distribution does not concentrate on overly large models. It is also of significant interest to assert the converse in some models, i.e. the posterior distribution does not concentrate on overly small models under additional problem-specific conditions. We refer to the readers to \cite{belitser2017coverage,castillo2015bayesian,rousseau2016asymptotic,yang2017bayesian} and references therein for more details in this direction.
		
		\item Assumption \ref{assump:laplace_cond_kl} implies, among other things,  the existence of a good test (cf. Lemma \ref{lem:local_test_general}). In this sense our approach here falls into the general testing approach adopted in \cite{ghosal2000convergence,ghosal2007convergence}. 
		Some alternative approaches for dealing with non-intrinsic metrics can be found in \cite{castillo2014bayesian,hoffmann2015adaptive,yoo2016supremum}.
		\item The constants $\{C_i\}_{i=1}^4$ in Theorem \ref{thm:general_ms} depend at most polynomially with respect to the constants involved in Assumption \ref{assump:laplace_cond_kl}. This will be useful in handling models where the local Gaussianity only holds locally on the parameter space (cf. Appendix \ref{section:more_examples}).
		
		\item If $d_n$ does not satisfy the triangle inequality, then (\ref{ineq:post_rate_general}) and (\ref{ineq:post_model_selection_general}) in Theorem \ref{thm:general_ms} hold if $f_0 \in \mathcal{F}_m$ for some $m$ (i.e. the form of an exact oracle inequality may be lost at a general level). 
	\end{enumerate}
\end{remark}

\begin{remark}\label{rmk:comparison_yang}
	We compare our results with Theorems 4 and 5 of \cite{yang2017bayesian}. Both their results and our Theorem \ref{thm:general_ms} shed light on the general problem of Bayes model selection, while differing in several important aspects:
	
	\begin{enumerate}
		\item Theorem 4 of \cite{yang2017bayesian} targets at exact model selection consistency, under a set of additional `separation' assumptions. Our Theorem \ref{thm:general_ms} (2) requires no extra assumptions, and shows that the posterior distribution does not concentrate on overly large models. This is significant in non-parametric problems: the true signal typically need not belong to any specific model.
		\item Theorem 5 of \cite{yang2017bayesian} contains a term involving the cardinality of the models, so their bound will be finite only if there are finitely many models. It remains open to see if this can be removed.
	\end{enumerate}
\end{remark}

\subsection{The localization (sieving) principle}\label{section:localization}

Consider a sequence of models $\{\bar{\mathcal{F}}_n\}$, where $\bar{\mathcal{F}}_n$ is regarded as the localized model of $\mathcal{F}$ at sample size $n$. Note that any prior $\Pi_n$ on $\mathcal{F}$ can be localized to a prior $\bar{\Pi}_n$ on $\bar{\mathcal{F}}_n$: for any $B \subset \bar{\mathcal{F}}_n$, define $\bar{\Pi}_n(B)\equiv \Pi_n(B\cap \bar{\mathcal{F}}_n)/\Pi_n(\bar{\mathcal{F}}_n)$. Now the quantity in Theorem \ref{thm:general_ms} concerning posterior distribution can be decomposed by
\begin{align}\label{ineq:localize_bernstein}
& P_{f_0}^{(n)} \Pi_n\big(f \in \mathcal{F}: d_n^2(f,f_{0})> C_1 \epsilon_{n,m}^2\big\lvert X^{(n)}\big) \\
&\leq P_{f_0}^{(n)} \bar{\Pi}_n\big(f \in \bar{\mathcal{F}}_n: d_n^2(f,f_{0})> C_1 \epsilon_{n,m}^2\big\lvert X^{(n)}\big)+ P_{f_0}^{(n)} \Pi_n\big(f \notin \bar{\mathcal{F}}_n\big\lvert X^{(n)}\big).\nonumber
\end{align}
In essence, (\ref{ineq:localize_bernstein}) suggests that we can use the machinery of Assumptions \ref{assump:laplace_cond_kl}-\ref{assump:prior_mass_general} to the localized model $\mathcal{F}_n$ (typically by choosing the constants $c_2,c_3,d_0$ depending on $n$), as long as the residue term $P_{f_0}^{(n)} \Pi_n\big(f \notin \bar{\mathcal{F}}_n\big\lvert X^{(n)}\big)$ is well-controlled. This typically reduces to a reasonable control of $\Pi_n(\mathcal{F}\setminus \bar{\mathcal{F}}_n)$ (cf. Lemma 1 of \cite{ghosal2007convergence}, see also examples in Appendix \ref{section:more_examples}). The localization principle is under the name `sieving' in \cite{ghosal2000convergence,ghosal2007convergence}.

\subsection{Proof sketch}\label{section:proof_sketch}

Here we sketch the main steps in the proof of our main abstract result Theorem \ref{thm:general_ms}. The details will be deferred to Appendix \ref{section:proof_main_result}. The proof can be roughly divided into two main steps.

\noindent (\textbf{Step 1}) We first solve a localized problem on the model $\mathcal{F}_m$ by `projecting' the underlying probability measure from $P_{f_0}$ to $P_{f_{0,m}}$. In particular, we establish exponential deviation inequality for the posterior contraction rate via the existence of tests guaranteed by Lemma \ref{lem:local_test_general}:
\begin{align}\label{ineq:proof_sketch_1}
P_{f_{0,m}}^{(n)}\Pi_n\big(f \in \mathcal{F}: d_n^2(f,f_{0,m}))>M\delta_{n,\tilde{m}}^2|X^{(n)}\big)\lesssim e^{-c_1 n\delta_{n,\tilde{m}}^2},
\end{align}
where $\tilde{m}$ is the smallest index $\geq m$ such that $\delta_{n,\tilde{m}}^2\gtrsim d_n^2(f_0,f_{0,m})$. This index may deviate from $m$ substantially for small indices.

\noindent (\textbf{Step 2}) We argue that, the cost of the projection in Step 1 is essentially a multiplicative $\mathcal{O}\big(\exp(c_2 n\delta_{n,\tilde{m}}^2)\big)$ factor in the probability bound (\ref{ineq:proof_sketch_1}), cf. Lemma \ref{lem:change_variable}, which is made possible by the local Gaussianity Assumption \ref{assump:laplace_cond_kl}. Then by choosing $c_1$ much larger than $c_2$ we obtain the conclusion by the definition of $\delta_{n,\tilde{m}}^2$ and the fact that $\delta_{n,\tilde{m}}^2\approx d_n^2(f_0,f_{0,m})\vee \delta_{n,m}^2$.

The existence of tests (Lemma \ref{lem:local_test_general}) is used in Step 1. Step 2 is inspired by the work of \cite{chatterjee2015risk} in the context of frequentist least squares estimator over a polyhedral cone in the Gaussian regression setting, where the localized problem therein is estimation of signals on a low-dimensional face (where `risk adaptation' happens). In the Bayesian context, \cite{castillo2015bayesian,castillo2012needles} used a change of measure argument in the Gaussian regression setting for a different purpose. Our proof strategy can be viewed as an extension of these ideas beyond the (simple) Gaussian regression model.

\section{Models and applications}\label{section:models}

In this section we work out a couple of specific statistical models that satisfy the local Gaussianity Assumption \ref{assump:laplace_cond_kl} to illustrate the scope of the general results in Section \ref{section:gen_frame}. Some of the examples come from \cite{ghosal2007convergence}; we identify the `intrinsic' metric to use in these models.  Some concrete applications are also given. The applications presented in this section serve as a demonstration of the scope of our general results in deriving new contraction rate results. More applications can be found in Appendix \ref{section:more_examples} to illustrate the localization principle (cf. Section \ref{section:localization}) and aid calculations/formulation in complicated list of models.


\subsection{Regression models}\label{section:regression}

Suppose we want to estimate $\theta=(\theta_{1},\ldots,\theta_{n})$ in a given model $\Theta_n \subset \R^n$ in the following settings: for $1\leq i\leq n$,
\begin{enumerate}
	\item (\emph{Gaussian}) $X_i = \theta_{i} + \epsilon_i$ where $\epsilon_i$'s are i.i.d. $\mathcal{N}(0,1)$ and $\Theta_n \subset\R^n$;
	\item (\emph{Laplace}) $X_i = \theta_{i} + \epsilon_i$ where $\epsilon_i$'s are i.i.d. errors with density $x\mapsto \frac{1}{2} e^{-\abs{x}}$, and $\Theta_n \subset[-M,M]^n$; 
	\item (\emph{Binary}) $X_i \sim_{} \mathrm{Bern}(\theta_{i})$ are independent, where $\Theta_n \subset [\eta,1-\eta]^n$ for some $\eta>0$;
	\item (\emph{Poisson}) $X_i\sim_{\mathrm{i.i.d.}} \mathrm{Poisson}(\theta_{i})$ where $\Theta_n \subset [1/M,M]^n$ for some $M\geq 1$;
\end{enumerate}
For any $\theta_0,\theta_1 \in \Theta_n$, $
\ell_n^2(\theta_0,\theta_1)\equiv n^{-1}\sum_{i=1}^n\big(\theta_{0,i}-\theta_{1,i}\big)^2$.

\begin{lemma}\label{lem:bernstein_regression} 
	Assumption \ref{assump:laplace_cond_kl} holds for $\ell_n$ with
	\begin{enumerate}
		\item (Gaussian) $c_1=c_2=c_3=\kappa_g=1$ and $\kappa_\Gamma=0$;
		\item (Laplace) $\kappa_\Gamma=0$, $\kappa_g$ an absolute constant and constants $\{c_i\}_{i=1}^3$ depending on $M$ only;
		\item (Binary) $\kappa_\Gamma=0$ and the constants $\{c_i\}_{i=1}^3,\kappa_g$ depend on $\eta$ only; 
		\item (Poisson) constants $\{c_i\}_{i=1}^3, \kappa$ depending on $M$ only.
	\end{enumerate} 
\end{lemma}

\begin{corollary}\label{cor:regression_model}
	For Gaussian/Laplace/binary/Poisson regression models, let $d_n\equiv \ell_n$. If Assumptions \ref{assump:local_ent_general}-\ref{assump:prior_mass_general} hold, then (\ref{ineq:post_rate_general})-(\ref{ineq:post_mean_general}) hold.
\end{corollary}

Using similar techniques we can derive analogous results for Gaussian regression with random design and white noise model. We omit the details.

\begin{remark}
	The boundedness assumption in Laplace/binary/Poisson models is imposed here for simplicity, and can be removed using the localization principle (cf. Section \ref{section:localization}) for more concrete $\Theta_n$'s and priors. See Appendix \ref{section:more_examples} for an example.
\end{remark}

Below we give three concrete applications in the Gaussian regression model $y_i = f_0(x_i)+\epsilon_i (1\leq i\leq n)$, where $\epsilon_i$'s are i.i.d. $\mathcal{N}(0,1)$. We slightly abuse $\ell_n$ to denote $\ell_n^2(f,g)\equiv n^{-1}\sum_{i=1}^n (f(x_i)-g(x_i))^2$. 

\subsubsection{Example: Trace regression}\label{section:trace_reg}
Consider fitting the Gaussian regression model $
y_i = f_0(x_i) + \epsilon_i (1\leq i\leq n)$ by $\mathcal{F}\equiv\{f_A: A \in \R^{m_1\times m_2}\}$ where $f_A(x)=\trace(x^\top A)$ for all $x \in \mathfrak{X}\equiv \R^{m_1\times m_2}$. Let $\underline{m}\equiv m_1\wedge m_2$ and $\bar{m}\equiv m_1\vee m_2$. The index set is $\mathcal{I}= \mathcal{I}_1\cup\mathcal{I}_2\equiv \{1,\ldots, r_{\max}\}\cup\{r_{\max}+1,\ldots\} = \mathbb{N}$ where $r_{\max}\leq \underline{m}$. For $r \in \mathcal{I}_1$, let $\mathcal{F}_r\equiv\{f_A: A \in \R^{m_1\times m_2}, \rank(A)\leq r\}$, and for $r \in \mathcal{I}_2$, $\mathcal{F}_r\equiv \mathcal{F}_{r_{\max}}$\footnote{This trick of defining models for high-dimensional experiments will also used in other applications in later subsections, but we will not explicitly state it again.}.

Although various Bayesian methods have been proposed in the literature (cf. see \cite{alquier2014bayesian} for a state-to-art summary), theoretical understanding has been limited. \cite{alquier2015bayesian} derived an oracle inequality for an exponentially aggragated estimator for the matrix completion problem. Their result is purely frequentist. Below we consider a two step prior similar to \cite{alquier2014bayesian,alquier2015bayesian}, and derive the corresponding posterior contraction rates.

For a matrix $B=(b_{ij}) \in \R^{m_1\times m_2}$ let $\pnorm{B}{p}$ denote its Schatten $p$-norm\footnote{That is, $\pnorm{B}{p}\equiv \left(\sum_{j=1}^{\underline{m}} \sigma_j(B)^p\right)^{1/p}$, where $\{\sigma_j(B)\}$ are the singular values of $B$.}. $p=1$ and $2$ correspond to the \emph{nuclear norm} and  the \emph{Frobenius norm} respectively. To introduce the notion of RIP, let $\mathcal{X}:\R^{m_1\times m_2}\to \R^n$ be the linear map defined via $A \mapsto (\trace(x_i^\top A))_{i=1}^n$.

\begin{definition}
	The linear map $\mathcal{X}:\R^{m_1\times m_2}\to \R^n$ is said to satisfy \emph{RIP}$(r,\bm{\nu}_r)$ for some $1\leq r \leq r_{\max}$ and some $\bm{\nu}_r=(\underline{\nu}_r,\bar{\nu}_r)$ with $0<\underline{\nu}_r\leq \bar{\nu}_r<\infty$ iff $
	\underline{\nu}_r\leq  \frac{\pnorm{\mathcal{X}(A)}{2}}{\sqrt{n}\pnorm{A}{2}}\leq \bar{\nu}_r$
	holds for all matrices $A \in \R^{m_1\times m_2}$ such that $\rank(A)\leq r$. For $r>r_{\max}$, $\mathcal{X}$ satisfies RIP$(r,\bm{\nu}_r)$ iff $\mathcal{X}$ satisfies RIP$(r_{\max},\bm{\nu}_r)$. Furthermore,  $\mathcal{X}:\R^{m_1\times m_2}\to \R^n$ is said to satisfy \emph{uniform RIP} $(\bm{\nu};\mathcal{I})$ on an index set $\mathcal{I}$ iff $\mathcal{X}$ satisfies RIP$(2r,\bm{\nu})$ for all $r \in \mathcal{I}$.
\end{definition}
RIP$(r,\bm{\nu}_r)$ is a variant of the RIP condition introduced in \cite{candes2005decoding,candes2011tight,recht2010guaranteed} with scaling factors $\bar{\nu}_r = 1/(1-\delta_r)$ and $\underline{\nu}_r=1/(1+\delta_r)$ for some $0<\delta_r<1$. This condition quantifies the degree in which the linear map $\mathcal{X}$ behaves like an isometry between $\R^{m_1\times m_2}$ and $\R^n$ in terms of the $\ell_2$ metric. Below are two canonical examples.

\begin{example}[Matrix completion]
	Suppose that $
	x_i\in \R^{m_1\times m_2}$ takes value $1$ at one position and $0$ otherwise.
	Further assume that $\underline{A}\leq \abs{A_0}_{ij}\leq \bar{A}$ for all $1\leq i\leq m_1$ and $1\leq j\leq m_2$\footnote{This assumption is usually satisfied in applications: in fact in the Netflix problem (which is the main motivating example for matrix completion), $A_0$ is the rating matrix with rows indexing the users and columns indexing movies, and we can simply take $\underline{A}=1$ (one star) and $\bar{A}=5$ (five stars).}. Let $\Omega\equiv \Omega_{\mathcal{X}}$ denote the indices for which $\{x_i\}$'s take value $1$. Then $\pnorm{\mathcal{X}(A)}{2}=\pnorm{A\bm{1}_\Omega}{2}$. Easy calculations show that we can take $\bm{\nu}=(\bar{\nu},\underline{\nu})$ defined by $
	\bar{\nu}={(\bar{A}\sqrt{m_1m_2\wedge n})}/{(\underline{A}\sqrt{nm_1m_2})}, \underline{\nu}={ (\underline{A}\sqrt{m_1m_2\wedge n})}/{(\bar{A}\sqrt{nm_1m_2})}$
	so that $\mathcal{X}$ is uniform RIP$(\bm{\nu};\mathcal{I})$.
\end{example}

\begin{example}[Gaussian measurement ensembles]
	Suppose $x_i$'s are i.i.d. random matrices whose entries are i.i.d. standard normal. 
	Theorem 2.3 of \cite{candes2011tight} entails that $\mathcal{X}$ is uniform RIP$(\bm{\nu};\mathcal{I})$ with $\bar{\nu}=1+\delta, \underline{\nu}=1-\delta$ for some $\delta \in (0,1)$, with probability at least $1-C\exp(-cn)$\footnote{Note here we used the union bound to get a probability estimate $r_{\max}\exp(-cn)\lesssim \exp(-c'n)$ for some $c'<c$ under the assumption that $n\gtrsim \bar{m}r_{\max}$.}, provided $n\gtrsim \bar{m}r_{\max}$. 
\end{example}

Consider a prior $\Lambda_{n}$ on $\mathcal{I}$ of form
\begin{align}\label{eqn:prior_tracereg}
\lambda_{n}(r)\propto \exp\big(-c \cdot (m_1+m_2) r \log \bar{m}\big),
\end{align}
where $c>0$ is a constant to be specified later. Given the chosen index $r \in \mathcal{I}_1$, a prior on $\mathcal{F}_r$ is induced by a prior on all $m_1\times m_2$ matrices of form $\sum_{i=1}^r u_i v_i^\top$ where $u_i \in \R^{m_1}$ and $v_i \in \R^{m_2}$. Here we use a product prior distribution $G$ with Lebesgue density $(g_1\otimes g_2)^{\otimes r}$ on $(\R^{m_1}\times \R^{m_2})^r$. For simplicity we use $g_i\equiv g^{\otimes m_i}$ for $i=1,2$ where $g$ is symmetric about $0$ and non-increasing on $(0,\infty)$\footnote{We will always use such $g$ in the prior design in the examples in this section.}. Let $\tau_{r,g}^{\trace}\equiv \sup\limits_{A_{0,r}\in \arg\min_{B:\rank(B)\leq r}\ell_n^2(f_B,f_{0})} g\big(\sigma_{\max}(A_{0,r})+ 1\big)$ where $\sigma_{\max}$ denotes the largest singular value.

\begin{theorem}\label{thm:rate_tracereg}
	Fix $0<\eta<1/2$ and $r_{\max}\leq n$. Suppose that there exists some $\mathcal{M}\subset \mathcal{I}_1$ such that the linear map $\mathcal{X}:\R^{m_1\times m_2}\to \R^n$ satisfies uniform RIP$(\bm{\nu};\mathcal{M})$, and that for all $r \in \mathcal{M}$, we have
	\begin{align}\label{cond:rate_tracereg}
	\tau_{r,g}^{\trace} \geq e^{-\log \bar{m}/(2\eta)}, \quad \bar{m}\geq 3 \vee \big(2\bar{\nu} (1\vee \sigma_{\max}(A_{0,r})) n^2\big)^{2\eta}.
	\end{align}
	Then there exists some $c>0$ in (\ref{eqn:prior_tracereg}) depending on $\bar{\nu}/\underline{\nu},\eta$ such that for any $r \in \mathcal{M}$, 
	\begin{align}\label{ineq:post_rate_tracereg}
	& P_{f_0}^{(n)} \Pi_n\big(A \in \R^{m_1\times m_2}:  \ell_n^2(f_A,f_{0})>  C_1 (\epsilon_{n,r}^{\mathrm{tr}})^2\big\lvert Y^{(n)}\big) \leq C_2e^{- n(\epsilon_{n,r}^{\mathrm{tr}})^2/ C_2 }.
	\end{align}
	Here $
	(\epsilon_{n,r}^{\mathrm{tr}})^2\equiv \max\{\inf_{B: \rank(B)\leq r}\ell_n^2(f_0,f_B),  { (m_1+m_2)r \log \bar{m}}/{n}\}$, and the constants $C_i (i=1,2)$ depend on $\bar{\nu}/\underline{\nu},\eta$.
\end{theorem}

By Theorem 5 of \cite{rohde2011estimation}, the rate in (\ref{ineq:post_rate_tracereg}) is minimax optimal up to a logarithmic factor. 
To the best knowledge of the author, the theorem above is the first result in the literature that addresses the posterior contraction rate  in the context of trace regression in a fully Bayesian setup.  

(\ref{cond:rate_tracereg})  may be verified in a case-by-case manner; or generically we can take $\mathcal{M}=\{r_0,r_0+1,\ldots\}$ if the model is well specified, at the cost of sacrificing the form of oracle inequalities (but still get nearly optimal posterior contraction rates) in (\ref{ineq:post_rate_tracereg}). In particular, the first condition of (\ref{cond:rate_tracereg}) prevents the largest eigenvalue of $A_{0,r}$ from growing too fast. This is in similar spirit with Theorem 2.8 of \cite{castillo2012needles}, showing that the magnitude of the signals cannot be too large for light-tailed priors to work in the sparse normal mean model. The second condition of (\ref{cond:rate_tracereg}) is typically a mild technical condition: we only need to choose $\eta>0$ small enough.

\subsubsection{Example: Isotonic regression}\label{section:iso_reg}
Consider the isotonic regression model $Y_i=f_0(x_i)+\epsilon_i$ by $\mathcal{F}\equiv\{f:[0,1]\to \R: f \textrm{ is non-decreasing}\}$. For simplicity the design points are assumed to be $x_i=i/(n+1)$ for all $1\leq i\leq n$.  Bayesian approaches for the  isotonic regression model received considerable attention, cf. \cite{holmes2003generalized,shively2009bayesian,neelon2004bayesian,lin2014bayesian,salomond2014adaptive}. Let $
\mathcal{F}_m\equiv\big\{f \in \mathcal{F}, f\textrm{ is piecewise constant with at most } m \textrm{ constant pieces}\big\}.$
Consider the following  prior $\Lambda_n$ on $\mathcal{I}=\N$:
\begin{align}\label{eqn:prior_iso}
\lambda_{n}(m)\propto \exp\big(-c\cdot m \log(e n)\big),
\end{align}
where $c>0$ is a constant to be specified later. Let $g_m\equiv g^{\otimes m}$ where $g$ is symmetric and non-increasing on $(0,\infty)$. Then $\bar{g}_m(\bm{\mu})\equiv m! g_m \bm{1}_{\{\mu_1\leq\ldots\leq \mu_m\}}(\bm{\mu})$ is a valid density on $\{\mu_1\leq\ldots\leq \mu_m\}$. Given a chosen model $\mathcal{F}_m$ by the prior $\Lambda_n$, we randomly pick a set of change points $\{x_{i(k)}\}_{k=1}^m ( i(1)<\ldots <i(m) )$ and put a prior $\bar{g}_m$ on $\{f(x_{i(k)})\}$'s. \cite{holmes2003generalized} proposed a similar prior with $\Lambda_n$ being uniform since they assumed the maximum number of change points is known \emph{apriori}. Below we derive a theoretical result without assuming the knowledge of this. Let $\tau_{m,g}^{\textrm{iso}}=\sup\limits_{f_{0,m} \in \arg\min_{g \in \mathcal{F}_m} \ell_n^2(f_0,g)}g\big(\pnorm{f_{0,m}}{\infty}+1\big)$\footnote{The value of $f_{0,m}$ outside of $[1/(n+1),n/(n+1)]$ can be defined in a canonical way by extending $f_{0,m}(1/(n+1))$ and $f_{0,m}(n/(n+1))$ towards the endpoints.}.

\begin{theorem}\label{thm:rate_iso}
	Fix $0<\eta<1/2$. Suppose that
	\begin{align}\label{cond:iso}
	\tau_{m,g}^{\mathrm{iso}}\geq e^{- \log (en)/(2\eta) }.
	\end{align}
	Then there exists some $c>0$ in (\ref{eqn:prior_iso}) depending on $\eta$ such that
	\begin{align}\label{ineq:post_rate_iso}
	& P_{f_0}^{(n)} \Pi_n\big(f \in \mathcal{F}: \ell_n^2(f,f_{0})> C_1(\epsilon_{n,m}^{\mathrm{iso}})^2\big\lvert Y^{(n)}\big)  \leq C_2 e^{-n (\epsilon_{n,m}^{\mathrm{iso}})^2/C_2 }.
	\end{align}
	Here $
	(\epsilon_{n,m}^{\mathrm{iso}})^2\equiv \max\{\inf_{g \in \mathcal{F}_m}\ell_n^2(f_0,g),{ m\log (en)}/{n}\}$, and the constants $C_i(i=1,2)$ depend on $\eta$.
\end{theorem}

(\ref{ineq:post_rate_iso}) implies that if $f_0$ is piecewise constant, the posterior distribution contracts at nearly a parametric rate. For general isotonic signals $f_0 \in \mathcal{F}$ with $\pnorm{f_0}{\infty}<\infty$, by using Theorem 4.1 of \cite{chatterjee2015risk}, we obtain a contraction rate on the order of $n^{-2/3}\log (en)$ in $\ell_n^2$. (\ref{cond:iso}) can be checked by the following.

\begin{lemma}\label{lem:verify_cond_iso}
	If $f_0$ is square integrable, and the prior density $g$ is heavy-tailed in the sense that there exists some $\alpha>0$ such that $\liminf_{\abs{x}\to \infty} x^\alpha g(x)>0$. 
	Then for any $\eta \in (0,1/\alpha)$, (\ref{cond:iso}) holds uniformly in all $m \in \N$ for $n$ large enough depending on $\alpha$ and $\pnorm{f_0}{L_2([0,1])}$.
\end{lemma}

\subsubsection{Example: Convex regression}
Consider fitting the Gaussian regression model $Y_i=f_0(x_i)+\epsilon_i $ by $\mathcal{F}$, the class of convex functions on $\mathfrak{X}=[0,1]^d$. Let $
\mathcal{F}_m\equiv \big\{f(x)=\max_{1\leq i\leq m} (a_i\cdot x+b_i): a_i \in \R^d, b_i \in \R\big\}$
denote the class of piecewise affine convex functions with at most $m$ pieces.

We will focus on the multivariate case since the univariate case can be easily derived using the techniques exploited in isotonic regression. A prior on each model  $\mathcal{F}_m$ can be induced by a prior on the slopes and the intercepts $\{(a_i,b_i) \in \R^d \times \R\}_{i=1}^m$. We use a prior with density $\bigotimes_{i=1}^m g^{\otimes d}\otimes g$ on $(\R^d\times \R)^m$ to induce a prior on $\mathcal{F}_m$. For any $f_{0,m} \in \arg\min_{g \in \mathcal{F}_m} \ell_n^2(f_0,g)$, it can be represented as $f_{0,m}(x)\equiv \max_{1\leq i\leq m}\big(a_i^{(m)}\cdot x + b_i^{(m)}\big)$. Let $\tau_{m,g}^{\mathrm{cvx}}\equiv \sup\limits_{f_{0,m} \in \arg\min_{g \in \mathcal{F}_m} \ell_n^2(f_0,g)} \min_{1\leq i\leq m} \big\{g\big(\pnorm{a_i^{(m)}}{\infty}+ 1\big), g\big(\abs{b_i^{(m)}}+ 1\big)\big\}$.

The prior $\Lambda_n$ we will use on the index $\mathcal{I}=\N$ is given by
\begin{align}\label{eqn:prior_cvxreg}
\lambda_{n}(m)\propto \exp\big(-c\cdot d m\log 3m\cdot \log n\big),
\end{align}
where $c>0$ is a constant to be specified later. The first step prior used in \cite{hannah2011bayesian} is a Poisson proposal, which slightly differs from (\ref{eqn:prior_cvxreg}) by a logarithmic factor. This would affect the contraction rate only by a logarithmic factor. 

\begin{theorem}\label{thm:rate_cvxreg}
	Fix $0<\eta<1/4$. Suppose that
	\begin{align}\label{cond:cvxreg_1}
	\tau_{m,g}^{\mathrm{cvx}}\geq e^{- \log n\cdot \log 3m/8\eta },
	\end{align}
	and $n\geq d$. Then there exists some $c>0$ in (\ref{eqn:prior_cvxreg}) depending on $\eta$ such that
	\begin{align}\label{ineq:post_rate_cvxreg}
	& P_{f_0}^{(n)} \Pi_n\big(f \in \mathcal{F}: \ell_n^2(f,f_{0})>  C_1	(\epsilon_{n,m}^{\mathrm{cvx}})^2\big\lvert Y^{(n)}\big) \leq C_2 e^{- n (\epsilon_{n,m}^{\mathrm{cvx}})^2/C_2}.
	\end{align}
	Here 
	$
	(\epsilon_{n,m}^{\mathrm{cvx}})^2\equiv \max\{\inf_{g \in \mathcal{F}_m}\ell_n^2(f_0,g),{d\log n\cdot m\log 3m}/{n}\}$, and the constants $C_i(i=1,2)$ depend on $\eta$.
\end{theorem}

The above oracle inequality shows that the posterior contraction rate of \cite{hannah2011bayesian} (Theorem 3.3 therein) is far from optimal. (\ref{cond:cvxreg_1}) can be satisfied by using heavy-tailed priors $g(\cdot)$ in the same spirit as Lemma \ref{lem:verify_cond_iso}---if $f_0$ is square integrable and the design points are regular enough (e.g. using regular grids on $[0,1]^d$). Explicit rates can be obtained using approximation techniques in \cite{han2016multivariate}. Using the same proof as Lemma 4.10 therein, if $f_0$ is Lipschitz, the contraction rate in $\ell_2^2$ becomes the familiar one in the sense that $
\inf_{m \in \N}(\epsilon_{n,m}^{\mathrm{cvx}})^2\lesssim \inf_{m \in \N}\max\{m^{-4/d},{\log n\cdot m\log 3m}/{n}\}\asymp (\log^2 n/n)^{4/(d+4)}$. 

\begin{remark}
	For univariate convex regression, the term $\log(3m)$ in (\ref{eqn:prior_cvxreg})-(\ref{ineq:post_rate_cvxreg}) can be removed. The logarithmic term is due to the fact that the \emph{pseudo-dimension} of $\mathcal{F}_m$ scales as $m\log (3m)$ for $d\geq 2$, cf. Lemma \ref{lem:pdim_cvxfcn}. 
\end{remark}

\begin{remark}\label{rmk:support_fcn}
	Using similar priors and proof techniques we can construct a (nearly) rate-optimal adaptive Bayes estimator for the support function regression problem for convex bodies \cite{guntuboyina2012optimal}. There the models $\mathcal{F}_m$ are support functions indexed by polytopes with $m$ vertices, and a prior on $\mathcal{F}_m$ is induced by a prior on the location of the $m$ vertices. The pseudo-dimension of $\mathcal{F}_m$ can be controlled using techniques developed in \cite{guntuboyina2012optimal}. Details are omitted.
\end{remark}

\subsubsection{Example: High-dimensional partially linear model}
Consider fitting the Gaussian regression model $Y_i=f_0(x_i,z_i)+\epsilon_i$ where $(x_i,z_i) \in \R^p\times [0,1]$,  by a partially linear model $\mathcal{F}\equiv \{f_{\beta,u}(x,z)=x^\top \beta+u(z)\equiv h_\beta(x)+u(z): \beta \in \R^p, u \in \mathcal{U}\}$ where the dimension of the parametric part can diverge. We consider $\mathcal{U}$ to be the class of non-decreasing functions as an illustration (cf. Section \ref{section:iso_reg}). Consider models $\mathcal{F}_{(s,m)}\equiv \{f_{\beta,u}: \beta \in B_0(s), u \in \mathcal{U}_m\}$ where $\mathcal{U}_m$ denotes the class of piecewise constant non-decreasing functions with at most $m$ constant pieces, and $B_0(s)\equiv \{v \in \R^p: \abs{\mathrm{supp}(v)}\leq s\}$.  In this example the model index $\mathcal{I}$ is a 2-dimensional lattice. Our goal here is to construct an estimator that satisfies an oracle inequality over the models $\{\mathcal{F}_{(s,m)}\}_{(s,m) \in \{1,\ldots,p\}\times \{1,\ldots,n\}}$. Consider the following model selection prior:
\begin{align}\label{eqn:prior_hp}
\lambda_{n}((s,m))\propto \exp\big(-c\cdot (s\log(ep)\wedge \rank(X)+m\log(en))\big),
\end{align}
where $c>0$ is a constant to be specified later. Here $X \in \R^{n\times p}$ is the design matrix so that $X^\top X/n$ is normalized with diagonal elements taking value $1$\footnote{This is a common assumption, cf. Section 6.1 of \cite{buhlmann2011statistics}.}. For a chosen model $\mathcal{F}_{(s,m)}$, consider the following prior $\Pi_{n,(s,m)}$: pick randomly a support $S \subset \{1,\ldots,p\}$ with $\abs{S}=s$ and a set of change points $Q\equiv \{z_{i(k)}\}_{k=1}^m(i(1)<\ldots i(m))$, and then put a prior $g_{S,Q}$ on $\beta_S$ and $u(z_{i(k)})$'s. For simplicity we use a  product prior $g_{S,Q}\equiv g^{\otimes s}\otimes \bar{g}_m$ where $\bar{g}_m$ is a prior on $\{\mu_1\leq \ldots \leq \mu_m\}\subset \R^m$ constructed in Section \ref{section:iso_reg}. For any $f_{0,(s,m)} \in \inf_{g \in \mathcal{F}_{(s,m)}} \ell_n^2(f_0,g)$, write $f_{0,(s,m)}(x,z)=x^\top \beta_{0,s}+u_{0,m}(z)\equiv h_{0,s}(x)+u_{0,m}(z)$. Let 
$\tau_{m,g}\equiv\sup\limits_{f_{0,(s,m)} \in \inf_{g \in \mathcal{F}_{(s,m)}} \ell_n^2(f_0,g) } g(\pnorm{u_{0,m}}{\infty}+1)$. 

\begin{theorem}\label{thm:rate_hp}
	Fix $0<\eta<1/4$. Suppose $p\geq n$ and $L>\big(\log (ep)/\rank(X)\big)\vee \inf\limits_{f_{0,(s,m)} \in \inf_{g \in \mathcal{F}_{(s,m)}} \ell_n^2(f_0,g) }\pnorm{\beta_{0,s}}{\infty}\vee \sigma_{\max}(X)$ for some $L>0$. Suppose that
	\begin{align}\label{cond:hpreg_1}
	g(L+1)>e^{-\log (ep)/2\eta},\quad \tau_{m,g}\geq e^{-\log(en)/(2\eta) }.
	\end{align}
	Then there exists some $c>0$ in (\ref{eqn:prior_hp}) depending on $\eta, L$ such that
	\begin{align}\label{ineq:post_rate_hp}
	& P_{f_0}^{(n)} \Pi_n\big(f \in \mathcal{F}: \ell_n^2(f,f_{0})> C_1 (\epsilon_{n,(s,m)}^{\mathrm{hp}})^2\big\lvert Y^{(n)}\big) \leq C_2 e^{-n (\epsilon_{n,(s,m)}^{\mathrm{hp}})^2/C_2}.
	\end{align}
	Here $
	(\epsilon_{n,(s,m)}^{\mathrm{hp}})^2\equiv \max\{\inf_{f_{\beta,u} \in \mathcal{F}_{(s,m)}} \ell_n^2(f_0,f_{\beta,u}),( s\log(ep)\wedge \rank(X)+m\log (en))/{n}\}$, and the constants $C_i(i=1,2)$ depend on $\eta, L$.
\end{theorem}

The condition $p\geq n$ can be replaced by $p\geq n^\delta$ for any $\delta>0$ by changing the constants. $L>0$ prevents $p$, $\pnorm{\beta_{0,s}}{\infty}$ and the maximal singular value of $X$ from being too large. 
The second condition of (\ref{cond:hpreg_1}) is the same as in (\ref{cond:iso}) (so in particular can be checked using Lemma \ref{lem:verify_cond_iso}). When the model is well-specified in the sense that $f_0(x,z)=x^\top \beta_0+u_0(z)$ for some $\beta_0 \in B_0(s_0)$ and $u_0 \in \mathcal{U}$, the oracle rate in (\ref{ineq:post_rate_hp}) becomes
\begin{align}\label{rate:hp}
\frac{s_0\log(ep)\wedge \rank(X)}{n}+ \inf_{m \in \N}\bigg(\inf_{u \in \mathcal{U}_m} \ell_n^2(u_0,u)+\frac{m \log(en)}{n}\bigg).
\end{align}
The two terms in the rate (\ref{rate:hp}) trades off two structures of the experiment: the sparsity of $h_\beta(x)$ and the smoothness level of $u(z)$. The resulting phase transition of the rate (\ref{rate:hp}) in terms of these structures is in a sense similar to the results of \cite{yu2016minimax,yuan2015minimax}. It is also easy to derive some explicit rate results from (\ref{rate:hp}). For instance, if $u_0 \in \mathcal{U}$ and $\pnorm{u_0}{\infty}<\infty$, then by using Theorem 4.1 of \cite{chatterjee2015risk}, (\ref{rate:hp}) reduces to $(s_0\log(ep)\wedge \rank(X))/{n}+ n^{-2/3}\log(en)$.

\subsection{Density estimation}

Suppose $X_1,\ldots,X_n$'s are i.i.d. samples from a density $f\in \mathcal{F}$ with respect to a measure $\nu$ on the sample space $(\mathfrak{X},\mathcal{A})$. We consider the following form of $\mathcal{F}$: $
f(x)={e^{g(x)}}/{\int_{\mathfrak{X}} e^{g}\ \d{\nu}}$ for some $g \in \mathcal{G}$ for all $x \in \mathfrak{X}$. For any $f_0,f_1\in \mathcal{F}$, $
h^2(f_0,f_1)\equiv \frac{1}{2}\int_{\mathfrak{X}} (\sqrt{f_0}-\sqrt{f_1})^2\ \d{\nu}$.
\begin{lemma}\label{lem:bernstein_density_estimation}
	Suppose that $\mathcal{G}$ is uniformly bounded. Then Assumption \ref{assump:laplace_cond_kl} is satisfied for $h$ with constants $\{c_i\}_{i=1}^3,\kappa$ depending on $\mathcal{G}$ only.
\end{lemma}


\begin{corollary}\label{cor:density_est}
	For density estimation, let $d_n\equiv h$. If $\mathcal{G}$ is a class of uniformly bounded functions and Assumptions \ref{assump:local_ent_general}-\ref{assump:prior_mass_general} hold, then (\ref{ineq:post_rate_general})-(\ref{ineq:post_mean_general}) hold.
\end{corollary}

\begin{remark}
	Similar to the above remark, the uniform boundedness is included here for simplicity.  See Appendix \ref{section:more_examples} for an example on location mixture model where this restriction is removed.
\end{remark}

\subsection{Gaussian autoregression}\label{section:gaussian_autoreg}
Suppose $X_0,X_1,\ldots,X_n$ is generated from $
X_i = f(X_{i-1})+\epsilon_i$ for $1\leq i\leq n$, where $f$ belongs to a function class $\mathcal{F}$ with a uniform bound $M$, and $\epsilon_i$'s are i.i.d. $\mathcal{N}(0,1)$. Then $X_n$ is a Markov chain with transition density $p_f(y|x)=\phi(y-f(x))$ where $\phi$ is the normal density. By the arguments on page 209 of \cite{ghosal2007convergence}, this chain has a unique stationary distribution with density $q_f$ with respect to the Lebesgue measure $\lambda$ on $\R$. We assume that $X_0$ is generated from this stationary distribution under the true $f$. For any $f_0,f_1 \in \mathcal{F}$, $d_{r,M}^2(f_0,f_1)\equiv \int (f_0-f_1)^2 r_M\ \d{\lambda}$
where $r_M(x)\equiv \frac{1}{2}\left(\phi(x-M)+\phi(x+M)\right)$. 

\begin{lemma}\label{lem:bernstein_gaussian_autoregression}
	Suppose that $\mathcal{F}$ is uniformly bounded by $M$. Then Assumption \ref{assump:laplace_cond_kl} is satisfied for $d_{r,M}$ with constants $\{c_i\}_{i=1}^3,\kappa$ depending on $M$ only.
\end{lemma}

\begin{corollary}\label{cor:rate_auto}
	For Gaussian autoregression model, if $\mathcal{F}$ is uniformly bounded by $M$, let $d_n\equiv d_{r,M}$. If Assumptions \ref{assump:local_ent_general}-\ref{assump:prior_mass_general} hold, then (\ref{ineq:post_rate_general})-(\ref{ineq:post_mean_general}) hold.
\end{corollary}

\cite{ghosal2007convergence} (cf. Section 7.4) uses a weighted $L_s(s>2)$ norm to check the local entropy condition, and an average Hellinger metric as the loss function. Our results here use the metric $d_{r,M}$ defined as a weighted $L_2$ norm.

\subsection{Gaussian time series}
Suppose $X_1,X_2,\ldots$ is a stationary Gaussian process with spectral density $f \in \mathcal{F}$ defined on $[-\pi,\pi]$. Then the covariance matrix of $X^{(n)}=(X_1,\ldots,X_n)$ is given by  $
(T_n(f))_{kl}\equiv \int_{-\pi}^{\pi}e^{\sqrt{-1} \lambda(k-l)}f(\lambda)\ \d{\lambda}$.
We consider a special form of $\mathcal{F}$: $f\equiv f_g\equiv e^{g}$ for some $g \in \mathcal{G}$. For any $g_0,g_1 \in \mathcal{G}$, $
D_n^2(g_0,g_1)\equiv n^{-1} \pnorm{T_n(f_{g_0})-T_n(f_{g_1})}{F}^2$, 
where $\pnorm{\cdot}{F}$ denotes the matrix Frobenius norm.

\begin{lemma}\label{lem:bernstein_gaussian_time_series}
	Suppose that $\mathcal{G}$ is uniformly bounded. Then Assumption \ref{assump:laplace_cond_kl} is satisfied for  $D_n$ with constants $\{c_i\}_{i=1}^3,\kappa$ depending on $\mathcal{G}$ only.
\end{lemma}


\begin{corollary}\label{cor:gaussian_time_series}
	For the Gaussian time series model, if $\mathcal{G}$ is uniformly bounded, let $d_n\equiv D_n$. If Assumptions \ref{assump:local_ent_general}-\ref{assump:prior_mass_general} hold, then (\ref{ineq:post_rate_general})-(\ref{ineq:post_mean_general}) hold.
\end{corollary}

$D_n$ is bounded from above by the usual $L_2$ metric, and can be related to the $L_2$ metric from below (cf. Lemma B.3 of \cite{gao2016rate}). Our result then shows that the metric to use in the entropy condition  can be weakened to the $L_2$ norm rather than the much stronger $L_\infty$ norm as in page 202 of \cite{ghosal2007convergence}. Such improvements are particularly important in, e.g. shape constrained models that are not totally bounded in $L_\infty$ (cf. \cite{guntuboyina2013covering}). See also \cite{choudhuri2004bayesian,rousseau2012bayesian} for some related works in Bayesian spectral density estimation.

\subsection{Covariance matrix estimation}

Suppose $X_1,\ldots,X_n \in \R^p$ are i.i.d. observations from $\mathcal{N}_p(0,\Sigma)$ where $\Sigma \in \mathscr{S}_p(L)$, the set of $p\times p$ covariance matrices whose minimal and maximal eigenvalues are bounded by $L^{-1}$ and $L$ (where $L>1$), respectively. For any $\Sigma_0,\Sigma_1 \in \mathscr{S}_p(L)$, $
D_F^2(\Sigma_0,\Sigma_1)\equiv \pnorm{\Sigma_0-\Sigma_1}{F}^2$.
\begin{lemma}\label{lem:bernstein_cov_est}
	Under the above setting, Assumption \ref{assump:laplace_cond_kl} holds for the metric $D_F$ with constants $\{c_i\}_{i=1}^3,\kappa$ depending on $L$ only.
\end{lemma}


\begin{corollary}\label{cor:rate_cov_general}
	For covariance matrix estimation in $\mathscr{S}_p(L)$ for some $L<\infty$, let $d_n\equiv D_F$. If Assumptions \ref{assump:local_ent_general}-\ref{assump:prior_mass_general} hold, then (\ref{ineq:post_rate_general})-(\ref{ineq:post_mean_general}) hold.
\end{corollary}

\subsubsection{Example: Covariance matrix estimation in the sparse factor model}\label{section:cov_matrix_est}

Suppose we observe i.i.d. $X_1,\ldots,X_n\in \R^p$ from $\mathcal{N}_p(0,\Sigma_0)$. The covariance matrix is modelled by the sparse factor model $\mathfrak{M}\equiv \cup_{(k,s)\in \N^2 }\mathfrak{M}_{(k,s)}$ where $\mathfrak{M}_{(k,s)}\equiv \{\Sigma=\Lambda \Lambda^\top +I : \Lambda \in \mathscr{R}_{(k,s)}(L)\}$ with $\mathscr{R}_{(k,s)}(L)\equiv \{\Lambda\in \R^{p\times k}, \Lambda_{\cdot j} \in B_0(s),  \abs{\sigma_j(\Lambda)}\leq L^{1/2},\forall 1\leq j\leq k\}$. In this example, the model index $\mathcal{I}$ is a 2-dimensional lattice, and the sparsity structure depends on the rank structure. Consider the following model selection prior:
\begin{align}\label{eqn:prior_cov}
\lambda_{n}((k,s))\propto \exp\left(-c \cdot ks\log(e p)\right),
\end{align}
where $c>0$ is a constant to be specified later.

\begin{theorem}\label{thm:rate_cov}
	Let $p\geq n$. There exist some $c>0$ in (\ref{eqn:prior_cov}) and some sequence of sieve priors $\Pi_{n,(k,s)}$ on $\mathfrak{M}_{(k,s)}$ depending on $L$ such that
	\begin{align}\label{ineq:post_rate_cov}
	& P_{\Sigma_0}^{(n)} \Pi_n\big(\Sigma \in \mathfrak{M}: \pnorm{\Sigma-\Sigma_0}{F}^2> C_1(\epsilon_{n,(k,s)}^{\mathrm{cov}})^2\big\lvert X^{(n)}\big) \leq C_2 e^{-n (\epsilon_{n,(k,s)}^{\mathrm{cov}})^2/C_2 }.\nonumber
	\end{align}
	Here $
	(\epsilon_{n,(k,s)}^{\mathrm{cov}})^2\equiv \max\{ \inf_{\Sigma' \in \mathfrak{M}_{(s,k)}}\pnorm{\Sigma'-\Sigma_0}{F}^2, {ks\log(e p)}/{n}\}$, and the constants $C_i(i=1,2)$ depend on $L$.
\end{theorem}

Since spectral norm (non-intrinsic) is dominated by Frobenius norm (intrinsic), our result shows that if the model is well-specified (i.e. $\Sigma_0 \in \mathfrak{M}$), then we can construct an adaptive Bayes estimator with convergence rates in both norms no worse than $\sqrt{ks \log p/n}$. \cite{pati2014posterior} considered the same sparse factor model, where they proved a strictly sub-optimal rate $\sqrt{k^3 s \log p \log n/n}$ in spectral norm under $ks\gtrsim \log p$. \cite{gao2015ratepca} considered a closely related sparse PCA problem, where the convergence rate under spectral norm achieves the same rate as here (cf. Theorem 4.1 therein), while a factor of $\sqrt{k}$ is lost when using Frobenius norm as a loss function (cf. Remark 4.3 therein). 

It should be mentioned that the sieve prior $\Pi_{n,(k,s)}$ is constructed using the metric entropy of $\mathfrak{M}_{(k,s)}$ and hence the resulting Bayes estimator and the posterior mean as a point estimator are purely theoretical. We use this example to illustrate (i) the construction scheme of a (nearly) optimal adaptive procedure for a multi-structured experiment based on the metric entropy of the underlying parameter space, and (ii) derivation of contraction rates in non-intrinsic metrics when these metrics can be related to the intrinsic metrics nicely.

It is also possible to use similar strategies as above in the closely related problem of estimating a sparse precision matrix (cf. \cite{banerjee2015bayesian}), but we refrain from repetitive details here.

\subsection{Image boundary detection}\label{section:image_detection}
Consider the setup in \cite{li2017bayesian} as follows. Let $\{f(\cdot;\phi): \phi \in  \R^p\}$ be a class of densities dominated by a $\sigma$-finite measure $\mu$ and indexed by a $p$-dimensional parameter $\phi$ \footnote{For instance, for the binary model considered in Section \ref{section:regression}, we may take $p=1$, $\phi \in [0,1]$ and $f(\cdot,\phi)$ to be the density of $\mathrm{Bern}(\phi)$ with respect to the counting measure on $\{0,1\}$.}. Suppose we observe $\{(X_i,Y_i) \in [0,1]^d\times \R\}_{i=1}^n$ according to the following law: $X_i$'s are i.i.d. uniformly distributed on $[0,1]^d$, and there exists a closed region $\Gamma_0 \subset [0,1]^d$ such that $Y_i \sim f(\cdot;\xi_0)\bm{1}_{X_i \in \Gamma_0}+f(\cdot;\rho_0)\bm{1}_{X_i \in \Gamma_0^c}$. Here $X_i$ can be understood as the location of $i$-th observation and $Y_i$ the corresponding pixel intensity. Let $\theta = (\xi,\rho,\Gamma)\in \Theta$ be the parameter and define for any $\theta_i = (\xi_i,\rho_i,\Gamma_i)(i=0,1)$,
\begin{align*}
d_n^2(\theta_0,\theta_1)&\equiv \pnorm{\xi_0-\xi_1}{2}^2 \lambda(\Gamma_0\cap \Gamma_1)+\pnorm{\rho_0-\rho_1}{2}^2\lambda(\Gamma_0^c\cap \Gamma_1^c)\\
&\qquad  + \pnorm{\xi_0-\rho_1}{2}^2 \lambda(\Gamma_0\cap \Gamma_1^c)+\pnorm{\rho_0-\xi_1}{2}^2 \lambda(\Gamma_0^c\cap \Gamma_1).
\end{align*}
Here $\lambda$ denotes the Lebesgue measure on $[0,1]^d$ and $\lambda(B)=\int_B\ \d{\lambda}$. Clearly $d_n$ is symmetric, but may not satisfy the triangle inequality.

\begin{lemma}\label{lem:bernstein_image}
	Suppose that $\{f(\cdot;\phi): \phi \in \Theta \subset \R^p\}$ is any parametric class considered in Section \ref{section:regression} (i.e. Gaussian/Laplace/binary/Poisson models). Then Assumption \ref{assump:laplace_cond_kl} holds for $d_n$ defined above with constants depending only through the specific parametric class.
\end{lemma}

The following lemma relates $d_n$ to the metric $\lambda(\cdot \Delta \cdot)$ of interest when two elements in $\Theta$ are close to each other in $d_n$.

\begin{lemma}\label{lem:lower_bound_metric_image}
	Suppose that $\pnorm{\xi_0-\rho_0}{2}^2=r_0^2>0$ and $ \lambda(\Gamma_0^c\cap \Gamma_1^c)\geq \lambda_0^2>0$. If $d_n^2(\theta_0,\theta_1)\leq (\frac{\lambda_0^2}{4}\wedge \frac{\lambda(\Gamma_0)}{8})r_0^2$, then $\lambda(\Gamma_0\Delta\Gamma_1)\leq (8/r_0^2)\cdot d_n^2(\theta_0,\theta_1)$.
\end{lemma}

Now we can state our main result in this section. Let $\Theta_1\subset\ldots\subset\Theta_m\subset\ldots\subset \Theta$ be a sequence of nested models.

\begin{corollary}\label{cor:rate_image}
	Suppose that $\{f(\cdot;\phi): \phi \in \Theta \subset \R^p\}$ is any parametric class considered in Section \ref{section:regression}, and that there exist some $m \in \N, \eta>0$ such that $\theta_0=(\xi_0,\rho_0,\Gamma_0) \in \Theta_m$ with $\Gamma_0 \subset [\eta,1-\eta]^d$ and $\xi_0\neq \rho_0$, and $\Pi_n(\Gamma \subset [\eta,1-\eta]^d)=1$. If Assumptions \ref{assump:local_ent_general}-\ref{assump:prior_mass_general} hold for $d_n$ described above with $\theta_{0,m}$ replaced by $\theta_0$, then for $n$ large enough (depending only on $\xi_0,\rho_0,\eta$), we have
	\begin{align*}
	P_{\theta_0}^{(n)}\Pi_n\big(\Gamma: \lambda(\Gamma\Delta\Gamma_0)>C_1\delta_{n,m}^2\big\lvert \big(X^{(n)},Y^{(n)}\big) \big)\leq C_2 e^{-n\delta_{n,m}^2/C_2}.
	\end{align*}
	Here the constants $\{C_i\}_{i=1}^2>0$ depend on $\xi_0,\rho_0,\eta$.
\end{corollary}


Our result can be used for smooth boundaries as studied in \cite{li2017bayesian}, but we will be mainly interested in non-smooth boundaries. Indeed, we will propose a hierarchical prior (cf. Section \ref{section:application_detection_polytopal_boundary}) so that the posterior distribution is nearly parametrically rate-adaptive to non-smooth polytopal regions $\Gamma$.

\subsubsection{Example: Detection of polytopal image boundaries}\label{section:application_detection_polytopal_boundary}

For simplicity of presentation, we specify the binary model for $\{f(\cdot;\phi):\phi \in [\eta,1-\eta]\}$, and consider $d=2$. Suppose that $\theta_0= (\xi_0,\rho_0,\Gamma_0)$ where $\Gamma_0 \subset [\eta,1-\eta]^2$ is a convex polytope. A natural nested sequence of models $\{\Theta_m\}_{m \in \N}$ is given by $\Theta_m\equiv \{(\xi,\rho,\Gamma): \xi\neq \rho, \Gamma \in \mathscr{C}_m\}$ where $\mathscr{C}_m$ contains all convex polytopes in $[\eta,1-\eta]^2$ with at most $m$ vertices. Consider the following model selection prior:
\begin{equation}\label{eqn:prior_im}
\lambda_n(m)\propto \exp\big(-c\cdot m \log (en)\big),
\end{equation}
where $c>0$ is a constant to be specified later. A prior $\Pi_{n,m}$ on the model $\Theta_m$ can be induced by a product prior on $(\xi,\rho,\Gamma)$. In particular, we put priors on $\xi$ and $\rho$ with densities $g_\xi$ and $g_\rho$ respectively, and a prior on $\Gamma$ can be induced by taking the convex hull of randomly generated $m$ points in $[\eta,1-\eta]^2$ with density $g^{\otimes m}_\Gamma$. For simplicity, we assume that $g_\xi,g_\rho,g_\Gamma$ all follow the uniform distribution on $[\eta,1-\eta]$.

\begin{theorem}\label{thm:rate_image_polytope}
	In the above setting, if $\theta_0 \in \Theta_m$ with $\xi_0\neq \rho_0$, then there exists some $c>0$ in (\ref{eqn:prior_im}) such that for $n$ large enough,
	\begin{align*}
	P_{\theta_0}^{(n)}\Pi_n\big(\Gamma: \lambda(\Gamma\Delta\Gamma_0)>C_1 {m\log n}/{n}\big\lvert \big(X^{(n)},Y^{(n)}\big)\big)\leq C_2 e^{-m\log n/C_2 }.
	\end{align*}
	Here the constants $\{C_i \}_{i=1}^2$ depend on $\xi_0,\rho_0,\eta$.
\end{theorem}

\subsection{Intensity estimation in a Poisson point process model}\label{section:intensity_ppp}

Suppose we observe $\{(X_i,Y_i) \in [0,1]\times \R\}$ from a Poisson point process $N$ defined on $[0,1]\times \R$ with intensity $
\lambda(x,y)\equiv \lambda_f(x,y) = n \bm{1}_{f(x)\leq y}$. 
The goal is to recover the boundary $f:[0,1]\to \R$ of the support of the intensity $\lambda$ \footnote{This model can be regarded as a continuous analogue of the regression problem with irregular errors \cite{meister2013asymptotic}.}.

Note that a dominating measure $\mu$ is not well-specified for all probability distributions $P_f^{(n)}$, and the likelihood ratio $\d{P_{f_0}^{(n)}}/\d{P_{f_1}^{(n)}}$ is well-defined only if $f_1\leq f_0$. Indeed, \cite{reiss2017nonparametric} showed (cf. Lemma 2.1 therein) that for $f_1\leq f_0$, $
{\d{P_{f_0}^{(n)}}}/{\d{P_{f_1}^{(n)}}}= e^{n \pnorm{f_0-f_1}{1}} \bm{1}_{\forall i: f_0(X_i)\leq Y_i}$, 
and therefore the Kullback-Leibler divergence is given by
\begin{align*}
\bar{L}_1(f_0,f_1)&=
\begin{cases}
\pnorm{f_0-f_1}{1}, & f_1\leq f_0;\\
\infty, & \hbox{otherwise}.
\end{cases}
\end{align*}
%
The technical problem here is that $\bar{L}_1$ is not symmetric---fortunately by a slight modification, our machinery can still be applied. To this end, suppose $\inf_{g \in \mathcal{F}_m} \bar{L}_1(f_0,g)<\infty$, and let $f_{0,m} \in \arg\min_{g \in \mathcal{F}_m} \bar{L}_1(f_0,g)$ (so that $f_{0,m}\leq f_0$), assumed to be well-defined.

\begin{corollary}\label{cor:rate_support_boundary}
	For the support boundary recovery problem described above, let $d_n^2\equiv  \bar{L}_1$. If (i) Assumption \ref{assump:local_ent_general} holds under entropy with left bracketing\footnote{For a generic function class $\mathcal{G}$ defined on $[0,1]$, the left bracketing number $\mathcal{N}_{[}(\epsilon,\mathcal{G},\bar{L}_1)$ is the smallest number $M$ of functions $g_1,\ldots,g_M$ such that for any $g \in \mathcal{G}$ there exists some $j \in \{1,\ldots,M\}$ with $g_j\leq g$ and $\int_0^1 (g-g_j)\leq \epsilon$. Note that in this definition $g_j$ need not belong to $\mathcal{G}$.
	} and the set in (\ref{cond:entropy_essen}) restricted to $f\geq f_0$; (ii) Assumption \ref{assump:prior_mass_general} holds with the set in (P2) restricted to $f\geq f_{0,m}$, then (\ref{ineq:post_rate_general})-(\ref{ineq:post_mean_general}) hold with the posterior distribution restricted to $f\geq f_0$.
\end{corollary}

In Section \ref{section:application_intensity} we will use the above result to derive oracle contraction rates for estimating piecewise constant intensities. 

It is also possible to consider the two-sided $L_1$ loss, at the expense of stronger conditions. Below is a result in this direction.

\begin{corollary}\label{cor:rate_support_boundary_L1}
	Suppose that for $m \in \mathcal{M}$, (i) $\log \mathcal{N}(\delta_{n,m}^2, \mathcal{F}_m, L_\infty)\leq C_1 n\delta_{n,m}^2$ and (ii) $\Pi_{n,m}(f \in \mathcal{F}_m:L_1(f,f_{0,m})\leq C_2 \delta_{n,m}^2, f\leq f_{0,m} )\geq e^{-C_2n\delta_{n,m}^2}$ hold for some $f_{0,m}\leq f_0$. Then using the prior (\ref{eqn:prior_generic}), there exists some constant $C'>0$ such that for any $m \in \mathcal{M}$,
	\begin{align*}
	P_{f_0}^{(n)} \Pi_n \big(f \in \mathcal{F}: L_1(f,f_0)\geq C' \epsilon_{n,m}^2|(X^{(n)},Y^{(n)})\big)\leq C' e^{-\epsilon_{n,m}^2/C'}.
	\end{align*}
	Here $\epsilon_{n,m}^2\equiv \max\{L_1(f_0,f_{0,m}),\delta_{n,m}^2\}$. 
\end{corollary}

\subsubsection{Example: Estimating piecewise constant intensity in a Poisson point process model}\label{section:application_intensity}

Consider fitting the intensity $\lambda_f$ in the Poisson point process model by the class of piecewise constant functions $\mathcal{F}\equiv \cup_{m=1}^\infty \mathcal{F}_m\equiv  \{f: f=\sum_{j=1}^m a_j \bm{1}_{[t_{j-1},t_j)}, 0=t_0<t_1<\ldots<t_{m-1}<t_m=1\}$. A prior on $\mathcal{F}_m$ can be induced by a prior $\Pi_{n,m}^t$ on $\{t_1<\ldots<t_{m-1}\}$ followed by a prior $\Pi_{n,m}^a$ on $\{a_j\}_{j=1}^m$. More specifically, we choose $\Pi_{n,m}^t$ with density $\bm{t}=(t_1,\ldots,t_{m-1})\mapsto (m-1)! \bm{1}_{t_1<\ldots<t_{m-1}}(\bm{t})$, and $\Pi_{n,m}^a$ with product density $g_a^{\otimes m}$. As before, we assume that $g_a$ is symmetric, non-increasing and satisfies the following: $g_a$ has full support, and there exists some sequence $\{R_n\}$ with $\log R_n\lesssim \log n$, and a large enough absolute constant $C'>0$ such that
\begin{align}\label{cond:prior_density_intensity}
\int_{\abs{x}>R_n} g_a(x)\ \d{x}\leq n^{-C'}.
\end{align}
It is easily seen that this condition is very weak, and essentially does not require any tail condition on $g_a$. The reason for this to occur is that the information geometry of the model studied here does not change with the $L_\infty$ size of the model---the impact of this only occurs through the complexity of the model by logarithmic factors.

Consider the following prior $\Lambda_n$ on the model index $\mathcal{I}\equiv \N$:
\begin{align}\label{eqn:prior_intensity}
\lambda_{n}(m)\propto \exp\big(-c\cdot m  \log (en)\big),
\end{align}
where $c>0$ is a constant to be specified later. 

\begin{theorem}\label{thm:rate_intensity}
	Suppose that $\pnorm{f_0}{\infty}<\infty$ and (\ref{cond:prior_density_intensity}) holds for the prior density $g_a$. There exists some $c>0$ in (\ref{eqn:prior_intensity}) such that for $n$ large enough (depending only on $f_0$ and the prior $g_a$), with $(\epsilon_{n,m}^{\mathrm{int}})^2\equiv \max\{ \inf_{g \in \mathcal{F}_m}\bar{L}_1(f_0,g), { m\log (en)}/{n}\}$, 
	\begin{align*}
	& P_{f_0}^{(n)} \Pi_n\big(f \geq f_0: \bar{L}_1(f,f_{0})> C_1 (\epsilon_{n,m}^{\mathrm{int}})^2\big\lvert N\big)  \leq C_2 e^{- m\log (en)/C_2}.\nonumber
	\end{align*}
	Here the constants $C_i(i=1,2)$ are absolute.
\end{theorem}

Compared with Theorem 5.3 of \cite{reiss2017nonparametric}, our Theorem \ref{thm:rate_intensity} works with a slightly weaker one-sided $L_1$ loss, but enjoys an exact form of an oracle posterior contraction rate. From here it is straightforward to derive rate result assuming H\"older smoothness on $f_0$ (as in \cite{reiss2017nonparametric}). Note that here we do not require the technical condition $\log m\gtrsim \log n$ as in \cite{reiss2017nonparametric}, so our result here shows rate-adaptivity of the posterior distribution to intensities with fixed number of constant pieces.

\appendix

\section{Proofs for Section 2}\label{section:proof_main_result}

\subsection{Proof of Theorem \ref{thm:general_ms}: main steps}

First we need a lemma allowing a change-of-measure argument.
\begin{lemma}\label{lem:change_variable}
	Let Assumption \ref{assump:laplace_cond_kl} hold. There exists some constant $c_4\geq 1$ only depending on $c_1, c_3$ and $\kappa$ such that for any random variable $U \in [0,1]$, any $\delta_n\geq d_n(f_0,f_1)$ and any $j \in \N$, 
	\begin{align*}
	P_{f_0}^{(n)} U \leq  c_4[P_{f_1}^{(n)} U\cdot e^{c_4 nj\delta_n^2}+e^{-c_4^{-1} nj\delta_n^2}].
	\end{align*} 
\end{lemma}

The next propositions solve the posterior contraction problem for the `local' model $\mathcal{F}_m$.

\begin{proposition}\label{prop:overfit_general}
	Fix $m \in \mathcal{M}$ such that $\delta_{n,m}^2\geq d_n^2(f_0,f_{0,m})$. Then there exists some constant $c_8\geq 1$ (depending on the constants in Assumption \ref{assump:laplace_cond_kl}) such that for $j\geq 8\mathfrak{c}^2/c_7\mathfrak{h}$,
	\begin{align}\label{ineq:overfit_oracle}
	P_{f_{0,m}}^{(n)} \Pi_n(f \in \mathcal{F}: d_n^2(f,f_{0,m})> \mathfrak{c}^2 (j\mathfrak{h})^\gamma \delta_{n,m}^2 \big\lvert X^{(n)})\leq  c_8 e^{-{n j \mathfrak{h} \delta_{n,m}^2}/{c_8 \mathfrak{c}^2}}.
	\end{align}
\end{proposition}

\begin{proposition}\label{prop:underfit_general}
	Fix $m \in \mathcal{M}$ such that $\delta_{n,m}^2<d_n^2(f_0,f_{0,m})$. Let $
	\tilde{m}\equiv \tilde{m}(m)\equiv \inf\{m' \in \mathcal{M}, m'\geq m: \delta_{n,m'}\geq d_n(f_0,f_{0,m})\}$. Then for $j\geq 8\mathfrak{c}^2/c_7\mathfrak{h}$,
	\begin{align}\label{ineq:underfit_oracle}
	P_{f_{0,m}}^{(n)} \Pi_n(f \in \mathcal{F}: d_n^2(f,f_{0,m})>\mathfrak{c}^4 (2j\mathfrak{h})^\gamma d_n^2(f_0,f_{0,m}) \big\lvert X^{(n)})\leq  c_8 e^{-{n j \mathfrak{h} \delta_{n,\tilde{m}}^2}/{c_8 \mathfrak{c}^2}}.
	\end{align}
\end{proposition}

The proofs of these results will be detailed in later subsections.

\begin{proof}[Proof of Theorem \ref{thm:general_ms}: main steps]
	Instead of (\ref{ineq:post_rate_general}), we will prove a slightly stronger statement as follows: for any $ j\geq 8\mathfrak{c}^2/c_7\mathfrak{h}$, and $\mathfrak{h}\geq 2c_4c_8\mathfrak{c}^2$,
	\begin{align}\label{ineq:post_rate_general_j}
	P_{f_0}^{(n)} \Pi_n\big(f \in \mathcal{F}: d_n^2(f,f_{0})> \mathfrak{c}_1 j^\gamma \epsilon_{n,m}^2\big\lvert X^{(n)}\big) \leq \mathfrak{c}_2 e^{-jn\epsilon_{n,m}^2/\mathfrak{c}_2}.
	\end{align}
	Here the constants $\mathfrak{c}_i(i=1,2)$ depends on the constants involved in Assumption \ref{assump:laplace_cond_kl} and $\mathfrak{c},\mathfrak{h}$.
	
	\noindent \underline{\textbf{Proof of  (\ref{ineq:post_rate_general_j}).}}  
	
	First consider the overfitting case. By Proposition \ref{prop:overfit_general} and Lemma \ref{lem:change_variable}, we see that when $\delta_{n,m}^2\geq d_n^2(f_0,f_{0,m})$ holds, for $ j\geq 8\mathfrak{c}^2/c_7\mathfrak{h}$, it holds that 
	\begin{align*}
	&P_{f_0}^{(n)} \Pi_n\big(f \in \mathcal{F}: d_n^2(f,f_{0})>2d_n^2(f_0,f_{0,m})+2\mathfrak{c}^2 (j\mathfrak{h})^\gamma \delta_{n,m}^2\big\lvert X^{(n)}\big)\\
	& \leq P_{f_0}^{(n)} \Pi_n\big(f \in \mathcal{F}: d_n^2(f,f_{0,m})>\mathfrak{c}^2 (j\mathfrak{h})^\gamma \delta_{n,m}^2 \big\lvert X^{(n)}\big)\\ 
	& \leq c_4\big[P_{f_{0,m}}^{(n)} \Pi_n\big(f \in \mathcal{F}: d_n^2(f,f_{0,m})>\mathfrak{c}^2 (j\mathfrak{h})^\gamma \delta_{n,m}^2 \big\lvert X^{(n)}\big) e^{c_4 nj\delta_{n,m}^2}+e^{-c_4^{-1} nj\delta_{n,m}^2}\big]\\
	& \leq c_8 c_4 e^{-nj\delta_{n,m}^2\left(\frac{  \mathfrak{h} }{c_8 \mathfrak{c}^2}-c_4\right)}+c_4 e^{-c_4^{-1} nj\delta_{n,m}^2}\leq 2c_8 c_4 e^{-jn\delta_{n,m}^2\min\{c_4,c_4^{-1}\}}.
	\end{align*}
	Here in the second line we used the fact that $d_n^2(f,f_{0,m})\geq d_n^2(f,f_0)/2 - d_n^2(f_0,f_{0,m})$. 
	
	Next consider the underfitting case: fix $m \in \mathcal{M}$ such that $\delta_{n,m}^2<d_n^2(f_0,f_{0,m})$. Apply Proposition \ref{prop:underfit_general} and Lemma \ref{lem:change_variable}, and use similar arguments to see that for $ j\geq 8\mathfrak{c}^2/c_7\mathfrak{h}$, 
	\begin{align*}
	& P_{f_0}^{(n)} \Pi_n\big(f \in \mathcal{F}: d_n^2(f,f_{0})> \big[2\mathfrak{c}^4 (2j\mathfrak{h})^\gamma+2\big]d_n^2(f_0,f_{0,m}) \big\lvert X^{(n)}\big)\\
	& \leq c_4\big[P_{f_{0,m}}^{(n)} \Pi_n\big(f \in \mathcal{F}: d_n^2(f,f_{0,m})> \mathfrak{c}^4 (2j\mathfrak{h})^\gamma d_n^2(f_0,f_{0,m}) \big\lvert X^{(n)}\big) e^{c_4nj\delta_{n,\tilde{m}}^2 }+e^{-c_4^{-1} jn \delta_{n,\tilde{m}}^2 }\big]\\
	& \leq 2c_8c_4 e^{-nj\delta_{n,\tilde{m}}^2\min\{c_4,c_4^{-1}\}}.
	\end{align*}
	Here in the second line we used (i) $2d_n^2(f,f_{0,m})\geq d_n^2(f,f_0)-2d_n^2(f_0,f_{0,m})$, and (ii) $\delta_{n,\tilde{m}}\geq d_n(f_0,f_{0,m})$. The claim of (\ref{ineq:post_rate_general_j}) follows by combining the estimates.

	\noindent \underline{\textbf{Proof of  (\ref{ineq:post_mean_general}).}} The proof is essentially integration of tail estimates by a peeling device. Let the event $A_j$ be defined via
	\begin{align*}
	A_j:=\{\mathfrak{c}_1 j^\gamma \big( d_n^2(f_0,f_{0,m})+\delta_{n,m}^2\big)<d_n^2(f,f_0)\leq  \mathfrak{c}_1 (j+1)^\gamma \big(d_n^2(f_0,f_{0,m})+\delta_{n,m}^2\big)\}.
	\end{align*}  
	Then,
	\begin{align*}
	&P_{f_0}^{(n)} d_n^2(\hat{f}_n,f_0) = P_{f_0}^{(n)} d_n^2\left(\Pi_n(f|X^{(n)}),f_0\right)\leq P_{f_0}^{(n)} \Pi_n \left(d_n^2(f,f_0)|X^{(n)}\right) \\ 
	&\leq C_{\mathfrak{c}_1,\mathfrak{c},c_7,\mathfrak{h},\gamma} \big( d_n^2(f_0,f_{0,m})+\delta_{n,m}^2\big)+\sum_{j\geq 8\mathfrak{c}^2/c_7\mathfrak{h}} P_{f_0}^n \Pi_n\big(d_n^2(f,f_0)\bm{1}_{A_j}\big\lvert X^{(n)}\big)\\ 
	&\leq C_{\mathfrak{c}_1,\mathfrak{c},c_7,\mathfrak{h},\gamma} \big( d_n^2(f_0,f_{0,m})+\delta_{n,m}^2\big)+ \frac{2^{\gamma+1} \mathfrak{c}_1 \mathfrak{c}_2}{n} \sum_{j\geq 8\mathfrak{c}^2/c_7\mathfrak{h}}  j^\gamma n\epsilon_{n,m}^2  e^{-jn\epsilon_{n,m}^2/\mathfrak{c}_2}.
	\end{align*}
	The inequality in the first line of the above display is due to Jensen's inequality applied with $d_n^2(\cdot,f_0)$ (the convexity follows since $f\mapsto d_n(f,f_0)$ is non-negatively convex, so is its square), followed by Cauchy-Schwarz inequality. The summation can be bounded up to a constant depending on $\gamma,\mathfrak{c}_1,\mathfrak{c}_2$ by 
	\begin{align*}
	\sum_{j\geq 8\mathfrak{c}^2/c_7\mathfrak{h}}  (j n\epsilon_{n,m}^2)^\gamma  e^{-jn\epsilon_{n,m}^2/\mathfrak{c}_2}\leq \sum_{j\geq 8\mathfrak{c}^2/c_7\mathfrak{h}}  (j n\epsilon_{n,m}^2)^\gamma  e^{-jn\epsilon_{n,m}^2/\mathfrak{c}_2} \big((j+1)n\epsilon_{n,m}^2-jn\epsilon_{n,m}^2\big),
	\end{align*} 
	where the inequality follows since $n\epsilon_{n,m}^2\geq n\epsilon_{n,1}^2\geq 1$. This quantity can be bounded by a constant multiple of $\int_0^\infty x^\gamma e^{-x/\mathfrak{c}_2}\ \d{x}$ independent of $m$. Now the proof is complete by noting that $\delta_{n,m}^2$ majorizes $1/n$ up to a constant, and then taking infimum over $m \in \mathcal{M}$.
\end{proof}

\subsection{Proofs of Propositions \ref{prop:overfit_general} and \ref{prop:underfit_general}}

We will need several lemmas before the proof of Propositions \ref{prop:overfit_general} and \ref{prop:underfit_general}.

\begin{lemma}\label{lem:global_test_general}
	Let Assumption \ref{assump:laplace_cond_kl} hold. Let $\mathcal{F}$ be a function class defined on the sample space $\mathfrak{X}$. Suppose that $N :\R_{\geq 0}\to \R_{\geq 0}$ is a non-increasing function such that for some $\epsilon_0\geq \sqrt{2/(c_2\wedge c_3)} \cdot d_0$ and every $\epsilon\geq \epsilon_0$, the following entropy estimate holds: 
	\begin{align*}
	\mathcal{N}\left(c_5\epsilon, \{f \in \mathcal{F}: \epsilon<d_n(f,f_0)\leq 2\epsilon\},d_n\right)\leq N(\epsilon).
	\end{align*} 
	Then for any $\epsilon\geq \epsilon_0$, there exists some test $\phi_n$ such that 
	\begin{align*}
	P_{f_0}^{(n)} \phi_n \leq   {c_6 N (\epsilon)e^{-c_7 n\epsilon^2}}/(1-e^{-c_7n\epsilon^2 }),\quad  \sup_{f\in \mathcal{F}: d_n(f,f_0)\geq \epsilon} P_f^{(n)} (1-\phi_n)\leq c_6 e^{-c_7 n\epsilon^2}.
	\end{align*}
	The constants $c_5,c_6,c_7$ are taken from Lemma \ref{lem:local_test_general}.
\end{lemma}

\begin{lemma}\label{lem:posterior_denum_control}
	Fix $\epsilon>0$. Let Assumption \ref{assump:laplace_cond_kl} holds for some $d_0$ such that $\epsilon\geq \sqrt{2/(c_2\wedge c_3)}\cdot d_0$. Suppose that $\Pi$ is a probability measure on $\{f \in \mathcal{F}: d_n(f,f_0)\leq \epsilon\}$. Then for every $C>0$, there exists some $C'>0$ depending on $C,\kappa$ such that $
	P_{f_0}^{(n)}\big(\int  {p_f^{(n)}}/{p_{f_0}^{(n)}}\ \d{\Pi(f)}\leq e^{-(C+c_3)n\epsilon^2}\big)\leq c_1 e^{-C' n\epsilon^2}$.
\end{lemma}

The proof of these lemmas can be found in Appendix \ref{section:proof_lemma_main_result}.

\begin{proof}[Proof of Proposition \ref{prop:overfit_general}]
	Fix $m' \in \mathcal{M}$ with $m'\geq m$. 
	Now we invoke Lemma \ref{lem:global_test_general} with $\mathcal{F}\equiv \mathcal{F}_{m'}$, $f_0\equiv {f_{0,m}} \in \mathcal{F}_m\subset \mathcal{F}_{m'}$ [since $m'\geq m$], $\epsilon_{0}\equiv \delta_{n,m'}$ and $\log N (\epsilon) \equiv (c_7/2)n\delta_{n,m'}^2$ for $\epsilon= \epsilon_0$  to see that, there exists some test $\phi_{n,m'}$ such that
	\begin{align}\label{ineq:generic_pdim_1}
	P_{f_{0,m}}^{(n)}\phi_{n,m'} \leq {c_6 e^{\log N(\epsilon)-c_7n \delta_{n,m'}^2}}/(1-e^{-c_7n\delta_{n,m'}^2})\leq 2c_6 e^{-c_7 n\delta_{n,m'}^2/2},
	\end{align}
	and that 
	\begin{align}\label{ineq:generic_pdim_2}
	\sup_{f \in \mathcal{F}_{m'}:d_n^2(f,f_{0,m})\geq  \delta_{n,m'}^2} P_f^{(n)}(1-\phi_{n,m'})\leq c_6 e^{-c_7 n\delta_{n,m'}^2}.
	\end{align}
	Note that here in (\ref{ineq:generic_pdim_1}) we used the fact that $n\delta_{n,m'}^2\geq 2/c_7$ by definition of $\delta_{n,m'}$. Now for the fixed $j, m $ as in the statement of the proposition, we let $
	\phi_n:= \sup_{m' \in \mathcal{I}: m'\geq j\mathfrak{h}m} \phi_{n,m'}$ 
	be a global test for big models. Then by (\ref{ineq:generic_pdim_1}), 
	\begin{align*}
	P_{f_{0,m}}^{(n)} \phi_n \leq \sum_{m' \geq j \mathfrak{h} m } P_{f_{0,m}}^{(n)} \phi_{n,m'}\leq \sum_{m' \geq j \mathfrak{h} m } 2c_6 e^{-c_7 n\delta_{n,m'}^2/2}  \leq  4c_6 e^{-(c_7/2\mathfrak{c}^2)nj \mathfrak{h}\delta_{n,m}^2}.
	\end{align*} 
	Here we used the left side of (\ref{cond:suplinearity_delta}). This implies that for any random variable $U \in [0,1]$, we have
	\begin{align}\label{ineq:generic_pdim_3}
	P_{f_{0,m}}^{(n)}U\cdot  \phi_n \leq P_{f_{0,m}}^{(n)}\phi_n\leq  4c_6 e^{-(c_7/2\mathfrak{c}^2)nj \mathfrak{h}\delta_{n,m}^2}.
	\end{align}
	On the power side, with $m'=j\mathfrak{h}m$ applied to (\ref{ineq:generic_pdim_2}) we see that 
	\begin{align}\label{ineq:generic_pdim_4}
	&\sup_{ \substack{f \in \mathcal{F}_{j\mathfrak{h}m}:\\ d_n^2(f,f_{0,m})\geq \mathfrak{c}^2 (j\mathfrak{h})^\gamma \delta_{n,m}^2}} P_{f}^{(n)}(1-\phi_n) \leq \sup_{\substack{f \in \mathcal{F}_{j\mathfrak{h}m}:\\ d_n^2(f,f_{0,m})\geq \delta_{n,j\mathfrak{h}m}^2}} P_{f}^{(n)}(1-\phi_n) \\
	&\leq c_6e^{- c_7n\delta_{n,j\mathfrak{h}m}^2}\leq 2c_6 e^{-(c_7/\mathfrak{c}^2)nj \mathfrak{h}\delta_{n,m}^2}.\nonumber
	\end{align}
	The first inequality follows from the right side of (\ref{cond:suplinearity_delta}) since $\mathfrak{c}^2 (j\mathfrak{h})^\gamma \delta_{n,m}^2\geq \delta_{n,j\mathfrak{h}m}^2$, and the last inequality follows from the left side of (\ref{cond:suplinearity_delta}). On the other hand, by applying Lemma \ref{lem:posterior_denum_control} with $C=c_3$ and  $\epsilon^2\equiv {c_7 j\mathfrak{h}\delta_{n,m}^2}/{8c_3\mathfrak{c}^2}$, we see that there exists some event $\mathcal{E}_n$ such that 
	\begin{align*}
	P_{f_{0,m}}^{(n)}(\mathcal{E}_n^c) \leq c_1 e^{- C' {c_7 n j\mathfrak{h}\delta_{n,m}^2}/{8c_3\mathfrak{c}^2}}
	\end{align*}
	and it holds on the event $\mathcal{E}_n$ that
	\begin{align}\label{ineq:generic_pdim_5}
	&\int {p_f^{(n)}}/{p_{f_{0,m}}^{(n)} }\ \d{\Pi(f)}\geq \lambda_{n}(m) \int_{\{f \in \mathcal{F}_m: d_n^2(f,f_{0,m})\leq {c_7 j\mathfrak{h}\delta_{n,m}^2}/{8c_3\mathfrak{c}^2}  \}}  {p_f^{(n)}}/{p_{f_{0,m}}^{(n)} }\ \d{\Pi_{n,m}(f)}\\
	&\geq \lambda_{n}(m)e^{-\frac{c_7n j\mathfrak{h}\delta_{n,m}^2}{4\mathfrak{c}^2}} \Pi_{n,m}\left(\left\{f \in \mathcal{F}_m: d_n^2(f,f_{0,m})\leq {c_7 j\mathfrak{h}\delta_{n,m}^2}/{8c_3\mathfrak{c}^2} \right\}\right).\nonumber
	\end{align}
	Note that
	\begin{align}
	\label{ineq:generic_pdim_6}
	& P_{f_{0,m}}^{(n)} \Pi_n\big(f \in \mathcal{F}: d_n^2(f,f_{0,m})>  \mathfrak{c}^2 (j\mathfrak{h})^\gamma \delta_{n,m}^2 \big\lvert X^{(n)}\big)(1-\phi_n)\bm{1}_{\mathcal{E}_n}\\
	& = P_{f_{0,m}}^{(n)} \bigg[\frac{\int_{ f \in \mathcal{F}: d_n^2(f,f_{0,m})>  \mathfrak{c}^2 (j\mathfrak{h})^\gamma \delta_{n,m}^2}  {p_f^{(n)}}/{p_{f_{0,m}}^{(n)} }\ \d{\Pi_n(f)} }{\int  {p_f^{(n)}}/{p_{f_{0,m} }^{(n)} }\ \d{\Pi_n(f)}}(1-\phi_n)\bm{1}_{\mathcal{E}_n}\bigg]\nonumber\\
	& \leq \frac{e^{{c_7n j\mathfrak{h}\delta_{n,m}^2}/{4\mathfrak{c}^2}}}{\lambda_{n}(m) \Pi_{n,m}(\{f \in \mathcal{F}_m: d_n^2(f,f_{0,m})\leq {c_7 j\mathfrak{h}\delta_{n,m}^2}/{8c_3\mathfrak{c}^2} \}) }\nonumber\\
	&\qquad \qquad\times  P_{f_{0,m}}^{(n)} \bigg[ \int_{ f \in \mathcal{F}: d_n^2(f,f_{0,m})> \mathfrak{c}^2 (j\mathfrak{h})^\gamma \delta_{n,m}^2 }  {p_f^{(n)}}/{p_{f_{0,m}}^{(n)} }\ \d{\Pi_n(f)} (1-\phi_n)\bigg] \nonumber\\
	& \equiv (I)\cdot (II)\nonumber
	\end{align}
	where the inequality follows from (\ref{ineq:generic_pdim_5}). On the other hand, the expectation term in the above display can be further calculated as follows: 
	\begin{align*}
	&(II) =\int_{ f \in \mathcal{F}: d_n^2(f,f_{0,m})>  \mathfrak{c}^2 (j\mathfrak{h})^\gamma \delta_{n,m}^2 } P_f^{(n)}(1-\phi_n)\ \d{\Pi_n(f)}\\
	& \leq \sup_{ f \in \mathcal{F}_{j\mathfrak{h}m}: d_n^2(f,f_{0,m})> \mathfrak{c}^2 (j\mathfrak{h})^\gamma \delta_{n,m}^2  }P_f^{(n)}(1-\phi_n) + \Pi_n\big(\mathcal{F}\setminus \mathcal{F}_{j\mathfrak{h} m}\big)\\ 
	&\leq 2c_6 e^{-(c_7/\mathfrak{c}^2)nj \mathfrak{h}\delta_{n,m}^2}+  4 e^{- (1/\mathfrak{c}^2) n j\mathfrak{h}\delta_{n,m}^2}\leq 6c_6 e^{-(c_7/\mathfrak{c}^2)nj \mathfrak{h}\delta_{n,m}^2}.
	\end{align*} 
	The first term in the second inequality follows from (\ref{ineq:generic_pdim_4}) and the second term follows from (P1) in Assumption \ref{assump:prior_mass_general} along with the left side of (\ref{cond:suplinearity_delta}).  By (P1)-(P2) in Assumption \ref{assump:prior_mass_general} and $j\geq 8\mathfrak{c}^2/c_7\mathfrak{h}$, 
	\begin{align*}
	P_{f_{0,m}}^{(n)} \Pi_n\big(f \in \mathcal{F}: d_n^2(f,f_{0,m})>  \mathfrak{c}^2 (j\mathfrak{h})^\gamma \delta_{n,m}^2 \big\lvert X^{(n)}\big)(1-\phi_n)\bm{1}_{\mathcal{E}_n}  \leq C e^{-(c_7/4\mathfrak{c}^2)nj \mathfrak{h}\delta_{n,m}^2}.
	\end{align*} 
	We conclude (\ref{ineq:overfit_oracle}) from (\ref{ineq:generic_pdim_3}), probability estimate on $\mathcal{E}_n^c$.
\end{proof}

\begin{proof}[Proof of Proposition \ref{prop:underfit_general}]
	The proof largely follows the same lines as that of Proposition \ref{prop:overfit_general}. See Appendix \ref{section:proof_lemma_main_result} for details.
\end{proof}

\subsection{Completion of proof of Theorem \ref{thm:general_ms}}

\begin{proof}[Proof of  (\ref{ineq:post_model_selection_general})]
	For any $m \in \mathcal{M}$ such that $\delta_{n,m}^2\geq d_n^2(f_0,f_{0,m})$, following the similar reasoning in (\ref{ineq:generic_pdim_6}) with $j=8\mathfrak{c}^2/c_7\mathfrak{h}$, 
	\begin{align*}
	&P_{f_{0,m}}^{(n)} \Pi_n\big(f \notin \mathcal{F}_{j\mathfrak{h}m}\lvert X^{(n)}\big)\bm{1}_{\mathcal{E}_n}\\
	&\leq \frac{e^{{c_7n j\mathfrak{h}\delta_{n,m}^2}/{4\mathfrak{c}^2}}}{\lambda_{n}(m) \Pi_{n,m}\left(\left\{f \in \mathcal{F}_m: d_n^2(f,f_{0,m})\leq {c_7 j\mathfrak{h}\delta_{n,m}^2}/{8c_3\mathfrak{c}^2} \right\}\right) }\cdot \Pi\big(\mathcal{F}\setminus \mathcal{F}_{j\mathfrak{h}m}\big)\\
	&\leq C e^{-(c_7/4\mathfrak{c}^2)nj \mathfrak{h}\delta_{n,m}^2}.
	\end{align*} 
	From here (\ref{ineq:post_model_selection_general}) can be established by controlling the probability estimate for $\mathcal{E}_n^c$ as in Proposition \ref{prop:overfit_general}, and a change of measure argument using Lemma \ref{lem:change_variable}.
\end{proof}

\subsection{Proof of Lemma \ref{lem:local_test_general}}

\begin{proof}[Proof of Lemma \ref{lem:local_test_general}]
	Without loss of generality, we assume that $d_0=0$. Let $c>0$ be a constant to be specified later. Consider the test statistics $
	\phi_n\equiv \bm{1}\big(\log ({p_{f_0}^{(n)}}/{p_{f_1}^{(n)}})\leq -c n d_n^2(f_0,f_1)\big).$
	We first consider type I error. Under the null hypothesis, we have for any $\lambda_1\in(0,1/\kappa_\Gamma)$, 
	\begin{align*}
	P_{f_0}^{(n)} \phi_n  &\leq P_{f_0}^{(n)}\big[\big(\log ({p_{f_0}^{(n)}}/{p_{f_1}^{(n)}})-P_{f_0}\log ({p_{f_0}^{(n)}}/{p_{f_1}^{(n)}})\big)\leq -(c+c_2) n d_n^2(f_0,f_1) \big]\\ 
	&\leq c_1 e^{\psi_{\kappa_g nd_n^2(f_0,f_1),\kappa_\Gamma}(-\lambda_1)}\cdot e^{-\lambda_1(c+c_2)nd_n^2(f_0,f_1)}.
	\end{align*} 
	Choosing $\lambda_1= \min\{1/(\kappa_\Gamma), (c+c_2)/(2\kappa_g)\}$ we get $
	P_{f_0}^{(n)} \phi_n \leq c_1 e^{-C_1 nd_n^2(f_0,f_1)}$
	where $C_1=\lambda_1(c+c_2)/2$. Next we handle the type II error. To this end, for a constant $c'>c_3c_5$ to be specified later, consider the event $
	\mathcal{E}_n\equiv \bm{1}\big( \log({p_f^{(n)}}/{p_{f_1}^{(n)}})<c' n d_n^2(f_0,f_1)\big)$, where $f \in \mathcal{F}$ is such that $d_n^2(f,f_1)\leq c_5 d_n^2(f_0,f_1)$, and $\lambda_2 \in (0,1/\kappa_\Gamma)$, 
	\begin{align*}
	&P_{f}^{(n)} (\mathcal{E}_n^c)  \leq P_f^{(n)} \big(\log ({p_f^{(n)}}/{p_{f_1}^{(n)}})-P_f^{(n)}\log ({p_f^{(n)}}/{p_{f_1}^{(n)}})>c'n d_n^2(f_0,f_1)-c_3nd_n^2(f,f_1)\big)\\
	& \leq  P_f^{(n)} \big(\log({p_f^{(n)}}/{p_{f_1}^{(n)}})-P_f^{(n)}\log({p_f^{(n)}}/{p_{f_1}^{(n)}})>(c'-c_3c_5)n d_n^2(f_0,f_1)\big)\\
	& \leq e^{-\lambda_2 (c'-c_3c_5)n d_n^2(f_0,f_1)}\cdot c_1 e^{\psi_{\kappa_g n d_n^2(f,f_1),\kappa_\Gamma}(\lambda_2)}.
	\end{align*} 
	By choosing $\lambda_2= \min\{1/(\kappa_\Gamma), (c'-c_3c_5)/(2\kappa_g)\}$, we see that $
	P_f^{(n)}(\mathcal{E}_n^c)\leq c_1 e^{-C_2 n d_n^2(f_0,f_1)}$
	where $C_2=\lambda_2(c'-c_3c_5)/2$. On the other hand, using the symmetry of $d_n(\cdot,\cdot)$ and for $0<c<c_2$, $\lambda_3\in (0,1/\kappa_\Gamma)$, 
	\begin{align*}
	&P_{f_1}^{(n)}\big(1-\phi_n)  = P_{f_1}^{(n)}\big(\log ({p_{f_1}^{(n)}}/{p_{f_0}^{(n)}})< c n d_n^2(f_0,f_1)\big) \\ 
	&= P_{f_1}^{(n)}\big(\log ({p_{f_1}^{(n)}}/{p_{f_0}^{(n)}})-P_{f_1}^{(n)}\log ({p_{f_1}^{(n)}}/{p_{f_0}^{(n)}})< -(c_2-c) n d_n^2(f_0,f_1)\big)\\ 
	&\leq e^{-\lambda_3 (c_2-c)nd_n^2(f_0,f_1)}\cdot c_1 e^{\psi_{\kappa_g nd_n^2(f_0,f_1),\kappa_\Gamma}(-\lambda_3)}.
	\end{align*} 
	Choosing $\lambda_3= \min\{1/(\kappa_\Gamma), (c_2-c)/(2\kappa_g)\}$ we see that $
	P_{f_1}^{(n)}(1-\phi_n)\leq c_1 e^{-C_3 n d_n^2(f_0,f_1)}$, 
	where $C_3 = \lambda_3(c_2-c)/2$. Hence it follows that 
	\begin{align*}
	P_f^{(n)}(1-\phi_n)&= P_{f_1}^{(n)}\big[\big(1-\phi_n\big)\cdot ({p_f^{(n)}}/{p_{f_1}^{(n)}})\big(\bm{1}_{\mathcal{E}_n}+\bm{1}_{\mathcal{E}_n^c}\big)\big]\\
	&\leq e^{c' n d_n^2(f_0,f_1)} P_{f_1}^{(n)}\big(1-\phi_n)+ c_1e^{-C_2 n d_n^2(f_0,f_1)}\\
	&\leq 2c_1 e^{-\min\{(C_3-c'),C_2\}nd_n^2(f_0,f_1)}. 
	\end{align*} 
	Now it suffices to choose $c,c',c_5$ such that $c'>c_3c_5$, $c<c_2$ and $c'<C_3$. To this end, we choose $c=c_2/2$, $
	c' \equiv \frac{C_3}{2}=\frac{\lambda_3(c_2-c)}{4}=\frac{\lambda_3c_2}{8}=\frac{c_2}{8\kappa_\Gamma}\wedge \frac{c_2^2}{32 \kappa_g}$, and $
	c_5 = \frac{c'}{2c_3}\wedge \frac{1}{4} = \frac{c_2}{16c_3\kappa_\Gamma}\wedge \frac{c_2^2}{64c_3 \kappa_g}\wedge \frac{1}{4}$, 
	completing the proof.
\end{proof}

\subsection{Proof of Lemma \ref{lem:change_variable}}
We recall a standard fact.
\begin{lemma}\label{lem:berstein_ineq}
	If a random variable $X$ satisfies $\E e^{\lambda X}\leq e^{\psi_{v,c}(\lambda)}$, then for $t>0$, $
	\Prob(X\geq t)\vee \Prob(X\leq -t)\leq e^{-\frac{t^2}{2(2v+ct)}}$. 
\end{lemma}
\begin{proof}
	Noting that $\E e^{\lambda X}\leq e^{\psi_{v,c}(\lambda)}\leq e^{\frac{(2v)\lambda^2}{2(1-c\abs{\lambda})}}$. Then using arguments in page 29 of \cite{boucheron2013concentration} and Exercise 2.8 therein, we obtain the claim.
\end{proof}

\begin{proof}[Proof of Lemma \ref{lem:change_variable}]
	For $c=2c_3$, consider the event $
	\mathcal{E}_n\equiv \big\{\log ({p_{f_0}^{(n)}}/{p_{f_1}^{(n)}})<  cj n\delta_n^2\big\}.$
	By Lemma \ref{lem:berstein_ineq}, we have for some constant $C>0$ depending on $c_1,c_3$ and $\kappa$, 
	\begin{align*}
	P_{f_0}^{(n)}(\mathcal{E}_n^c)&\leq P_{f_0}^{(n)}\big(\log ({p_{f_0}^{(n)}}/{p_{f_1}^{(n)}})-P_{f_0}^{(n)}\log ({p_{f_0}^{(n)}}/{p_{f_1}^{(n)}})\geq cjn\delta_n^2-c_3nd_n^2(f_0,f_1)\big)\\
	&\stackrel{(*)}{\leq} P_{f_0}^{(n)}\big(\log ({p_{f_0}^{(n)}}/{p_{f_1}^{(n)}})-P_{f_0}^{(n)}\log ({p_{f_0}^{(n)}}/{p_{f_1}^{(n)}})\geq c_3 j n \delta_n^2\big)\leq Ce^{-C^{-1}nj\delta_n^2}.
	\end{align*} 
	Here in $(\ast)$ we used $d_n(f_0,f_1)\leq \delta_n$. Then 
	\begin{align*}
	P_{f_0}^{(n)} U&=P_{f_0}^{(n)} U\bm{1}_{\mathcal{E}_n} + P_{f_0}^{(n)} U\bm{1}_{\mathcal{E}_n^c}\leq P_{f_1}^{(n)} \big[U ({p_{f_0}^{(n)}}/{p_{f_1}^{(n)}})\bm{1}_{\mathcal{E}_n}\big] + Ce^{-C^{-1}nj\delta_n^2}\\
	&\leq P_{f_1}^{(n)}U\cdot e^{cnj \delta_n^2}+Ce^{-C^{-1}nj\delta_n^2},
	\end{align*}
	completing the proof.
\end{proof}

\subsection{Proof of Proposition \ref{prop:prior_generic}}

\begin{proof}[Proof of Proposition \ref{prop:prior_generic}]
	Let $\Sigma_n=\sum_{m}e^{-2n\delta_{n,m}^2}$ be the total mass. Then $
	e^{-2n\delta_{n,1}^2}\leq \Sigma_n\leq 2e^{-2n\delta_{n,1}^2/\mathfrak{c}^2}\leq 2$. The first condition of (P1) is trivial. We only need to verify the second condition of (P1): 
	\begin{align*}
	\sum_{k>\mathfrak{h}m}\lambda_{n}(k)=\Sigma_n^{-1}\sum_{k>\mathfrak{h}m} e^{-2n\delta_{n,k}^2}\leq e^{2n\delta_{n,1}^2} \cdot 2e^{-(2\mathfrak{h}/\mathfrak{c}^2)n\delta_{n,m}^2}\leq 2e^{-2n\delta_{n,m}^2},
	\end{align*}
	where the first inequality follows from (\ref{cond:suplinearity_delta}) and the second by the condition $\mathfrak{h}\geq 2\mathfrak{c}^2$. 
\end{proof}

\section{Proofs in Section \ref{section:models} Part I: results for models}

\begin{proof}[Proof of Lemma \ref{lem:bernstein_regression}]
	Let $P_{\theta_0}^{(n)}$ denote the probability measure induced by the joint distribution of $(X_1,\ldots,X_n)$ when the underlying signal is $\theta_0$.
	
	First consider Gaussian regression case. Since 
	\begin{align*}
	\log ({p_{\theta_0}^{(n)}}/{p_{\theta_1}^{(n)}})(X^{(n)})&=\sum_{i=1}^n \bigg[-\frac{1}{2}(X_i-\theta_{0,i})^2+\frac{1}{2}(X_i-\theta_{1,i})^2\bigg],\\
	P_{\theta_0}^{(n)} \log ({p_{\theta_0}^{(n)}}/{p_{\theta_1}^{(n)}})&= \frac{1}{2} n \ell_n^2(\theta_0,\theta_1).
	\end{align*}
	we have 
	\begin{align*}
	& P_{\theta_0}^{(n)} e^{\lambda\big(\log ({p_{\theta_0}^{(n)}}/{p_{\theta_1}^{(n)}})(X^{(n)})-P_{\theta_0}^{(n)}\log ({p_{\theta_0}^{(n)}}/{p_{\theta_1}^{(n)}})\big)} \\
	& \leq P e^{\sum_{i=1}^n \epsilon_i \lambda\big(\theta_{0,i}-\theta_{1,i}\big)}\leq e^{\lambda^2 n\ell_n^2(\theta_0,\theta_1)/2}.
	\end{align*}

	Secondly consider Laplace regression. Note 
	\begin{align*}
	\log ({p_{\theta_0}^{(n)}}/{p_{\theta_1}^{(n)}})(X^{(n)})&=\sum_{i=1}^n \big[\abs{X_i-\theta_{1,i}}-\abs{X_i-\theta_{0,i}}\big],\\
	P_{\theta_0}^{(n)} \log ({p_{\theta_0}^{(n)}}/{p_{\theta_1}^{(n)}}) &= \sum_{i=1}^n \E \big[\abs{\epsilon_i+\theta_{0,i}-\theta_{1,i}}-\abs{\epsilon_i}\big].
	\end{align*}
	For any $v \in \R$, let $\varphi(v)\equiv \E \big(\abs{\epsilon+v}-\abs{\epsilon}\big)$. Clearly $\varphi$ is twice differentiable with a strictly positive second derivative on compacta. Since $\abs{\theta_{0,i}-\theta_{1,i}}\leq M$, this implies that there exists some $C_M>1$ such that $C_M^{-1} \abs{\theta_{0,i}-\theta_{1,i}}^2\leq \varphi(\theta_{0,i}-\theta_{1,i})\leq C_M\abs{\theta_{0,i}-\theta_{1,i}}^2$. Hence $
	P_{\theta_0}^{(n)} \log ({p_{\theta_0}^{(n)}}/{p_{\theta_1}^{(n)}}) \asymp_M n \ell_n^2(\theta_0,\theta_1)$.
	To verify the local Gaussianity condition, note that 
	\begin{align*}
	\abs{Z_i}\equiv \abs{ \big[\abs{\epsilon_i+\theta_{0,i}-\theta_{1,i}}-\abs{\epsilon_i}\big]- \E \big[\abs{\epsilon_i+\theta_{0,i}-\theta_{1,i}}-\abs{\epsilon_i}\big]  }\leq 2\abs{\theta_{0,i}-\theta_{1,i}},
	\end{align*} 
	so it follows from the Hoffmann-Jorgensen inequality (cf. Proposition A.1.6 of \cite{van1996weak}) that 
	\begin{align*}
	&\biggpnorm{\log \log ({p_{\theta_0}^{(n)}}/{p_{\theta_1}^{(n)}})-P_{\theta_0}^{(n)}\log ({p_{\theta_0}^{(n)}}/{p_{\theta_1}^{(n)}})}{\psi_2} \\
	&= \biggpnorm{\sum_{i=1}^n Z_i }{\psi_2} \lesssim \biggpnorm{\sum_{i=1}^n Z_i}{1}+\bigg(\sum_{i=1}^n \pnorm{Z_i}{\psi_2}^2\bigg)^{1/2}\\
	& \lesssim \bigg(\sum_{i=1}^n Z_i^2\bigg)^{1/2}+\left(\sum_{i=1}^n \pnorm{Z_i}{\psi_2}^2\right)^{1/2}  \lesssim \big(n \ell_n^2(\theta_0,\theta_1)\big)^{1/2},
	\end{align*} 
	where $\pnorm{\cdot}{\psi}$ denotes the usual Orlicz norm given by $\pnorm{X}{\psi}\equiv \inf \{C>0: \E \psi(\abs{X}/C)\leq 1\}$, and $\psi_2(x)= e^{x^2}-1$. Hence 
	\begin{align*}
	P_{\theta_0}^{(n)} e^{\lambda(\log ({p_{\theta_0}^{(n)}}/{p_{\theta_1}^{(n)}})-P_{\theta_0}^{(n)}\log ({p_{\theta_0}^{(n)}}/{p_{\theta_1}^{(n)}}) } \leq e^{C\lambda^2 n\ell_n^2(\theta_0,\theta_1)},
	\end{align*}
	where $C>0$ is an absolute constant.

	Next consider binary regression. Note 
	\begin{align*}
	\log ({p_{\theta_0}^{(n)}}/{p_{\theta_1}^{(n)}})(X^{(n)}) &=\sum_{i=1}^n X_i \log \frac{\theta_{0,i}}{\theta_{1,i}}+(1-X_i) \log \frac{1-\theta_{0,i}}{1-\theta_{1,i}},\\
	P_{\theta_0^{(n)}} \log ({p_{\theta_0}^{(n)}}/{p_{\theta_1}^{(n)}})  &= \sum_{i=1}^n \theta_{0,i} \log \frac{\theta_{0,i}}{\theta_{1,i}}+(1-\theta_{0,i}) \log \frac{1-\theta_{0,i}}{1-\theta_{1,i}}.
	\end{align*}
	Using the inequality $cx\leq \log(1+x)\leq x$ for all $-1<x\leq c'$ for some $c>0$ depending on $c'>-1$ only, we have shown $P_{\theta_0^{(n)}} \log ({p_{\theta_0}^{(n)}}/{p_{\theta_1}^{(n)}})\asymp n\ell_n^2(\theta_0,\theta_1)$ under the assumed condition that $\Theta_n\subset [\eta,1-\eta]^n$. Now we verify the local Gaussianity condition: 
	\begin{align*}
	P_{\theta_0}^{(n)} e^{\lambda \big(\log ({p_{\theta_0}^{(n)}}/{p_{\theta_1}^{(n)}})-P_{\theta_0^{(n)}} \log ({p_{\theta_0}^{(n)}}/{p_{\theta_1}^{(n)}}) \big)} = P_{\theta_0}^{(n)} e^{\lambda \sum_{i=1}^n (X_i-\theta_{0,i})t_i}\leq e^{\lambda^2 \sum_{i=1}^n t_i^2/8}
	\end{align*}
	where $t_i\equiv t_i(\theta_0,\theta_1) =\log \big(\frac{\theta_{0,i}}{1-\theta_{0,i}}\cdot \frac{1-\theta_{1,i}}{\theta_{1,i}}\big)$ and the last inequality follows from Hoeffding's inequality (cf. Section 2.6 of \cite{boucheron2013concentration}). The claim follows by noting that $
	t_i^2 = \big[\log \big(\frac{\theta_{0,i}-\theta_{1,i}}{(1-\theta_{0,i})\theta_{1,i}}+1\big)\big]^2\asymp (\theta_{0,i}-\theta_{1,i})^2$
	by the assumed condition and the aforementioned inequality $\log(1+x)\asymp x$ in a constrained range.
	
	Finally consider Poisson regression. 
	It is easy to see that 
	\begin{align*}
	\log ({p_{\theta_0}^{(n)}}/{p_{\theta_1}^{(n)}})(X^{(n)}) &=\sum_{i=1}^n X_i \log \frac{\theta_{0,i}}{\theta_{1,i}} + (\theta_{1,i}-\theta_{0,i}),\\
	P_{\theta_0}^{(n)} \log ({p_{\theta_0}^{(n)}}/{p_{\theta_1}^{(n)}}) &=\sum_{i=1}^n \theta_{0,i} \log \frac{\theta_{0,i}}{\theta_{1,i}} + (\theta_{1,i}-\theta_{0,i}).
	\end{align*} 
	Note that for any $1/M\leq p,q\leq M$, 
	\begin{align*}
	p\log\frac{p}{q}-(p-q)=p\big(-\log\frac{q}{p}-1+\frac{q}{p}\big)\asymp p\cdot (\frac{q}{p}-1)^2\asymp (p-q)^2,
	\end{align*} 
	where in the middle we used the fact that $-\log x-1+x \asymp (x-1)^2$ for $x$ bounded away from $0$ and $\infty$. This shows that $P_{\theta_0}^{(n)} \log ({p_{\theta_0}^{(n)}}/{p_{\theta_1}^{(n)}})\asymp n\ell_n^2(\theta_0,\theta_1)$. Hence 
	\begin{align*}
	P_{\theta_0}^{(n)} e^{\lambda \big(\log ({p_{\theta_0}^{(n)}}/{p_{\theta_1}^{(n)}})-P_{\theta_0^{(n)}} \log ({p_{\theta_0}^{(n)}}/{p_{\theta_1}^{(n)}}) \big)}\leq P_{\theta_0}^{(n)} e^{\lambda \sum_{i=1}^n (X_i-\theta_{0,i})t_i}\leq e^{\sum_{i=1}^n \theta_{0,i}(e^{\lambda t_i}-1-\lambda t_i)},
	\end{align*} 
	where $t_i = \log(\theta_{0,i}/\theta_{1,i})$. Now for any $\abs{\lambda}\leq 1$, we have $e^{\lambda t_i}-1-\lambda t_i\asymp \lambda^2 t_i^2$. On the other hand, $\theta_{0,i}t_i^2 =\theta_{0,i}\left(\log (\theta_{0,i}/\theta_{1,i})\right)^2\asymp (\theta_{0,i}-\theta_{1,i})^2$, completing the proof.
\end{proof}

\begin{proof}[Proof of Corollary \ref{cor:regression_model}]
	The claim follows from Lemma \ref{lem:bernstein_regression} and Theorem \ref{thm:general_ms}.
\end{proof}

\begin{proof}[Proof of Lemma \ref{lem:bernstein_density_estimation}]
	Since the log-likelihood ratio for $X_1,\ldots,X_n$ can be decomposed into sums of the log-likelihood ratio for single samples, and the log-likelihood ratio is uniformly bounded over $\mathcal{F}$ (since $\mathcal{G}$ is bounded), classical Bernstein inequality applies to see that for any couple $(f_0,f_1)$, the local Gaussianity condition in Assumption \ref{assump:laplace_cond_kl} holds with $v=\kappa_g n \mathrm{Var}_{f_0}(\log f_0/f_1), c=\kappa_\Gamma$ where $\kappa_g,\kappa_\Gamma$ depend only on $\mathcal{G}$. Hence we only need to verify that $
	\mathrm{Var}_{f_0}(\log f_0/f_1)\lesssim h^2(f_0,f_1)$ and $ P_{f_0}(\log f_0/f_1)\asymp h^2(f_0,f_1)$. 
	This can be seen by Lemma 8 of \cite{ghosal2007posterior} and the fact that Hellinger metric is dominated by the Kullback-Leiber divergence.
\end{proof}

\begin{proof}[Proof of Corollary \ref{cor:density_est}]
	The claim follows from Lemma  \ref{lem:bernstein_density_estimation} and Theorem \ref{thm:general_ms}.
\end{proof}

\begin{lemma}\label{lem:log_exp_bound}
	Let $Z\geq 0$ be a non-negative random variable bounded by $M>0$. Then $\E \exp(Z)\leq \exp(e^M \E Z)$.
\end{lemma}
\begin{proof}
	Note that $
	\log \E \exp(Z)=\log (\E[\exp(Z)-1]+1)\leq \E[\exp(Z)-1]\leq e^M \E Z$, 
	where the last inequality follows from Taylor expansion $e^{x}-1=\sum_{k=1}^n x^k/k!\leq x\sum_{k\geq 1} M^{k-1}/k!\leq x e^M$ for $x\geq 0$.
\end{proof}

\begin{proof}[Proof of Lemma \ref{lem:bernstein_gaussian_autoregression}]
	We omit explicit dependence of $M$ on the notation $d_{r,M}$ and $r_M$ in the proof. Let $P_{f_0}^{(n)}$ denote the probability measure induced by the joint distribution of $(X_0,\ldots,X_n)$ where $X_0$ is distributed according to the stationary density $q_{f_0}$. Easy computation shows that 
	\begin{align*}
	\log ({p_{f_0}^{(n)}}/{p_{f_1}^{(n)}}) &= \sum_{i=0}^{n-1} \bigg[\epsilon_{i+1} (f_0(X_i)-f_1(X_i)) +\frac{1}{2}(f_0(X_i)-f_1(X_i))^2\bigg],\\
	P_{f_0}^{(n)} \log ({p_{f_0}^{(n)}}/{p_{f_1}^{(n)}}) &= \frac{n}{2}\int (f_0-f_1)^2 q_{f_0}\ \d{\lambda}.
	\end{align*} 
	Here $\lambda$ denotes the Lebesgue measure on $\R$. By the arguments on page 209 of \cite{ghosal2007convergence}, we see that $r\lesssim q_{f_0}\lesssim r$. Hence we only need to verify the local Gaussianity condition. By Cauchy-Schwarz,
	\begin{align}\label{ineq:berstein_autoregression_0}
	\big[P_{f_0}^{(n)}  e^{\lambda \log ({p_{f_0}^{(n)}}/{p_{f_1}^{(n)}})}\big]^2&\leq P_{f_0}^{(n)} e^{2\lambda \sum_{i=0}^{n-1}\epsilon_{i+1}(f_0(X_i)-f_1(X_i))}\\
	&\quad \times P_{f_0}^{(n)} e^{\lambda \sum_{i=0}^{n-1}(f_0(X_i)-f_1(X_i))^2}\equiv (I)\times (II).\nonumber
	\end{align}
	The first term $(I)$ can be handled by an inductive calculation. First note that for any $\abs{\mu}\leq 2$ and $X_1 \in \R$, 
	\begin{align}\label{ineq:berstein_autoregression_1}
	P_{p(\cdot|X_1)} e^{\mu^2 (f_0(X_2)-f_1(X_2)^2 }
	&\leq e^{e^{16M^2} \mu^2 P_{p(\cdot|X_1)} (f_0-f_1)(X_2)^2 }\leq e^{C_M\mu^2 d_r^2(f_0,f_1)}
	\end{align}
	where the first inequality follows from Lemma \ref{lem:log_exp_bound} and the second inequality follows from $r(\cdot)\lesssim p_f(\cdot|x)\lesssim r(\cdot)$ holds for all $x \in \R$ where the constant involved depends only on $M$. Let $S_n\equiv \sum_{i=0}^{n-1} \epsilon_{i+1}(f_0(X_i)-f_1(X_i))$ and $\bm{\epsilon}_n\equiv (\epsilon_1,\ldots,\epsilon_n)$. Then for $\abs{\lambda}\leq 1$, let $\mu\equiv 2\lambda$, 
	\begin{align*}
	&P_{f_0}^{(n)} e^{2\lambda S_n}=P_{f_0}^{(n)} e^{\mu S_n}= \E_{X_0, \bm{\epsilon}_{n-1}}\big[e^{\mu S_{n-1}} \E_{\epsilon_n}e^{ \mu \epsilon_n(f_0(X_{n-1})-f_1(X_{n-1}))}\big]\\
	&\leq \E_{X_0, \bm{\epsilon}_{n-1}}\big[e^{\mu S_{n-1}} e^{\mu^2(f_0(X_{n-1})-f_1(X_{n-1}) )^2/2}\big]\\
	&\leq \E_{X_0,\bm{\epsilon}_{n-2}} [e^{\mu S_{n-2}}\E_{\epsilon_{n-1}}e^{\mu \epsilon_{n-1}(f_0(X_{n-2})-f_1(X_{n-2}))+\mu^2(f_0(X_{n-1})-f_1(X_{n-1}) )^2/2} ]\\ 
	&\leq \E_{X_0,\bm{\epsilon}_{n-2}} \big[e^{\mu S_{n-2}}(\E_{\epsilon_{n-1}}e^{2\mu \epsilon_{n-1}(f_0(X_{n-2})-f_1(X_{n-2}))} )^{1/2}\\ 
	&\qquad\qquad \times (\E_{p(\cdot|X_{n-2})} e^{\mu^2(f_0(X_{n-1})-f_1(X_{n-1}) )^2})^{1/2}\big]\\
	&\leq \E_{X_0,\bm{\epsilon}_{n-2}} \big[e^{\mu S_{n-2}}e^{\mu^2 (f_0(X_{n-2})-f_1(X_{n-2}))^2} \big]\cdot e^{C_M \mu^2 d_r^2(f_0,f_1)/2}
	\end{align*} 
	where the last inequality follows from (\ref{ineq:berstein_autoregression_1}). Now we can iterate the above calculation to see that $
	(I)\leq e^{C_M \lambda^2 n d_r^2(f_0,f_1)}$.
	Next we consider $(II)$. Since for any non-negative random variables $Z_1,\ldots,Z_n$, we have $\E \prod_{i=1}^n Z_i \leq \prod_{i=1}^n (\E Z_i^n)^{1/n}$. So 
	\begin{align*}
	(II)\leq \prod_{i=1}^n (P_{f_0}^{(n)}e^{n\lambda (f_0(X_i)-f_1(X_i))^2} )^{1/n}=P_{q_{f_0}} e^{n\lambda (f_0(X_0)-f_1(X_0))^2},
	\end{align*}
	where the last inequality follows by stationarity. On the other hand, by Jensen's inequality,
	\begin{align*}
	e^{-\lambda P_{f_0}^{(n)} \log ({p_{f_0}^{(n)}}/{p_{f_1}^{(n)}}) }\leq e^{-\frac{\lambda n}{2}P_{q_{f_0}} (f_0-f_1)^2}\leq P_{q_{f_0}} e^{-\lambda n (f_0(X_0)-f_1(X_0))^2/2}.
	\end{align*} 
	Collecting the above estimates, we see that for $\abs{\lambda}\leq 1$, 
	\begin{align*}
	&P_{f_0^{(n)}}  e^{\lambda \log ({p_{f_0}^{(n)}}/{p_{f_1}^{(n)}}) -P_{f_0^{(n)}}\log ({p_{f_0}^{(n)}}/{p_{f_1}^{(n)}}) }\\
	&\leq \sqrt{(I)\cdot (II)}e^{-\lambda P_{f_0}^{(n)} \log ({p_{f_0}^{(n)}}/{p_{f_1}^{(n)}})}
	\leq e^{C_M'\lambda^2 n d_r^2(f_0,f_1)},
	\end{align*} 
	completing the proof.
\end{proof}

\begin{proof}[Proof of Corollary \ref{cor:rate_auto}]
	The claim follows from Lemma \ref{lem:bernstein_gaussian_autoregression} and Theorem \ref{thm:general_ms}.
\end{proof}

\begin{proof}[Proof of Lemma \ref{lem:bernstein_gaussian_time_series}]
	For any $g \in \mathcal{G}$, let $p_g^{(n)}$ denote the probability density function of a $n$-dimensional multivariate normal distribution with covariance matrix $\Sigma_g\equiv T_n(f_g)$, and $P_g^{(n)}$ the expectation taken with respect to the density $p_g^{(n)}$. Then for any $g_0, g_1 \in \mathcal{G}$,
	\begin{align}\label{ineq:bernstein_spectral_density_0}
	\log \frac{p_{g_0}^{(n)}}{p_{g_1}^{(n)}}(X^{(n)})& = -\frac{1}{2} (X^{(n)})^\top (\Sigma_{g_0}^{-1}-\Sigma_{g_1}^{-1})X^{(n)}-\frac{1}{2}\log\det(\Sigma_{g_0}\Sigma_{g_1}^{-1}),\\
	P_{g_0}^{(n)}\log \frac{p_{g_0}^{(n)}}{p_{g_1}^{(n)}}& = -\frac{1}{2} \trace(I-\Sigma_{g_0}\Sigma_{g_1}^{-1})-\frac{1}{2}\log\det(\Sigma_{g_0}\Sigma_{g_1}^{-1})\nonumber
	\end{align}
	where we used the fact that for a random vector $X$ with covariance matrix $\Sigma$, $\E X^\top A X=\trace(\Sigma A)$. Let $G\equiv \Sigma_{g_0}^{-1/2}X^{(n)}\sim \mathcal{N}(0,I)$ under $P_{g_0}^{(n)}$, and $B\equiv I-\Sigma_{g_0}^{1/2}\Sigma_{g_1}^{-1}\Sigma_{g_0}^{1/2}$, then 
	\begin{align*}
	Y_n&\equiv \log ({p_{g_0}^{(n)}}/{p_{g_1}^{(n)}})(X^{(n)})-P_{g_0}^{(n)}\log ({p_{g_0}^{(n)}}/{p_{g_1}^{(n)}})\\
	&=-\frac{1}{2} [ (X^{(n)})^\top (\Sigma_{g_0}^{-1}-\Sigma_{g_1}^{-1})X^{(n)}-\trace(I-\Sigma_{g_0}\Sigma_{g_1}^{-1})]=-\frac{1}{2}[G^\top B G-\trace (B)].
	\end{align*} 
	Let $B=U^\top \Lambda U$ be the spectral decomposition of $B$ where $U$ is orthonormal and $\Lambda=\mathrm{diag}(\lambda_1,\ldots,\lambda_n)$ is a  diagonal matrix. Then we can further compute 
	\begin{align*}
	-2Y_n =_d G^\top \Lambda G-\trace(\Lambda) =\sum_{i=1}^n \lambda_i (g_i^2-1),
	\end{align*}
	where $g_1,\ldots,g_n$'s are i.i.d. standard normal. Note that for any $\abs{t}<1/2$, 
	\begin{align*}
	\frac{1}{\sqrt{2\pi}}\int_{-\infty}^{\infty} e^{t (x^2-1)} e^{-x^2/2}\ \d{x} = \frac{e^{-t}}{\sqrt{1-2t}}=e^{\frac{1}{2}(-\log(1-2t)-2t)}\leq e^{t^2/(1-2t)},
	\end{align*}
	where the inequality follows from
	\begin{align*}
	-\log(1-2t)-2t = \sum_{k\geq 2} \frac{1}{k}(2t)^k= 4t^2 \sum_{k\geq 0}\frac{1}{k+2} (2t)^k\leq \frac{2t^2}{1-2t}.
	\end{align*}
	With $t=-\lambda \lambda_i/2$, we have that for any $\abs{\lambda}<1/\max_i \lambda_i$,
	\begin{align*}
	\E e^{\lambda Y_n}&=\prod_{i=1}^n \E e^{-\lambda \cdot \lambda_i(g_i^2-1)/2}=\prod_{i=1}^n \frac{1}{\sqrt{2\pi}}\int_{-\infty}^{\infty} e^{-\lambda\cdot \lambda_i (x^2-1)/2} e^{-x^2/2}\ \d{x}\\
	&\leq \prod_{i=1}^n e^{\frac{\lambda^2 \lambda_i^2}{4+4\lambda \lambda_i}}\leq \exp\bigg(\frac{\lambda^2 \sum_i \lambda_i^2}{4-4\abs{\lambda} \max_i \abs{\lambda_i}}\bigg).
	\end{align*} 
	Denote $\pnorm{\cdot}{}$ and $\pnorm{\cdot}{F}$ the matrix operator norm and Frobenius norm respectively. By the arguments on page 203 of \cite{ghosal2007convergence}, we have $\pnorm{\Sigma_g}{}\leq 2\pi\pnorm{e^g}{\infty}$ and $\pnorm{\Sigma_g^{-1}}{}\leq (2\pi)^{-1}\pnorm{e^{-g}}{\infty}$. Since $\mathcal{G}$ is a class of uniformly bounded function classes, the spectrum of the covariance matrices $\Sigma_g$ and their inverses running over $g$ must be bounded. Hence 
	\begin{align*}
	\max_i \abs{\lambda_i}=\pnorm{B}{}=\pnorm{(\Sigma_{g_1}-\Sigma_{g_0})\Sigma_{g_1}^{-1}}{}\leq \pnorm{\Sigma_{g_1}-\Sigma_{g_0}}{}\pnorm{\Sigma_{g_1}^{-1}}{}\leq C_{\mathcal{G}}<\infty.
	\end{align*} 
	Next, note that 
	\begin{align*}
	\bigg(\sum_i\lambda_i^2\bigg)^{1/2} &=(\trace(BB^\top))^{1/2}=\pnorm{B}{F}=\pnorm{(\Sigma_{g_1}-\Sigma_{g_0})\Sigma_{g_1}^{-1}}{F}\\
	&\leq \pnorm{\Sigma_{g_1}^{-1}}{}\pnorm{\Sigma_{g_1}-\Sigma_{g_0}}{F}\leq C'_{\mathcal{G}}\sqrt{n D_n^2(g_0,g_1)},
	\end{align*} 
	where in the first inequality we used $\pnorm{MN}{F}=\pnorm{NM}{F}$ for symmetric matrices $M,N$ and the general rule $\pnorm{PQ}{F}\leq \pnorm{P}{}\pnorm{Q}{F}$. Collecting the above estimates we see that Assumption \ref{assump:laplace_cond_kl} is satisfied for $v=\kappa_gnD_n^2(g_0,g_1)$ and $c=\kappa_\Gamma$ for constants $\kappa_g,\kappa_\Gamma$ depending on $\mathcal{G}$ only.
	
	Finally we relate $n^{-1} P_{g_0}^{(n)}\log ({p_{g_0}^{(n)}}/{p_{g_1}^{(n)}})$ and $D_n^2(g_0,g_1)$. First by (\ref{ineq:bernstein_spectral_density_0}), we have 
	\begin{align*}
	&P_{g_0}^{(n)}\log ({p_{g_0}^{(n)}}/{p_{g_1}^{(n)}})= -\frac{1}{2} \trace(I-\Sigma_{g_0}\Sigma_{g_1}^{-1})-\frac{1}{2}\log\det(\Sigma_{g_0}\Sigma_{g_1}^{-1}) \\
	&= \frac{1}{2}( \trace(\Sigma_{g_1}^{-1/2}(\Sigma_{g_0}-\Sigma_{g_1})\Sigma_{g_1}^{-1/2})-\log \det(I+\Sigma_{g_1}^{-1/2}(\Sigma_{g_0}-\Sigma_{g_1})\Sigma_{g_1}^{-1/2}))\\
	&\leq \frac{1}{4} \pnorm{I-\Sigma_{g_0}\Sigma_{g_1}^{-1}}{F}^2\leq \frac{1}{4}\pnorm{\Sigma_{g_1}-\Sigma_{g_0}}{F}^2\pnorm{\Sigma_{g_1}^{-1}}{}^2\leq C''_{\mathcal{G}} nD_n^2(g_0,g_1). 
	\end{align*} 
	Here in the second line we used the fact that $\det(AB^{-1})=\det(I+B^{-1/2}(A-B) B^{-1/2})$, and in the third line  we used the fact $-\log\det (I+A)+\trace(A)\leq \frac{1}{2}\trace(A^2)$ for any p.s.d. matrix $A$, due to the inequality $\log(1+x)-x\geq -\frac{1}{2}x^2$ for all $x\geq 0$. On the other hand, by using the reversed inequality $\log(1+x)-x\leq -cx^2$ for all $0\leq x\leq c'$ where $c$ is a constant depending only on $c'$, we can establish $
	P_{g_0}^{(n)}\log ({p_{g_0}^{(n)}}/{p_{g_1}^{(n)}})\geq C'''_{\mathcal{G}} n D_n^2(g_0,g_1)$, thereby
	completing the proof.
\end{proof}

\begin{proof}[Proof of Corollary \ref{cor:gaussian_time_series}]
	The claim follows from Lemma \ref{lem:bernstein_gaussian_time_series} and Theorem \ref{thm:general_ms}.
\end{proof}

\begin{proof}[Proof of Lemma \ref{lem:bernstein_cov_est}]
	Note that 
	\begin{align*}
	\log ({p_{\Sigma_0}^{(n)}}/{p_{\Sigma_1}^{(n)}})(X^{(n)})& = -\sum_{i=1}^n\bigg[\frac{1}{2} X_i^\top (\Sigma_{0}^{-1}-\Sigma_{1}^{-1})X_i-\frac{1}{2}\log\det(\Sigma_{0}\Sigma_{1}^{-1})\bigg],\\
	P_{\Sigma_0}^{(n)}\log ({p_{\Sigma_0}^{(n)}}/{p_{\Sigma_1}^{(n)}}) &= -\frac{n}{2} \trace(I-\Sigma_{0}\Sigma_{1}^{-1})-\frac{n}{2}\log\det(\Sigma_{0}\Sigma_{1}^{-1}).
	\end{align*} 
	The rest of the proof proceeds along the same line as in Lemma \ref{lem:bernstein_gaussian_time_series}. 
\end{proof}

\begin{proof}[Proof of Corollary \ref{cor:rate_cov_general}]
	The claim follows from Lemma \ref{lem:bernstein_cov_est} and Theorem \ref{thm:general_ms}.
\end{proof}

\begin{proof}[Proof of Lemma \ref{lem:bernstein_image}]
	Note that 
	\begin{align*}
	\log ({p_{\theta_0}^{(n)}}/{p_{\theta_1}^{(n)}})(X^{(n)},Y^{(n)}) &= \sum_{X_i \in \Gamma_0\cap \Gamma_1} \log \frac{f(Y_i;\xi_0)}{f(Y_i;\xi_1)} + \sum_{X_i \in \Gamma_0^c\cap \Gamma_1^c} \log \frac{f(Y_i;\rho_0)}{f(Y_i;\rho_1)}\\
	&\qquad\qquad + \sum_{X_i \in \Gamma_0\cap \Gamma_1^c} \log \frac{f(Y_i;\xi_0)}{f(Y_i;\rho_1)} +\sum_{X_i \in \Gamma_0^c\cap \Gamma_1} \log \frac{f(Y_i;\rho_0)}{f(Y_i;\xi_1)}.
	\end{align*}
	Then we may verify Assumption \ref{assump:laplace_cond_kl} along the lines in the proof of Lemma \ref{lem:bernstein_regression}, by considering each of the terms above by virtue of independence of $X_i$'s.
\end{proof}

\begin{proof}[Proof of Lemma \ref{lem:lower_bound_metric_image}]
	Let $r>0$ be such that $d_n^2(\theta_0,\theta_1)=r^2$. By definition of $d_n$, we have $\pnorm{\rho_0-\rho_1}{2}^2\leq r^2/\lambda(\Gamma_0^c\cap \Gamma_1^c)\leq r^2/\lambda_0^2$. This implies that 
	\begin{align*}
	\pnorm{\xi_0-\rho_1}{2}\geq \pnorm{\xi_0-\rho_0}{2}-\pnorm{\rho_0-\rho_1}{2}\geq r_0 - \frac{r}{\lambda_0}\geq \frac{r_0}{2}
	\end{align*}
	under the condition $r^2\leq \lambda_0^2r_0^2/4$. Hence $
	\lambda(\Gamma_0\cap \Gamma_1^c)\leq \frac{r^2}{\pnorm{\xi_0-\rho_1}{2}^2}\leq \frac{4r^2}{r_0^2}$,
	implying
	\begin{align*}
	\lambda(\Gamma_0\cap \Gamma_1) = \lambda(\Gamma_0)-\lambda(\Gamma_0\cap \Gamma_1^c)\geq \lambda(\Gamma_0) -  \frac{4r^2}{r_0^2}\geq \frac{\lambda(\Gamma_0)}{2}
	\end{align*}
	under the condition $r^2\leq \lambda(\Gamma_0) r_0^2/8$. This further implies that $
	\pnorm{\xi_0-\xi_1}{2}^2 \leq \frac{r^2}{\lambda(\Gamma_0\cap \Gamma_1)}\leq \frac{2r^2}{\lambda(\Gamma_0)}$, 
	whence
	\begin{align*}
	\pnorm{\rho_0-\xi_1}{2}\geq \pnorm{\rho_0-\xi_0}{2}-\pnorm{\xi_0-\xi_1}{2}\geq r_0 - \sqrt{2/\lambda(\Gamma_0)} r\geq \frac{r_0}{2}
	\end{align*}
	under the condition $r^2\leq \lambda(\Gamma_0) r_0^2/8$. Hence 
	\begin{align*}
	\lambda(\Gamma_0^c\cap \Gamma_1)\leq \frac{r^2}{\pnorm{\rho_0-\xi_1}{2}^2}\leq \frac{4r^2}{r_0^2}.
	\end{align*} 
	The claim follows by noting that $\lambda(\Gamma_0\Delta \Gamma_1) = \lambda(\Gamma_0^c\cap \Gamma_1)+\lambda(\Gamma_0\cap \Gamma_1^c) $.
\end{proof}

\begin{proof}[Proof of Corollary \ref{cor:rate_image}]
	By Lemma \ref{lem:bernstein_image} and Theorem \ref{thm:general_ms}, the claim of the corollary holds for $d_n$. Using Lemma \ref{lem:lower_bound_metric_image}, for $n$ large, we may replace $d_n$ with $\lambda(\cdot \Delta\cdot)$. 
\end{proof}

\begin{proof}[Proof of Corollary \ref{cor:rate_support_boundary}]
	The main modification of the proof lies in part of Lemma \ref{lem:local_test_general}. The modified Lemma \ref{lem:local_test_general} takes the following form: fix $f_0\leq f_1 $, there exists some test $\phi_n$ such that 
	\begin{align*}
	\sup_{f \geq f_1: \bar{L}_1(f,f_1)\leq c_5^2\bar{L}_1(f_1,f_0)}(P_{f_0}^{(n)} \phi_n+P_f^{(n)}(1-\phi_n))\leq c_6 e^{-c_7 n\bar{L}_1(f_1,f_0)},
	\end{align*} 
	where $c_5\leq 1/4,c_6\in [2,\infty),c_7\in (0,1)$ are absolute constants.

	In particular, the test $\phi_n$ is constructed in the `same way' as in the proof of Lemma \ref{lem:local_test_general} with a modified way of writing: 
	\begin{align*}
	\phi_n \equiv \bm{1}\big(\log ({\d{P_{f_1}^{(n)}}}/{\d{P_{f_0}^{(n)}}})\geq c n \bar{L}_1(f_1,f_0) \big).
	\end{align*} 
	Now for type I error,  
	\begin{align*}
	P_{f_0}^{(n)}\phi_n= P_{f_0}^{(n)} (\forall i: f_1(X_i)\leq Y_i )
	= P_{f_0}^{(n)} (N(\{(x,y): y\leq f_1(x)\})=0)= e^{-n \bar{L}_1(f_1,f_0)}.
	\end{align*} 
	Here the last equality follows as 
	\begin{align*}
	\int_{(x,y): y\leq f_1(x)} \lambda_{f_0}(x,y)\ \d{x}\d{y}= \int_0^1 \d{x} \int_{-\infty}^{f_1(x)} n\bm{1}_{f_0(x)\leq y}\ \d{y} = n \int_{0}^1 \big(f_1(x)-f_0(x)\big)\ \d{x}.
	\end{align*} 
	For type II error, note that as soon as $f\geq f_1$, 
	\begin{align*}
	&P_{f}^{(n)}(1-\phi_n) = P_f^{(n)} \big(\log ({\d{P_{f_1}^{(n)}}}/{\d{P_{f_0}^{(n)}}})<cn \bar{L}_1(f_1,f_0)\big) \\
	&= P_f^{(n)} (\bm{1}_{\forall i: f_1(X_i)\leq Y_i}<e^{-(1-c)n \bar{L}_1(f_1,f_0)}) = P_f^{(n)} (\exists i: f_1(X_i)>Y_i)=0.
	\end{align*} 
	This proves the modified version of Lemma \ref{lem:local_test_general} in the current setting. Then in the proof of Lemma \ref{lem:global_test_general}, the entropy condition needs to be replaced by the entropy with left bracketing, due to the reasoning towards the last display in the proof of Lemma \ref{lem:global_test_general}. Now in the proof of Proposition \ref{prop:overfit_general}, we apply Lemma \ref{lem:global_test_general} with the set restricted to $f\geq f_0$. The set in the control of denominator in (\ref{ineq:generic_pdim_5}) can be restricted to $f\geq f_{0,m}$. The rest of the proofs carry over exactly so we omit the details.
\end{proof}

\begin{proof}[Proof of Corollary \ref{cor:rate_support_boundary_L1}]
	The proof is a combination of the change of measure idea in the current paper combined with the results in \cite{reiss2017nonparametric}. Let $m \in \mathcal{M}$ be such that $\delta_{n,m}^2\geq L_1(f_0,f_{0,m})$. Note that condition (ii) entails that 
	\begin{align*}
	&\Pi_{n}(f \in \mathcal{F}:L_1(f,f_{0,m})\leq C_2 \delta_{n,m}^2, f\leq f_{0,m} )\\
	&\geq \lambda_n(m) \Pi_{n,m}(f\in \mathcal{F}_m:L_1(f,f_{0,m})\leq C_2 \delta_{n,m}^2, f\leq f_{0,m} ) \geq e^{-C_2' n\delta_{n,m}^2}
	\end{align*}
	Then use Theorem 2.3 of \cite{reiss2017nonparametric}, we conclude that
	\begin{align*}
	P_{f_{0,m}}^{(n)}\Pi_n(f: L_1(f,f_{0,m})\geq C_3 K\delta_{n,m}^2|(X^{(n)},Y^{(n)}))\leq C_3e^{-nK\delta_{n,m}^2/C_3},
	\end{align*}
	where $K>0$ is a constant to be chosen later. Hence
	\begin{align*}
	&P_{f_0}^{(n)} \Pi_n(f: L_1(f,f_0)\geq L_1(f_0,f_{0,m})+ C_3K\delta_{n,m}^2|(X^{(n)},Y^{(n)}))\\
	&\leq P_{f_0}^{(n)} \Pi_n(f: L_1(f,f_{0,m})\geq C_3K\delta_{n,m}^2|(X^{(n)},Y^{(n)})) \\ 
	&= P_{f_{0,m}}^{(n)} \Pi_n(f: L_1(f,f_{0,m})\geq C_3K\delta_{n,m}^2|(X^{(n)},Y^{(n)})) \big(\d{P_{f_0}^{(n)}}/ \d{P_{f_{0,m}}^{(n)}} \big)\\ 
	&\leq C_3 e^{-nK\delta_{n,m}^2/C_3+nL_1(f_0,f_{0,m})} \leq C_3 e^{-n\delta_{n,m}^2},
	\end{align*}
	by choosing $K=2C_3$. We may similar consider $m \in \mathcal{M}$ such that $\delta_{n,m}^2<L_1(f_0,f_{0,m})$.
\end{proof}

\section{Proofs in Section \ref{section:models} Part II: results for applications}\label{section:proof_remaining_application}

\subsection{Proof of Theorem \ref{thm:rate_tracereg}}

\begin{lemma}\label{lem:local_ent_tracereg}
	Let $r \in \mathcal{I}$. Suppose that the linear map $\mathcal{X}:\R^{m_1\times m_2}\to \R^n$ is uniform RIP$(\bm{\nu};\mathcal{I})$. Then for any $\epsilon>0$ and $A_0 \in \R^{m_1\times m_2}$ such that $\rank(A_0)\leq r$, we have 
	\begin{align*}
	\log \mathcal{N}\big(c_5\epsilon, \{f_A \in \mathcal{F}_r: \ell_n(f_A,f_{A_0})\leq 2\epsilon\}, \ell_n\big)\leq 2(m_1+m_2)r \cdot \log\big(18\bar{\nu}/c_5\underline{\nu}\big).
	\end{align*} 
\end{lemma}
We will need the following result.
\begin{lemma}\label{lem:local_ent_lr_F_mat}
	Let $S(r,B)=\{A \in \R^{m_1\times m_2}:\rank(A)\leq r,\pnorm{A}{2}\leq B\}$. Then $
	\mathcal{N}\big(\epsilon,S(r,B),\pnorm{\cdot}{2}\big)\leq \left(\frac{9B}{\epsilon}\right)^{(m_1+m_2-1)r}$.
\end{lemma}
\begin{proof}[Proof of Lemma \ref{lem:local_ent_lr_F_mat}]
	The case for $B=1$ follows from Lemma 3.1 of \cite{candes2011tight} and the general case follows by a scaling argument. We omit the details.
\end{proof}
\begin{proof}[Proof of Lemma \ref{lem:local_ent_tracereg}]
	We only need to consider the case $r\leq r_{\max}$. First note that the entropy in question equals 
	\begin{align*}
	\log \mathcal{N}\big(c_5\sqrt{n}\epsilon, \{\mathcal{X}(A-A_0): \pnorm{\mathcal{X}(A-A_0)}{2}\leq 2\sqrt{n}\epsilon,\rank{A}\leq r\}, \pnorm{\cdot}{2}\big). 
	\end{align*}
	By uniform RIP$(\bm{\nu};\mathcal{I})$, the set to be covered is contained in
	\begin{align*}
	\{\mathcal{X}(A-A_0): \pnorm{A-A_0}{2}\leq 2\epsilon/\underline{\nu},\rank{A}\leq r\}\subset \mathcal{X}(S(2r,2\epsilon/\underline{\nu})).
	\end{align*}
	On the other hand, again by uniform RIP$(\bm{\nu};\mathcal{I})$, a $c_5\epsilon/\bar{\nu}$-cover of the set $S(2r,2\epsilon/\underline{\nu})$ under the Frobenius norm $\pnorm{\cdot}{2}$ induces a $c_5\sqrt{n}\epsilon$-cover of $\mathcal{X}(S(2r,2\epsilon/\underline{\nu}))$ under the Euclidean $\pnorm{\cdot}{2}$ norm. This implies that the entropy can be further bounded from above by 
	\begin{align*}
	\log \mathcal{N}\big(c_5\epsilon/\bar{\nu},S(2r,2\epsilon/\underline{\nu}),\pnorm{\cdot}{2}\big)\leq 2(m_1+m_2)r \cdot \log\big(18\bar{\nu}/c_5\underline{\nu}\big),
	\end{align*} 
	where the last inequality follows from Lemma \ref{lem:local_ent_lr_F_mat}.
\end{proof}

Now we take $
\delta_{n,r}^2 = \big(\frac{4\log(18\bar{\nu}/c_5\underline{\nu})}{c_7}\vee \frac{1}{\eta} \big)\frac{\cdot (m_1+m_2)r \log \bar{m}}{n}.$
Clearly $\delta_{n,r}^2$ satisfies (\ref{cond:suplinearity_delta}) with $\mathfrak{c}=\gamma=1, \mathfrak{h}_0=\infty$.

\begin{lemma}\label{lem:submodel_suff_prior_tracereg}
	Suppose that $\mathcal{X}:\R^{m_1\times m_2}\to \R^n$ is uniform RIP$(\bm{\nu};\mathcal{I})$, and that (\ref{cond:rate_tracereg}) holds. Then (P2) in Assumption \ref{assump:prior_mass_general} holds.
\end{lemma}
\begin{proof}[Proof of Lemma \ref{lem:submodel_suff_prior_tracereg}]
	We only need to consider $r\leq r_{\max}$. First note that
	\begin{align}\label{ineq:suff_mass_tracereg_1}
	&\Pi_{n,r}\left(\left\{f_A \in \mathcal{F}_r: \ell_n^2(f_A,f_{A_{0,r}})\leq \delta_{n,r}^2/c_3\right\}\right)\\
	& = \Pi_G\left(\left\{A \in \R^{m_1\times m_2}: \pnorm{\mathcal{X}(A-A_{0,r})}{2} \leq \sqrt{n}\delta_{n,r}/\sqrt{c_3}, \rank(A)\leq r\right\}\right)\nonumber\\
	&\geq \Pi_G\left(\left\{A \in \R^{m_1\times m_2}: \pnorm{A-A_{0,r}}{2} \leq \delta_{n,r}/\bar{\nu}\sqrt{c_3}, \rank(A)\leq r \right\}\right).\nonumber
	\end{align}
	Let $A_{0,r}\equiv \sum_{i=1}^r \sigma_i \bar{u}_i \bar{v}_i^\top$ be the spectral decomposition of $A_{0,r}$, and let $u_i\equiv \sqrt{\sigma_i}\bar{u}_i$ and $v_i\equiv \sqrt{\sigma_i}\bar{v}_i$. Then $A_{0,r}\equiv \sum_{i=1}^r u_i v_i^\top$. Now for $u_i^\ast \in B_{m_1}(u_i,\epsilon)$ and $v_i^\ast \in B_{m_2}(v_i,\epsilon)$, $i=1,\ldots, r$, let $A^\ast\equiv \sum_{i=1}^r  u_i^\ast (v_i^\ast)^\top$, then by noting that the Frobenius norm is sub-multiplicative and that $\pnorm{u_i}{2}=\pnorm{v_i}{2}=\sqrt{\sigma_i}$, we have for $\epsilon\leq 1$,
	\begin{align*}
	\pnorm{A^\ast-A_{0,r}}{2}&\leq \sum_{i=1}^r \big( \lVert(u_i-u_i^\ast)v_i^\top\rVert_{2}+ \lVert u_i^\ast(v_i-v_i^\ast)^\top\rVert_{2}\big) \\
	&\leq \sum_{i=1}^r  \left(\epsilon \sqrt{\sigma_i}+(\sqrt{\sigma_i}+\epsilon)\epsilon\right)\leq \rho_r \epsilon,
	\end{align*}
	where $\rho_r\equiv \sum_{i=1}^r (2\sqrt{\sigma_i}+1)$. Now with $\bar{\epsilon}_{n,r}\equiv \frac{\delta_{n,r}}{\bar{\nu}\sqrt{c_3} \rho_r} \wedge 1$ we see that (\ref{ineq:suff_mass_tracereg_1}) can be further bounded from below by 
	\begin{align*}
	&\Pi_G\big(\cap_{i=1}^r\left\{({u}_i^\ast,{v}_i^\ast): {u}_i^\ast \in B_{m_1}({u}_i, \bar{\epsilon}_{n,r}), {v}_i^\ast \in B_{m_2}({v}_i,\bar{\epsilon}_{n,r})\right\}\big)\\ 
	& \geq  (\tau_{r,g}^{\trace})^{(m_1+m_2)r} \prod_{i=1}^ r \mathrm{vol}\left({B_{m_1}({u}_i,\bar{\epsilon}_{n,r})}\right)\cdot \mathrm{vol}\left(B_{m_2}({v}_i,\bar{\epsilon}_{n,r})\right) \\
	& \geq (\tau_{r,g}^{\trace}\cdot \bar{\epsilon}_{n,r})^{(m_1+m_2)r}v_{m_1}^r v_{m_2}^r\geq e^{-(m_1+m_2)r\cdot  \left(\log \bar{m}/2+\log \tau_{r,g}^{-1} + \log (\bar{\epsilon}_{n,r}^{-1} \vee 1)\right)},
	\end{align*} 
	where $v_d=\mathrm{vol}(B_d(0,1))$, and $v_d\geq (1/\sqrt{d})^d$. The right side of the above display is bounded from below by $e^{-2n\delta_{n,r}^2}$, if we require 
	\begin{align*}
	\max\left\{\log \tau_{r,g}^{-1}, \log (\bar{\epsilon}_{n,r}^{-1} \vee 1) \right\}\leq {\log \bar{m}}/{(2\eta)}.
	\end{align*} 
	It is easy to calculate that  
	\begin{align*}
	(\bar{\epsilon}_{n,r})^{-2} &\leq \big(\bar{\nu}^2 c_3 \rho_r^2\eta n\big)\vee 1\leq  8\eta \bar{\nu}^2 c_3(1\vee \sigma_{\max}(A_{0,r})) r_{\max}^2 n\\
	&\leq 4\bar{\nu}^2  (1\vee \sigma_{\max}(A_{0,r})) n^3\leq 4\bar{\nu}^2  (1\vee \sigma_{\max}(A_{0,r}))^2 n^4,
	\end{align*}
	by using $r_{\max}\leq n$ and $c_3=1$. 
	Now the conclusion follows by noting that (\ref{cond:rate_tracereg}) implies the requirement.
\end{proof}

\begin{proof}[Proof of Theorem \ref{thm:rate_tracereg}]
	The theorem follows by Corollary \ref{cor:regression_model}, Proposition \ref{prop:prior_generic} coupled with Lemmas \ref{lem:local_ent_tracereg} and \ref{lem:submodel_suff_prior_tracereg}.
\end{proof}

\subsection{Proof of Theorem \ref{thm:rate_iso}}

\begin{lemma}\label{lem:local_ent_iso}
	Let $n\geq 2$. Then for any $g \in \mathcal{F}_m$, $
	\log \mathcal{N}\big(c_5\epsilon, \{f \in \mathcal{F}_m: \ell_n(f,g)\leq 2\epsilon\}, \ell_n\big)\leq 2\log(6/c_5)\cdot m\log(en)$. 
\end{lemma}
\begin{proof}[Proof of Lemma \ref{lem:local_ent_iso}]
	Let $\mathscr{Q}_m$ denote all $m$-partitions of the design points $x_1,\ldots,x_n$. Then it is easy to see that $\abs{\mathscr{Q}_m}=\binom{n}{m-1}$. For a given $m$-partition $Q \in \mathscr{Q}_m$, let $\mathcal{F}_{m,Q}\subset \mathcal{F}_m$ denote all monotonic non-decreasing functions that are constant on the partition $Q$. Then the entropy in question can be bounded by 
	\begin{align*}
	\log \bigg[\binom{n}{m-1}\max_{Q \in \mathscr{Q}_m} \mathcal{N}\big(c_5 \epsilon , \{f \in \mathcal{F}_{m,Q}:\ell_n(f,g)\leq 2\epsilon\},\ell_n\big)\bigg].
	\end{align*} 
	On the other hand, for any fixed $m$-partition $Q \in \mathscr{Q}_m$, the entropy term above equals
	$
	\mathcal{N}\big(c_5\sqrt{n}\epsilon , \{\bm{\gamma} \in \mathcal{P}_{n,m,Q}:\pnorm{\bm{\gamma}-\bm{g}}{2}\leq 2\sqrt{n}\epsilon\},\pnorm{\cdot}{2}\big),$ where $\mathcal{P}_{n,m,Q}\equiv \{(f(x_1),\ldots,f(x_n)): f \in \mathcal{F}_{m,Q}\}$.
	By Pythagoras theorem, the set involved in the entropy is included in $\{\bm{\gamma} \in \mathcal{P}_{n,m,Q}:\pnorm{\bm{\gamma}-\pi_{\mathcal{P}_{n,m,Q}}(\bm{g})}{2}\leq 2\sqrt{n}\epsilon\}$ where $\pi_{\mathcal{P}_{n,m,Q}}$ is the natural projection from $\R^n$ onto the subspace $\mathcal{P}_{n,m,Q}$. 
	Clearly $\mathcal{P}_{n,m,Q}$ is contained in a linear subspace with dimension no more than $m$. Using entropy result for the finite-dimensional space [Problem 2.1.6 in \cite{van1996weak}, page 94 combined with the discussion in page 98 relating the packing number and covering number], 
	\begin{align*}
	\log \mathcal{N}\big(c_5\epsilon , \{f \in \mathcal{F}_{m,Q}:\ell_n(f,f_{0,m})\leq 2\epsilon\},\ell_n\big)\leq \log \big(\frac{3\cdot 2\sqrt{n}\epsilon}{c_5\sqrt{n}\epsilon}\big)^m =m\log(6/c_5). 
	\end{align*} 
	The claim follows by combining the estimates and $\log\binom{n}{m-1}\leq m\log(en)$.
\end{proof}
Hence we can take $
\delta_{n,m}^2\equiv \big(\frac{4\log(6/c_5)}{c_7}\vee \frac{1}{\eta}\big) \frac{ m \log(en)}{n}.$
It is clear that (\ref{cond:suplinearity_delta}) is satisfied with $\mathfrak{c}=\gamma=1, \mathfrak{h}_0=\infty$.
\begin{lemma}\label{lem:submodel_mass_iso}
	Suppose that (\ref{cond:iso}) holds . Then (P2) in Assumption \ref{assump:prior_mass_general} holds.
\end{lemma}
\begin{proof}[Proof of Lemma \ref{lem:submodel_mass_iso}]
	Let $Q_{0,m}=\{I_k\}_{k=1}^m$ be the associated $m$-partition of $\{x_1,\ldots,x_n\}$ of $f_{0,m} \in \mathcal{F}_m$ with the convention that $\{I_k\}\subset \{x_1,\ldots,x_n\}$ is ordered from smaller values to bigger ones. Then it is easy to see that  $\bm{\mu}_{0,m}=(\mu_{0,1},\ldots,\mu_{0,m})\equiv \big(f_{0,m}(x_{i(1)}),\ldots, f_{0,m}(x_{i(m)})\big) \in \R^m$ is well-defined and $\mu_{0,1}\leq \ldots \leq \mu_{0,m}$. It is easy to see that any $f \in \mathcal{F}_{m,Q_{0,m}}$ satisfying the property that $\sup_{1\leq k\leq m}\abs{f(x_{i(k)})-\mu_{0,k}}\leq \delta_{n,m}/\sqrt{c_3}$ leads to the error estimate $\ell_n^2(f,f_{0,m})\leq \delta_{n,m}^2/c_3$. Hence 
	\begin{align*}
	&\Pi_{n,m}(\{f \in \mathcal{F}_m: \ell_n^2(f,f_{0,m})\leq \delta_{n,m}^2/c_3\})\\  
	&\geq \binom{n}{m-1}^{-1}\Pi_{\bar{g}_m}(\{f \in \mathcal{F}_{m,Q_{0,m}}: \ell_n^2(f,f_{0,m})\leq \delta_{n,m}^2/c_3 \})\\ 
	&\geq  \binom{n}{m-1}^{-1}\Pi_{\bar{g}_m}(\{\bm{\mu} \in \R^m: \bm{\mu}\equiv \big(\mu_{0,k}+\epsilon_k\big)_{k=1}^m, 0\leq \epsilon_1\leq\ldots\leq \epsilon_m \leq \delta_{n,m}/\sqrt{c_3}\})\\
	&\geq \binom{n}{m-1}^{-1} \cdot \inf_{ \substack{\bm{\mu} \in \R^m: \bm{\mu}\equiv (\mu_{0,k}+\epsilon_k)_{k=1}^m,  0\leq \epsilon_1\leq \ldots \leq \epsilon_m \leq 1\wedge \delta_{n,m}/\sqrt{c_3}} } \bar{g}_m(\bm{\mu}) (1\wedge \delta_{n,m}/\sqrt{c_3} )^m \frac{1}{m!}\\ 
	&\geq \binom{n}{m-1}^{-1}\cdot  (\tau_{m,g}^{\textrm{iso}})^m (1\wedge \delta_{n,m}/\sqrt{c_3} )^m\\
	&\geq e^{- m\log(en) - m \log \big((\tau_{m,g}^{\textrm{iso}})^{-1}\vee 1\big)-m \log \big(\frac{\sqrt{c_3}}{\delta_{n,m}}\vee 1\big)}
	\end{align*} 
	Here the first inequality in the last line follows from the definition of $\bar{g}_m$ and $\tau_{m,g}^{\textrm{iso}}$. The claim follows by verifying (\ref{cond:iso}) implies that the second and third term in the exponent above are both bounded by $\frac{1}{2\eta} \cdot m\log (en)$ [the third term does not contribute to the condition since $\sqrt{c_3}\delta_{n,m}^{-1}\leq n$ by noting $c_3=1$ in the Gaussian regression setting and definition of $\eta$].
\end{proof}

\begin{proof}[Proof of Theorem \ref{thm:rate_iso}]
	The theorem follows by Corollary \ref{cor:regression_model}, Proposition \ref{prop:prior_generic} coupled with Lemmas \ref{lem:local_ent_iso} and \ref{lem:submodel_mass_iso}.
\end{proof}

We now prove Lemma \ref{lem:verify_cond_iso}. We need the following result.

\begin{lemma}\label{lem:proj_bounded}
	Let $\bm{f}_0:=(f_0(x_1),\ldots,f_0(x_n))\in \R^n$, and \newline $\bm{f}_{0,m}:=(f_{0,m}(x_1),\ldots,f_{0,m}(x_n)) \in \R^n$ where $f_{0,m} \in \arg\min_{g \in \mathcal{F}_m} \ell_n^2(f_0,g)$. Suppose that $\pnorm{\bm{f}_0}{2}\leq L$, and that there exists some element $f \in \mathcal{F}_m$ such that $\bm{f}\equiv (f(x_1),\ldots, f(x_n)) $ satisfies $\pnorm{\bm{f}}{2}\leq L$. Then  $\pnorm{\bm{f}_{0,m}}{2}\leq 3L$.
\end{lemma}
\begin{proof}[Proof of Lemma \ref{lem:proj_bounded}]
	It can be seen that 
	\begin{align*}
	\bm{f}_{0,m} \in \arg\min_{\bm{\gamma} \in \mathcal{P}_{n,m}}\mathcal{L}_{f_0}(\bm{\gamma})\equiv \arg\min_{\bm{\gamma} \in \mathcal{P}_{n,m}} \pnorm{\bm{f}_0-\bm{\gamma}}{2},
	\end{align*}
	where $\mathcal{P}_{n,m}\equiv \{(f(x_1),\ldots,f(x_n)):f \in \mathcal{F}_m\}$. 
	For any $\bm{\gamma} \in \mathcal{P}_{n,m}$ such that $\pnorm{\bm{\gamma}}{2}\leq L$, the loss function satisfies $\mathcal{L}_{f_0}(\bm{\gamma})\leq 2L$ by triangle inequality. If $\pnorm{\bm{f}_{0,m}}{2}> 3L$, then 
	\begin{align*}
	\mathcal{L}_{f_0}(\bm{f}_{0,m})=\pnorm{\bm{f}_0-\bm{f}_{0,m}}{2}\geq \pnorm{\bm{f}_{0,m}}{2}-\pnorm{\bm{f}_0}{2}>3L - L=2L,
	\end{align*}
	contradicting the definition of $\bm{f}_{0,m}$ as a minimizer of $\mathcal{L}_{f_0}(\cdot)$ over $\mathcal{P}_{m,n}$. This shows the claim. 
\end{proof}

\begin{proof}[Proof of Lemma \ref{lem:verify_cond_iso}]
	Let $L=\int_{0}^1 f^2$. Note that $
	\pnorm{\bm{f}_0}{2}^2\leq 2n \int_0^1 f^2(x)\ \d{x} =2n L^2$. 
	By Lemma \ref{lem:proj_bounded}, we see that $\pnorm{\bm{f}_{0,m}}{2}\leq 3\sqrt{2n}L$ which entails that $\pnorm{f_{0,m}}{\infty}\leq 3\sqrt{2n}L$. Now the conclusion follows from $g(3\sqrt{2nL}+1)\geq (en)^{-1/(2\eta)}$ while the left side is at least on the order of $n^{-\alpha/2}$ as $n \to \infty$.
\end{proof}

\subsection{Proof of Theorem \ref{thm:rate_cvxreg}}\label{section:proof_cvx}

Checking the local entropy assumption \ref{assump:local_ent_general} requires some additional work. The notion of \emph{pseudo-dimension} will be useful in this regard. Following \cite{pollard1990empirical} Section 4, a subset $V$ of $\R^d$ is said to have \emph{pseudo-dimension} $t$, denoted as $\mathrm{pdim}(V)=t$, if for every $x \in \R^{t+1}$ and indices $I=(i_1,\cdots,i_{t+1})\in\{1,\cdots,n\}^{t+1}$ with $i_{\alpha}\neq i_{\beta}$ for all $\alpha\neq \beta$, we can always find a sub-index set $J\subset I$ such that no $v \in V$ satisfies both $
v_i> x_i  \textrm{ for all } i \in J$  and $
v_i< x_i  \textrm{ for all } i \in I\setminus J$.

\begin{lemma}\label{lem:local_ent_Fm}
	Let $n\geq 2$. Suppose that $\mathrm{pdim}(\mathcal{P}_{n,m})\leq D_m$ where $\mathcal{P}_{n,m}:=\{\big(f(x_1),\ldots,f(x_n)\big)\in \R^n: f \in \mathcal{F}_m\}$. Then for all $g \in \mathcal{F}_m$, 
	\begin{align*}
	\log \mathcal{N}\big(c_5\epsilon, \{f \in \mathcal{F}_m: \ell_n(f,g)\leq 2\epsilon\},\ell_n) \leq C\cdot D_m \log n
	\end{align*}
	for some constant $C>0$ depending on $c_5$.
\end{lemma}

To prove Lemma \ref{lem:local_ent_Fm}, we need the following result, cf. Theorem B.2 \cite{guntuboyina2012optimal}.
\begin{lemma}\label{lem:entropy_pdim}
	Let $V$ be a subset of $\R^n$ with $\sup_{v \in V}\pnorm{v}{\infty}\leq B$ and pseudo-dimension at most $t$. Then, for every $\epsilon>0$, we have
	\begin{align*}
	\mathcal{N}(\epsilon,A,\pnorm{\cdot}{2})\leq\bigg(4+\frac{2B\sqrt{n}}{\epsilon}\bigg)^{\kappa t},
	\end{align*}
	holds for some absolute constant $\kappa\geq 1$.
\end{lemma}

\begin{proof}[Proof of Lemma \ref{lem:local_ent_Fm}]
	Note that the entropy in question can be bounded by $
	\log \mathcal{N}\big(c_5\epsilon\sqrt{n} , \{\mathcal{P}_{n,m}-\bm{g}\}\cap B_n(0,2\sqrt{n}\epsilon),\pnorm{\cdot}{2}\big)$. 
	Since translation does not change the pseudo-dimension of a set, $\mathcal{P}_{n,m}-\bm{g}$ has the same pseudo-dimension with that of $\mathcal{P}_{n,m}$, which is bounded from above by $D_m$ by assumption. Further note that $\{\mathcal{P}_{n,m}-\bm{g}\}\cap B_n(0,2\sqrt{n}\epsilon)$ is uniformly bounded by $2\sqrt{n}\epsilon$, hence an application of Lemma \ref{lem:entropy_pdim} yields that the entropy can be further bounded as follows: 
	\begin{align*}
	\log \mathcal{N}\big(c_5\epsilon,  \{f \in \mathcal{F}_m: \ell_n(f,g)\leq 2\epsilon\} ,\ell_n)  \leq \kappa D_m \log \big(4+4n/c_5)\leq C\cdot  D_m\log n
	\end{align*}
	for some constant $C>0$ depending on $c_5$ whenever $n\geq 2$. 
\end{proof}

The pseudo-dimension of the class of piecewise affine functions $\mathcal{F}_m$ can be well controlled, as the following lemma shows.

\begin{lemma}[Lemma 4.9 in \cite{han2016multivariate}]\label{lem:pdim_cvxfcn}
	$\mathrm{pdim}(\mathcal{P}_{n,m})\leq 6md\log 3m$.
\end{lemma}

As an immediate result of Lemmas \ref{lem:local_ent_Fm} and \ref{lem:pdim_cvxfcn}, we can take for $n\geq 2$, $
\delta_{n,m}^2:=(C \vee 1/\eta)d\cdot \frac{\log n}{n}\cdot m\log 3m$ for some $C\geq 2/c_7$ depending on $c_5,c_7$.

\begin{lemma}\label{lem:submodel_mass_cvxreg}
	Suppose that (\ref{cond:cvxreg_1}) holds and $n\geq d$. Then (P2) in Assumption \ref{assump:prior_mass_general} holds.
\end{lemma}
\begin{proof}[Proof of Lemma \ref{lem:submodel_mass_cvxreg}]
	We write $f_{0,m}\equiv \max_{1\leq i\leq m}\big(a_i\cdot x+b_i\big)$ throughout the proof. We first claim that for any $a_i^\ast \in B_d(a_i,\delta_{n,m}/2\sqrt{c_3 d})$ and $b_i^\ast \in B_1(b_i, \delta_{n,m}/2\sqrt{c_3})$, let $g_m^\ast(x):=\max_{1\leq i\leq m} (a_i^\ast \cdot x+ b_i^\ast)$, then $\ell_\infty(g_m^\ast,f_{0,m})\leq \delta_{n,m}/\sqrt{c_3}$. To see this, for any $x \in \mathfrak{X}$, there exists some index $i_x \in \{1,\ldots,m\}$ such that $g_m^\ast(x) = a_{i_x}^\ast\cdot x+b_{i_x}^\ast$. Hence
	\begin{align*}
	g_m^\ast(x)-f_{0,m}(x)&\leq \big(a_{i_x}^\ast - a_{i_x}\big)\cdot x + \big(b_{i_x}^\ast -b_{i_x}\big)\leq \pnorm{a_{i_x}^\ast - a_{i_x}}{2}\pnorm{x}{2}+\abs{b_{i_x}^\ast -b_{i_x}} \\
	& \leq \frac{\delta_{n,m}}{2\sqrt{c_3 d}}\cdot \sqrt{d} +  \frac{\delta_{n,m}}{2\sqrt{c_3 }}= \frac{\delta_{n,m}}{\sqrt{c_3}}. 
	\end{align*} 
	The reverse direction can be shown similarly, whence the claim follows by taking supremum over $x \in \mathfrak{X}$. This entails that 
	\begin{align*}
	&\Pi_{n,m}(\{f \in \mathcal{F}_m: \ell_n^2(f,f_{0,m})\leq {\delta_{n,m}^2}/{c_3}\})\\
	& \geq\Pi_{G} (\cap_{i=1}^m\{(a_i^\ast,b_i^\ast): a_i^\ast \in B_d(a_i,\delta_{n,m}/2\sqrt{c_3 d}), b_i^\ast \in B_1(b_i, \delta_{n,m}/2\sqrt{c_3})\} ) \\ 
	&=\prod_{i=1}^m \Pi_{g^{\otimes d}} \big(B_d(a_i,\delta_{n,m}/2\sqrt{c_3 d})\big)\cdot \Pi_g \big(B_1(b_i,\delta_{n,m}/2\sqrt{c_3})\big)\\ 
	&\geq \prod_{i=1}^m g(\pnorm{a_i}{\infty}+1)^d\cdot g(\abs{b_i}+ 1)\cdot \bigg(\frac{\delta_{n,m}}{\sqrt{4c_3 d}}\wedge 1\bigg)^d v_d \bigg(\frac{\delta_{n,m}}{\sqrt{4c_3}}\wedge 1\bigg)\\ 
	&\geq \exp\bigg(-2m(d+1)\log\big(\tau_{m,g}^{-1} \vee 1\big)-m(d+1) \log \big(\frac{\sqrt{4c_3 d}}{\delta_{n,m}} \vee 1\big)-\frac{1}{2} md \log d\bigg),
	\end{align*} 
	where $v_d\equiv \mathrm{vol}(B_d(0,1))$ and we used the fact that $v_d\geq (1/\sqrt{d})^d$. Now by requiring that $n\geq d$ and 
	\begin{align*}
	\max\bigg\{2m(d+1)\log\big(\tau_{m,g}^{-1} \vee 1\big), m(d+1) \log \bigg(\frac{\sqrt{4c_3 d}}{\delta_{n,m}} \vee 1\bigg) \bigg\} \leq \frac{d}{2\eta} \log n\cdot m\log 3m, 
	\end{align*} 
	the claim follows by verifying (\ref{cond:cvxreg_1}) implies this requirement [since $\sqrt{4c_3 d}\delta_{n,m}^{-1}\leq \sqrt{n}$, the second term is bounded by $md\log n$. The inequality follows by noting $\eta<1/4$]. 
\end{proof}

\begin{lemma}\label{lem:c_cvx}
	For $n\geq 2$, (\ref{cond:suplinearity_delta}) is satisfied for $\mathfrak{c}=1,\gamma=2, \mathfrak{h}_0=\infty$.
\end{lemma}
\begin{proof}
	For fixed $n\geq 2$ and $\eta>0$, write $n\delta_{n,m}^2=c \log n(m\log 3m)$ throughout the proof, where $c\geq 2/c_7$. Then for any $\alpha\geq c_7/2$ and $h\geq 1$, since $\log(3m')\geq \log(3hm)\geq \log(3m)$ for any $m'\geq hm$, we have 
	\begin{align*}
	\sum_{m'\geq hm} e^{-\alpha n\delta_{n,m'}^2} \leq  \sum_{m'\geq hm} e^{-\alpha c m' (\log n\cdot \log 3m)}  = \frac{e^{-\alpha c hm \log n\log 3m}}{1-e^{-\alpha c \log n \log 3m}}\leq 2e^{-\alpha h n \delta_{n,m}^2}. 
	\end{align*} 
	For the second condition of (\ref{cond:suplinearity_delta}), note that for $\gamma=2$, in order to verify $\delta_{n,hm}^2\leq h^2 \delta_{n,m}^2$, it suffices to have $
	hm \log (3hm)\leq h^2 m \log(3m)$, equivalently $3hm\leq (3m)^h$, and hence $3^{h-1}\geq h$ for all $h\geq 1$ suffices. This is valid and hence completing the proof.
\end{proof}

\begin{proof}[Proof of Theorem \ref{thm:rate_cvxreg}]
	This is a direct consequence of Corollary \ref{cor:regression_model}, Lemma \ref{lem:submodel_mass_cvxreg} and \ref{lem:c_cvx}, combined with Proposition \ref{prop:prior_generic}.
\end{proof}

\subsection{Proof of Theorem \ref{thm:rate_hp}}

\begin{lemma}\label{lem:local_ent_hp}
	Let $n\geq 2$, then for any $g \in \mathcal{F}_{(s,m)}$, 
	\begin{align*}
	&\log \mathcal{N}(c_5\epsilon,\{f \in \mathcal{F}_{(s,m)}:\ell_n(f,g)\leq 2\epsilon \},\ell_n)\\
	&\leq 2\log(6/c_5)\big(s\log(ep)\wedge \rank(X)+m\log(en)\big).
	\end{align*} 
	\begin{proof}
		The proof borrows notation from the proof of Lemma \ref{lem:local_ent_iso}. Further let $\mathscr{S}_s$ denote all subsets of $\{1,\ldots,p\}$ with cardinality at most $s$. Then the entropy in the statement of the lemma can be further bounded by 
		\begin{align*}
		&\log \bigg[\binom{p}{s}\binom{n}{m-1}\max_{S \in \mathscr{S}_s,Q \in \mathscr{Q}_m}\mathcal{N}(c_5\epsilon,\{f \in \mathcal{F}_{(s,m),(S,Q)}: \ell_n(f,g)\leq 2\epsilon \},\ell_n)\bigg]\\
		&\leq s\log(ep)+m\log(en)\\
		&\qquad\qquad +\max_{S \in \mathscr{S}_s,Q \in \mathscr{Q}_m}\log \mathcal{N}(c_5\sqrt{n}\epsilon, \{\bm{\gamma}\in \mathcal{P}_{n,(S,Q)}: \pnorm{\bm{\gamma}-\bm{g}}{2}\leq 2\sqrt{n}\epsilon\}, \pnorm{\cdot}{2})
		\end{align*} 
		where $\mathcal{P}_{n,(S,Q)}\equiv \{(x_i^\top \beta+u(z_i))_{i=1}^n \in \R^n: \mathrm{supp}(\beta)=S, \newline u \textrm{ is constant on the partitions of } Q\}$ is contained in a linear subspace of dimension no more than $s+m$. The entropy can also be bounded by 
		\begin{align*}
		m\log(en)+\max_{Q \in \mathscr{Q}_m}\log \mathcal{N}(c_5\sqrt{n}\epsilon, \{\bm{\gamma}\in \mathcal{P}_{n,(\{1,\ldots,p\},Q)}: \pnorm{\bm{\gamma}-\bm{g}}{2}\leq 2\sqrt{n}\epsilon\}, \pnorm{\cdot}{2}),	
		\end{align*}
		which is contained in a linear subspace of dimension no more than $\rank(X)+m$. Now using similar arguments as in Lemma \ref{lem:local_ent_iso} proves the claim.
	\end{proof}
\end{lemma}
Hence we can take $\delta_{n,(s,m)}^2\equiv c' \frac{s\log (ep)\wedge \rank(X)+m\log(en)}{n}$ for a large constant $c'>0$.

\begin{lemma}\label{lem:suplinearity_hp}
	(\ref{cond:suplinearity_delta}) holds with $\mathfrak{c},\gamma$ depending on $\mathfrak{h}_0 \in [1,\infty)$ and $L$.
\end{lemma}
\begin{proof}
	For the first condition of (\ref{cond:suplinearity_delta}), note that for any $h \in [1,\mathfrak{h}_0]$ and $\alpha\geq c_7/2$, choose $c'>0$ such that $\alpha c'\geq 2L\vee 2$, it follows that 
	\begin{align*}
	&\sum_{(s',m')\geq (hs,hm)} e^{-\alpha n \delta_{n,(s',m')}^2}=\sum_{s'\geq hs}e^{-\alpha c' (s\log (ep)\wedge \rank(X) )} \sum_{m'\geq hm} e^{-\alpha c' m\log(en)}\\
	&\leq (1-e^{-\alpha c' })^{-1} e^{-(\alpha c'/2\mathfrak{h}_0)h(s\log (ep)\wedge \rank(X)+m\log (en))}\leq 2e^{-\alpha n h\delta_{n,(s,m)}^2/\mathfrak{c}^2}.
	\end{align*}
	The inequality in the middle for the previous display follows as
	\begin{align*}
	&\sum_{s'\geq hs}e^{-\alpha c' (s\log (ep)\wedge \rank(X))}\leq e^{- \alpha c' (hs\log (ep)\wedge \rank(X))+\log p} \\ 
	&\leq e^{- \min\{\alpha c' hs\log (ep)-\log p, \alpha c'  \rank(X)-\log p\}}  \leq e^{- (\alpha c'/2) (hs\log (ep)\wedge \rank(X))}\\
	&\leq e^{- (\alpha c'/2\mathfrak{h}_0) h(s\log (ep)\wedge \rank(X))}.
	\end{align*}
	The second condition of (\ref{cond:suplinearity_delta}) is easy to verify.
\end{proof}

\begin{lemma}\label{lem:submodel_mass_hp}
	Suppose (\ref{cond:hpreg_1}) holds. Then (P2) in Assumption \ref{assump:prior_mass_general} holds.
\end{lemma}
\begin{proof}
	Let $\delta_{n,s}^2\equiv c' (s\log (ep)\wedge \rank(X)) /n$ and $\delta_{n,m}^2\equiv c' m\log(en)/n$. Let $\tau_{s,g}\equiv \sup_{f_{0,(s,m)} } g(\pnorm{\beta_{0,s}}{\infty}+1)$. 
	
	First consider $s\log(e p )\leq \rank(X)$. Using notation in Lemma \ref{lem:local_ent_hp}, 
	\begin{align*}
	&\Pi_{n,(s,m)}(\{f \in \mathcal{F}_{(s,m)}: \ell_n^2(f,f_{0,(s,m)})\leq \delta_{n,(s,m)}^2/c_3\})\\
	& \geq \binom{p}{s}^{-1}\binom{n}{m-1}^{-1}\Pi_{g^{\otimes s}\otimes \bar{g}_m}(\{f \in \mathcal{F}_{(s,m),(S_0,Q_0)}: \ell_n^2(f,f_{0,(s,m)})\leq \delta_{n,(s,m)}^2/c_3\})
	\end{align*}
	where $f_{0,(s,m)} \in \mathcal{F}_{(s,m),(S_0,Q_0)}$. To bound the prior mass of the above display from below, it suffices to bound the product of the following two terms:
	\begin{align}\label{ineq:hpreg_0}
	\pi_s&\equiv \Pi_{g^{\otimes s}}(\{\beta \in B_0(s):\beta_{S_0^c}=0, \ell_n^2(h_\beta,h_{\beta_{0,s}})\leq \delta_{n,s}^2/2c_3\}),\\
	\pi_m&\equiv\Pi_{\bar{g}_m}(\{u \in \mathcal{U}_{m,Q_0}:\ell_n^2(u,u_{0,m})\leq \delta_{n,m}^2/2c_3\}).\nonumber
	\end{align}
	The first term equals 
	\begin{align*}
	&\Pi_{g^{\otimes s}}(\{\beta \in B_0(s): \beta_{S_0^c}=0, \pnorm{X\beta-X\beta_{0,s}}{2}\leq {\sqrt{n}\delta_{n,s}}/{\sqrt{2c_3}}\})\\
	& \geq  \Pi_{g^{\otimes s}}\bigg(\bigg\{\beta \in B_0(s): \beta_{S_0^c}=0, \pnorm{\beta-\beta_{0,s}}{2}\leq \frac{1}{\sigma_\Sigma}\cdot \frac{\delta_{n,s}}{\sqrt{2c_3}}\bigg\}\bigg).
	\end{align*} 
	Here the inequality follows by noting $
	\pnorm{X\beta-X\beta_{0,s}}{2}^2\leq n (\beta-\beta_{0,s})^\top \Sigma (\beta-\beta_{0,s})\leq n\sigma^2_{\Sigma}\pnorm{\beta-\beta_{0,s}}{2}^2$, 
	where $\sigma_\Sigma$ denotes the largest singular value of $X^\top X/n$. Note that $\sigma_\Sigma\leq \sqrt{p}$ since the trace for $X^\top X/n$ is $p$ and the trace of a p.s.d. matrix dominates the largest eigenvalue. The set above is supported on $\R^p_{S_0}$ and hence can be further bounded from below by $
	\tau_{s,g}^s \big( \frac{1}{\sigma_\Sigma}\cdot \frac{\delta_{n,s}}{\sqrt{2c_3}} \wedge 1\big)^{s}v_s$ where $v_s = \mathrm{vol}(B_s(0,1))$. Hence  
	\begin{align*}
	\pi_s \geq  (\tau_{s,g}\wedge 1)^s \bigg( \frac{1}{\sigma_\Sigma}\cdot \frac{\delta_{n,s}}{\sqrt{2c_3}} \wedge 1 \bigg)^{s}v_s\geq e^{ -\frac{1}{2}s \log s- s\log\big(\tau_{s,g}^{-1} \vee 1\big)- \frac{s}{2}\log\big(\frac{2c_3 \sigma_\Sigma^2}{\delta_{n,s}^2} \vee 1\big)},
	\end{align*} 
	where in the last inequality we used that $v_s\geq (1/\sqrt{s})^s$. By repeating the arguments in the proof of Lemma \ref{lem:submodel_mass_iso}, we have 
	\begin{align*}
	\pi_m\geq e^{ - m\log\big(\tau_{m,g}^{-1} \vee 1\big)- \frac{m}{2}\log\big(\frac{2c_3 }{\delta_{n,m}^2} \vee 1\big)}.
	\end{align*}
	Combining above estimates, 
	\begin{align*}
	&\Pi_{n,(s,m)}(\{f \in \mathcal{F}_{(s,m)}: \ell_n^2(f,f_{0,(s,m)})\leq \delta_{n,(s,m)}^2/c_3\})\\ &\geq e^{-2s\log (ep)-m\log(en)-s\log\big(\tau_{s,g}^{-1} \vee 1\big)-m\log\big(\tau_{m,g}^{-1} \vee 1\big)} \times e^{- \frac{s}{2}\log\big(\frac{2c_3 \sigma_\Sigma^2}{\delta_{n,s}^2} \vee 1\big) - \frac{m}{2}\log\big(\frac{2c_3 }{\delta_{n,m}^2} \vee 1\big) }.
	\end{align*} 
	The right side  is bounded from below by $e^{-2n\delta_{n,(s,m)}^2}$, if we require both 
	\begin{align*}
	& \min\bigg\{e^{-s\log(\tau_{s,g}^{-1} \vee 1)},  e^{- s\log\big(\frac{\sqrt{2c_3} \sigma_\Sigma}{\delta_{n,s}} \vee 1\big)} \bigg\}\geq e^{-\frac{1}{2\eta} s\log(ep)},\\
	& \min\bigg\{e^{-m\log(\tau_{m,g}^{-1} \vee 1)},  e^{- m\log\big(\frac{\sqrt{2c_3} }{\delta_{n,m}} \vee 1\big)} \bigg\}\geq e^{-\frac{1}{2\eta} m\log(en)}.
	\end{align*}
	The first terms in the above two lines can be verified by (\ref{cond:hpreg_1}). The other terms in the above two lines do not contribute by noting that $2c_3/\delta_{n,m}^2\leq \frac{2c_3c_7}{4\log(6/c_5)} n\leq (1/2)n\leq en$ since $c_3=1$ (in Gaussian regression model) and $c_7 \in (0,1)$, while $2c_3\sigma_\Sigma^2/\delta_{n,s}^2\leq \sigma_\Sigma^2 n\leq pn\leq p^2$ and $\eta<1/4$.
	
	Next for $s\log (ep)>\rank(X)$, we may proceed with 
	\begin{align*}
	&\Pi_{n,(s,m)}(\{f \in \mathcal{F}_{(s,m)}: \ell_n^2(f,f_{0,(s,m)})\leq \delta_{n,(s,m)}^2/c_3\})\\ &\geq \binom{n}{m-1}^{-1}\Pi\Big(\{f \in \cup_{\abs{S}=s}\mathcal{F}_{(s,m),(S,Q_0)}: \ell_n^2(f,f_{0,(s,m)})\leq \delta_{n,(s,m)}^2/c_3\}\Big). 
	\end{align*}
	To bound the prior mass of the above display from below, it suffices to bound from below the product of $\pi_m$ and
	\begin{align}\label{ineq:hpreg_1}
	\tilde{\pi}_s&\equiv \Pi \Big(\{\beta \in B_0(s): \pnorm{X(\beta-\beta_{0,s})}{}\leq \sqrt{n} \delta_{n,s}/\sqrt{2c_3}\}\Big).
	\end{align}
	Let $U \in \R^{n\times n}$ and $V \in \R^{p\times p}$ give rise to the SVD of $X$: $X=U\Lambda V\equiv U\mathrm{diag}(\sigma_1,\ldots,\sigma_{\rank(X)},0) V $ where $\sigma_1\geq \ldots \geq \sigma_{\rank(X)}>0$ are non-trivial singular values of $X$. It follows by writing $V = (v_1^\top \cdots v_p^\top)^\top$ that 
	\begin{align*}
	&\tilde{\pi}_s \geq \Pi \big({\beta}: \pnorm{\Lambda V(\beta-\beta_{0,s}) }{}\leq \sqrt{n}\delta_{n,s}/\sqrt{2c_3}\big)\\ 
	&=\sum_{\abs{S}=s} \binom{p}{s}^{-1}\Pi \bigg({\beta}: \beta_{S^c}=0,  \sum_{j=1}^{\rank(X)} \sigma_j^2 (v_j^\top (\beta-\beta_{0,s}) ) ^2 \leq n\delta_{n,s}^2/2c_3\bigg)\\ 
	&\geq \sum_{\abs{S}=s} \binom{p}{s}^{-1} \Pi \big({\beta}: \beta_{S^c}=0, \pnorm{\beta-\beta_{0,s}}{2}^2 \leq c'/(2c_3 \sigma_1^2)\big).	
	\end{align*}
	By choosing $c'>2c_3\sigma_1^2 (\pnorm{\beta_{0,s}}{\infty}+1)^2$, the RHS of the previous display can be bounded from below by $g(1)$, as desired. $\pi_m$ can be handled similarly as in the case $s\log(ep)\leq \rank(X)$. 
\end{proof}

\begin{proof}[Proof of Theorem \ref{thm:rate_hp}]
	The claim of the theorem follows by Corollary \ref{cor:regression_model}, Proposition \ref{prop:prior_generic} and Lemmas \ref{lem:local_ent_hp}-\ref{lem:submodel_mass_hp}.
\end{proof}

\subsection{Proof of Theorem \ref{thm:rate_cov}}

\begin{lemma}\label{lem:local_ent_cov}
	For any $\Sigma_0 \in \mathfrak{M}_{(k,s)}$, the following entropy estimate holds: 
	\begin{align*}
	&\log \mathcal{N}\left(c_5\epsilon,\{\Sigma \in \mathfrak{M}_{(k,s)}: \pnorm{\Sigma-\Sigma_{0}}{F}\leq C_L\epsilon \},\pnorm{\cdot}{F}\right) \\
	& \leq ks\log(ep/s) +ks \log(6\sqrt{kL}/c_5\epsilon). 
	\end{align*} 
\end{lemma}
\begin{proof}
	The set involved in the entropy is equivalent to 
	\begin{align}\label{ineq:local_ent_cov_1}
	\left\{\Lambda \in \mathscr{R}_{(k,s)}(L):\pnorm{\Lambda\Lambda^\top-\Lambda_{0} \Lambda_{0}^\top }{F}\leq C_L\epsilon,\pnorm{\cdot}{F}\right\}.
	\end{align}
	We claim that $
	\sup_{\Lambda \in \mathscr{R}_{(k,s)}}\pnorm{\Lambda\Lambda^\top}{F}\leq \sqrt{kL}$. 
	To see this, let $\Lambda \equiv P \Xi Q^\top $ be the singular value decomposition of $\Lambda$, where $P\in \R^{p\times p}, Q\in \R^{k\times k}$ are unitary matrices and $\Xi \in \R^{p\times k}$ is a diagonal matrix. Then $\pnorm{\Lambda \Lambda^\top}{F}^2=\pnorm{\Xi \Xi^\top }{F}^2\leq k L$, proving the claim. Combined with (\ref{ineq:local_ent_cov_1}) and Euclidean embedding, we see that the entropy in question can be bounded as follows: 
	\begin{align*}
	&\log \mathcal{N}\left(c_5\epsilon,\{v \in B_0(ks;pk): \pnorm{v}{2}\leq 2\sqrt{kL}\},\pnorm{\cdot}{2}\right)\\ 
	&\leq \log \bigg[\binom{pk}{ks}\bigg(\frac{6\sqrt{kL}}{c_5\epsilon}\bigg)^{ks}\bigg]\leq ks\log(ep/s) +ks \log(6\sqrt{kL}/c_5\epsilon), 
	\end{align*} 
	where $B_0(s;pk)\equiv \{v \in \R^{pk}: \abs{\mathrm{supp}(v)}\leq s\}$.
\end{proof}

\begin{proof}[Proof of Theorem \ref{thm:rate_cov}]
	Take $\delta_{n,(k,s)}^2=KC' {ks}\log(C' p)/n$ for some $C'\geq e$ depending on $c_5,c_7,L$ and some absolute constant $K\geq 1$. Apparently (\ref{cond:suplinearity_delta}) holds with $\mathfrak{c}=1,\gamma=1,\mathfrak{h}_0=\infty$. The prior $\Pi_{n,(k,s)}$ on $\mathfrak{M}_{(k,s)}$ will be the uniform distribution on a minimal $\sqrt{C' {ks}\log(C' p)/c_3n}$ covering-ball of the set $\{\Sigma \in \mathfrak{M}_{(k,s)}\}$ under the Frobenius norm $\pnorm{\cdot}{F}$. The above lemma entails that the cardinality for such a cover is no more than $e^{C''ks \log(C'' p)}$ for another constant $C''\geq e$ depending on $c_3,c_5,c_7,L$. Hence we have that
	\begin{align*}
	\Pi_{n,(k,s)}(\{\Sigma\in \mathfrak{M}_{(k,s)}:\pnorm{\Sigma-\Sigma_{0,(k,s)}}{F}\leq \delta_{n,(k,s)}^2/c_3\})\geq e^{-C''ks \log(C''p)},
	\end{align*} 
	which can be bounded from below by $e^{-2n\delta_{n,(k,s)}^2}$ by choosing $K$ large enough. The claim of Theorem \ref{thm:rate_cov} now follows from these considerations along with Corollary \ref{cor:rate_cov_general}, Proposition \ref{prop:prior_generic}.
\end{proof}

\subsection{Proof of Theorem \ref{thm:rate_image_polytope}}

\begin{lemma}\label{lem:local_ent_image_polytope}
	For $\theta_0 \in \Theta_m$, we have $
	\log \mathcal{N}\big(c_5\epsilon,\{\theta \in \Theta_m, d_n(\theta,\theta_0)\leq 2\epsilon\}, d_n\big)\leq 4m \log \big(\frac{C_\eta m}{c_5^4\epsilon^4}\big)$. 
\end{lemma}
\begin{proof}
	We first claim that for $\epsilon\leq 1$, 
	\begin{align*}
	\log\mathcal{N}\big(\epsilon, \{\Gamma \in \mathscr{C}_m\}, \lambda(\cdot \Delta\cdot)\big)\leq m\log \left(\frac{9em}{\epsilon^2}\right).
	\end{align*} 
	To see this, fix $\delta>0$ to be chosen later, and partition $[0,1]^2$ into small squares with side length $\delta$. Let $\mathscr{D}_\delta$ be the set of all polytopes in $[0,1]^2$ with its at most $m$ vertices all located on the grid points of these small squares. Apparently $\abs{\Delta_\delta}\leq \binom{ (1+1/\delta)^2}{m}$. Then for each $\Gamma \in \mathscr{C}_m$, let $\Gamma_\delta \in \mathscr{D}_\delta$ be such that $\Gamma_\delta\supset \Gamma$ and that for every vertex $v$ of $\Gamma$, there exists a vertex $v_\delta$ of $\Gamma_\delta$ so that both $v$ and $v_\delta$ are in the same small square, with distance at most $\sqrt{2}\delta$. Then the points on the boundary of $\Gamma_\delta$ is within distance $\sqrt{2}\delta$ to $\Gamma$, and therefore $\lambda(\Gamma_\delta\Delta\Gamma)\leq \sqrt{2}(\sqrt{2}\delta)m=2\delta m$ (the estimate can be done in a conservative way by collapsing the set of vertices in $\Gamma$ that corresponding to the same vertex in $\Gamma_\delta$ into one vertex). Now let $\epsilon = 2\delta m$ yields the claim.
	
	Since 
	\begin{align*}
	d_n^2(\theta_0,\theta_1)\leq C_1^2\big( \abs{\xi_0-\xi_1}^2+\abs{\rho_0-\rho_1}^2+\lambda(\Gamma_0\Delta\Gamma_1)\big)
	\end{align*}
	for some constant $C_1^2>0$ depending only through $\eta$, it follows that 
	\begin{align*}
	&\log \mathcal{N}\big(c_5\epsilon,\{\theta \in \Theta_m, d_n(\theta,\theta_0)\leq 2\epsilon\}, d_n\big)\leq 2\log \mathcal{N}\big(c_5\epsilon/(\sqrt{3}C_1), [\eta,1-\eta],\abs{\cdot}\big)\\
	&\qquad\qquad\qquad +\log\mathcal{N}\big(c_5^2\epsilon^2/(3C_1^2), \{\Gamma \in \mathscr{C}_m,\Gamma \subset [\eta,1-\eta]^m\}, \lambda(\cdot \Delta\cdot)\big)\\
	&\leq 2\log \left(\frac{\sqrt{3}C_1}{c_5\epsilon}\right)+ m \log \left(\frac{81C_1^4 em}{c_5^4\epsilon^4}\right)\leq 4m \log \left(\frac{C_\eta m}{c_5^4\epsilon^4}\right),
	\end{align*}
	as desired.
\end{proof}

Now we take $\delta_{n,m}^2 \equiv C'_\eta \frac{m \log n}{n}$ for some large constant $C'_\eta>0$.

\begin{lemma}\label{lem:submodel_mass_image_polytope}
	For $\theta_0\in\Theta_m$, (P2) is satisfied for $n$ large enough depending on $\theta_0$.
\end{lemma}
\begin{proof}
	Let $\{v_i(\Gamma)\}_{i=1}^m$ be the vertices of $\Gamma \in \mathscr{C}_m$. Using again 
	\begin{align*}
	d_n^2(\theta_0,\theta)\leq C_1^2\big( \abs{\xi_0-\xi}^2+\abs{\rho_0-\rho}^2+\lambda(\Gamma_0\Delta\Gamma)\big),
	\end{align*}
	and that for $n$ large enough depending on $\Gamma_0$, for any $v_i \notin \Gamma_0$ such that $\pnorm{v_i-v_i(\Gamma_0)}{2}\leq \delta_{n,m}^2/(3\sqrt{2}m C_1^2 c_3)$, $\Gamma\equiv\mathrm{conv}(\{v_i\})$ has vertices exactly given by $\{v_i\}$, and $\lambda(\Gamma\Delta\Gamma_0)\leq \sqrt{2}\cdot \big(\delta_{n,m}^2/(3\sqrt{2}m C_1^2 c_3)\big) m=\delta_{n,m}^2/(3C_1^2c_3) $, we have 
	\begin{align*}
	& \Pi_{n,m}\big(\{\theta \in \Theta_m: d_n^2(\theta,\theta_0)\leq \delta_{n,m}^2/c_3\}\big)\\
	& \geq \Pi_\xi\big(\abs{\xi-\xi_0}^2\leq \delta_{n,m}^2/(3C_1^2 c_3)\big)\cdot \Pi_\rho \big(\abs{\rho-\rho_0}^2\leq \delta_{n,m}^2/(3C_1^2 c_3)\big)\\
	&\qquad\qquad \times \Pi_{\Gamma}\big(\pnorm{v_i(\Gamma)-v_i(\Gamma_0)}{2}\leq \delta_{n,m}^2/(3\sqrt{2}m C_1^2 c_3, v_i(\Gamma)\notin \Gamma_0)\big)\\
	&\gtrsim_\eta \delta_{n,m}^2\big(\delta_{n,m}^2/m^{3/2}\big)^{m}\geq \exp(-2n\delta_{n,m}^2),
	\end{align*}
	as long as $C'_\eta>0$ is large enough.
\end{proof}

\begin{proof}[Proof of Theorem \ref{thm:rate_image_polytope}]
	The claim follows by Corollary \ref{cor:rate_image}, Proposition \ref{prop:prior_generic} coupled with Lemmas \ref{lem:local_ent_image_polytope} and \ref{lem:submodel_mass_image_polytope}.
\end{proof}

\subsection{Proof of Theorem \ref{thm:rate_intensity}}

\begin{lemma}\label{lem:local_ent_intensity}
	For any $g \in \mathcal{F}_m$ such that $g\leq f_0$, and any $R\geq \pnorm{f_0}{\infty}\vee 1$, 
	\begin{align*}
	\log \mathcal{N}_{[}\big(c_5\epsilon^2,\{f \in \mathcal{F}_m,R\geq f\geq f_0: \bar{L}_1(f,g)\leq 4\epsilon^2\}, \bar{L}_1\big)\leq 2m \log \bigg(\frac{8e m R^2}{c_5\epsilon^2}\bigg).
	\end{align*} 
\end{lemma}
\begin{proof}[Proof of Lemma \ref{lem:local_ent_intensity}]
	Note that the local entropy with left bracketing in question can be bounded by its global counterpart $\mathcal{N}_{[}\big(c_5\epsilon^2,\{f \in \mathcal{F}_m,\abs{f}\leq R\}, \bar{L}_1\big)$.

	Let $m\geq 2$. Fix $\epsilon>0$, let $\delta^2= c_5\epsilon^2/(2Rm+1)$. Without loss of generality, we assume that $1/\delta^2 \in \N$, and we partition the interval $[0,1)$ into $\cup_{j=1}^{1/\delta^2}I_{j,\delta}\equiv \cup_{j=1}^{1/\delta^2}[(j-1)\delta^2, j\delta^2)$. For any $f \in \mathcal{F}_m$, let $f\equiv \sum_{j=1}^m a_j\bm{1}_{[t_{j-1},t_j)}$ for some  $0=t_0<t_1<\ldots<t_{m-1}<t_m=1$. Then $\{t_1,\ldots,t_{m-1}\}$ must be contained in $m-1$ intervals amongst $\{I_{j,\delta}\}_{j=1}^{1/\delta^2}$, namely, $\{\bar{I}_{k,\delta;f}\}_{k=1}^{m-1}$. Furthermore, $[0,1]\setminus \cup_{k=1}^{m-1} \bar{I}_{k,\delta;f}$ contains at most $m$ intervals. Now define $\bar{f}$ as follows: 
	\begin{align*}
	\bar{f} \equiv \sum_{k=1}^{m-1} (-R)\cdot \bm{1}_{\bar{I}_{k,\delta;f}} + \floor{\frac{f}{\delta^2}} \delta^2 \cdot \bm{1}_{[0,1]\setminus \cup_{k=1}^{m-1}\bar{I}_{k,\delta;f}}. 
	\end{align*} 
	Clearly $\bar{f}\leq f$, and 
	\begin{align*}
	\int_0^1 \big(f(x)-\bar{f}(x)\big)\ \d{x}\leq 2R m\delta^2+\int_{[0,1]\setminus \cup_{k=1}^{m-1}\bar{I}_{k,\delta;f}} \delta^2\ \d{x}\leq c_5 \epsilon^2.
	\end{align*} 
	On the other hand, there are at most $
	\binom{1/\delta^2}{m-1}\cdot \big(\frac{2R}{\delta^2}\big)^m$ many choices of $\bar{f}$, and hence 
	\begin{align*}
	&\log \mathcal{N}_{[}\big(c_5\epsilon^2,\{f \in \mathcal{F}_m,\abs{f}\leq R\}, \bar{L}_1\big)\leq \log \bigg[ \binom{1/\delta^2}{m-1}\cdot \big(\frac{2R}{\delta^2}\big)^m\bigg]\\
	&\leq (m-1)\log \bigg(\frac{e(2Rm+1)}{c_5\epsilon^2(m-1)}\bigg)+m \log \bigg(\frac{2R(2Rm+1)}{c_5\epsilon^2}\bigg)\leq 2m \log \bigg(\frac{8e m R^2}{c_5\epsilon^2}\bigg).
	\end{align*} 
	For $m=1$, it is clear the above bound holds so the proof is complete. 
\end{proof}

Hence we can take $\delta_{n,m}^2 \equiv \big(\frac{4}{c_7}\vee 2\big)\frac{m}{n}\log \big(8e nR^2/c_5\big)$. Clearly (\ref{cond:suplinearity_delta}) is satisfied with $\mathfrak{c}=\gamma=1, \mathfrak{h}_0=\infty$.

\begin{lemma}\label{lem:submodel_suff_prior_intensity}
	Suppose that $g_a$ has full support. For $n$ large enough depending on $f_0$ and the prior $g_a$, (P2) in Assumption \ref{assump:prior_mass_general} restricted to $\{f\geq f_{0,m}, \pnorm{f}{\infty}\leq \pnorm{f_0}{\infty}+1\}$ holds.
\end{lemma}
\begin{proof}
	Let $f_{0,m}\equiv \sum_{j=1}^m a_j^\ast \bm{1}_{[t^\ast_{j-1},t^\ast_j)}$ for some $t^\ast=(t_1^\ast,\ldots,t_{m-1}^\ast)$ with $0=t_0^\ast<t_1^\ast<\ldots<t_{m-1}^\ast<t_m^\ast=1$. Without loss of generality, we may assume that $\min\{ t_j^\ast-t_{j-1}^\ast: j\}>1/(2n\pnorm{f_0}{\infty})$ (otherwise we may merge such short intervals to construct a surrogate $\tilde{f}_{0,m}$, and the total difference between $\tilde{f}_{0,m}$ and $f_{0,m}$ in $L_1$ metric by doing this does not exceed $m/n$ so that there is no effect in the final oracle inequality). Let $u_j^\ast\equiv 2\cdot \bm{1}_{a_{j+1}^\ast<a_j^\ast}-1$. For any $t=(t_1,\ldots,t_{m-1})$ such that $t_j = t_j^\ast+ u_j^\ast \delta$ with $\delta<1/(4n\pnorm{f_0}{\infty}+1)$, and any $a=(a_1,\ldots,a_m)$ such that $a_j\geq a_j^\ast$ and $\max_j\abs{a_j-a_j^\ast}\leq 1/(4n)$, let $f\equiv \sum_{j=1}^m a_j\bm{1}_{[t_{j-1},t_j)}\geq f_{0,m}$. Then 
	\begin{align*}
	\int_0^1 \big(f(x)-f_{0,m}(x)\big)\ \d{x}\leq \frac{1}{4n}+m\cdot \frac{1}{4n\pnorm{f_0}{\infty}+1}\cdot \bigg(\pnorm{f_0}{\infty}+\frac{1}{4n}\bigg)\leq \frac{m}{n} \leq \delta_{n,m}^2/c_3
	\end{align*}
	by the definition of $\delta_{n,m}^2$ and the fact that $c_3=1$. This implies that with $\tau_{g_a}^{\textrm{int}}\equiv g_a(\pnorm{f_0}{\infty}+1)$ (it is easy to see $\pnorm{f_{0,m}}{\infty}\leq \pnorm{f_0}{\infty}$), 
	\begin{align*}
	&\Pi_{n,m} \big(\{f \in \mathcal{F}_m: f\geq f_{0,m}, \bar{L}_1(f,f_{0,m})\leq \delta_{n,m}^2/c_3\}\big)\\ 
	&\geq (4n\pnorm{f_0}{\infty}+1)^{-(m-1)}\bigg(1\wedge \frac{\tau_{g_a}^{\textrm{int}}}{4n}\bigg)^{m}\geq e^{-m \big(\log (4n\pnorm{f_0}{\infty}+1)+\log \big(1\vee \frac{4n}{\tau_{g_a}^{\textrm{int}}}\big)\big) }
	\end{align*} 
	Since $2n\delta_{n,m}^2\geq 4m\log(32en)$, it suffices to require that $
	\log (4n\pnorm{f_0}{\infty}+1)\vee \log \big(1\vee \frac{4n}{\tau_{g_a}^{\textrm{int}}}\big)\leq 2\log(32en)$, 
	which is satisfied for $n$ large.
\end{proof}

\begin{proof}[Proof of Theorem \ref{thm:rate_intensity}]
	Let $R_n\to \infty$ be a sequence such that $\log R_n\lesssim \log n$. We omit the superscript in the constants in the proof. Let $\bar{\mathcal{F}}_n\equiv \{f:[0,1]\to \R: \abs{f}\leq R_n, f \in \mathcal{F}_m\}$ be the localized subset of $\mathcal{F}$. By the decomposition (\ref{ineq:localize_bernstein}), the probability in question can be bounded by 
	\begin{align}\label{ineq:intensity_0}
	& P_{f_0}^{(n)} \bar{\Pi}_n\big(f\geq f_0, f \in \bar{\mathcal{F}}_n: \bar{L}_1(f,f_{0})> C_2 \big(\inf_{g \in \mathcal{F}_m}\bar{L}_1(f_0,g)+\frac{m\log (R_n^2 n)}{n}\big)\big\lvert N\big) \\
	&\qquad\qquad + P_{f_0}^{(n)} \Pi_n\big(f \notin \bar{\mathcal{F}}_n\big\lvert N\big).\nonumber
	\end{align}

	We first handle the first term in (\ref{ineq:intensity_0}). Now Corollary \ref{cor:rate_support_boundary} combined with Lemma \ref{lem:local_ent_intensity} and \ref{lem:submodel_suff_prior_intensity} yields that for $n$ large enough 
	\begin{align*}
	& P_{f_0}^{(n)} \bar{\Pi}_n\bigg(f \geq f_0, f \in \bar{\mathcal{F}}_n: \bar{L}_1(f,f_{0}) > C_2\big( \inf_{g \in \mathcal{F}_m\cap \bar{\mathcal{F}}_n}\bar{L}_1(f_0,g)+ { m\log (R_n^2 n)}/{n}\big)\big\lvert N\bigg)\\
	& \leq C_3e^{-n\epsilon_{n,m}^2/C_3},
	\end{align*} 
	where $\epsilon_{n,m}^2\equiv \max \{\inf_{g \in \mathcal{F}_m\cap \bar{\mathcal{F}}_n}\bar{L}_1(f_0,g), { m\log (R_n^2 n)}/{n}\}$. Here $C_2,C_3>0$ are absolute constants that do not depend on $R_n$. Note that in applying (modified) Lemma \ref{lem:submodel_suff_prior_intensity} we (implicitly) used the fact that the induced localized prior mass satisfies the following: 
	\begin{align*}
	& \bar{\Pi}_{n,m}\big(\{f \in \mathcal{F}_m\cap \bar{\mathcal{F}}_n: f\geq f_{0,m}, \bar{L}_1(f,f_{0,m})\leq \delta_{n,m}^2/c_3\}\big)\\
	&\geq \Pi_{n,m}\big(\{f \in \mathcal{F}_m\cap \bar{\mathcal{F}}_n: f\geq f_{0,m}, \bar{L}_1(f,f_{0,m})\leq \delta_{n,m}^2/c_3\}\big).
	\end{align*} 
	Next we handle the second term in (\ref{ineq:intensity_0}). Applying Lemma \ref{lem:posterior_denum_control} to the localized model with
	\begin{align*}
	\bar{\epsilon}_{n,m}^2\equiv \inf_{g \in \mathcal{F}_m\cap \bar{\mathcal{F}}_n}\bar{L}_1(f_0,g)+ C_4{ m\log (en)}/{n}=\bar{L}_1(f_0,f_{0,m})+C_4{m \log (en)}/{n}
	\end{align*}
	for $C_4>0$ large enough and $n$ large enough, we see that on an event $\mathcal{E}_n$ with $P_{f_0}^{(n)}$ probability at least $1-e^{-C_5n\bar{\epsilon}_{n,m}^2}$, it holds that 
	\begin{align*}
	&\int_{\bar{\mathcal{F}}_n} p_{f}^{(n)}/p_{f_0}^{(n)}\ \d{\bar{\Pi}_n(f)}\geq \lambda_n(m) \int_{f \in\mathcal{F}_m \cap \bar{\mathcal{F}}_n: \bar{L}_1(f,f_0)\leq  \bar{\epsilon}_{n,m}^2 }p_{f}^{(n)}/p_{f_0}^{(n)}\ \d{\bar{\Pi}_{n,m}(f)}\\
	& \gtrsim e^{-C_6 m\log(en)} \times \bar{\Pi}_{n,m}\big(\{f \in \mathcal{F}_m \cap \bar{\mathcal{F}}_n,f\geq f_{0,m}: \bar{L}_1(f,f_{0,m})\leq C_4 \frac{m}{n} \log(en)\}\big)\\ &\gtrsim e^{-C_7 m \log (en) }\big(\Pi_n(\bar{\mathcal{F}}_n)\big)^{-1}
	\end{align*}
	where the last inequality holds for $n$ large enough, and follows essentially from the same argument used in the proof of Lemma \ref{lem:submodel_suff_prior_intensity}. Now we have
	\begin{align*}
	& P_{f_0}^{(n)} \Pi_n\big(f \notin \bar{\mathcal{F}}_n|N\big)\leq P_{f_0}^{(n)} \Pi_n\big(f \notin \bar{\mathcal{F}}_n|N\big)\bm{1}_{\mathcal{E}_n} + P_{f_0}^{(n)}(\mathcal{E}_n^c) \\
	& \leq P_{f_0^{(n)}} \big[\frac{ \int_{f \notin \bar{\mathcal{F}}_n} p_{f}^{(n)}/p_{f_0}^{(n)}\ \d{\Pi_n(f)} }{\int_{\bar{\mathcal{F}}_n} p_{f}^{(n)}/p_{f_0}^{(n)}\ \d{{\Pi}_n(f)} }\bm{1}_{\mathcal{E}_n}\big]+P_{f_0}^{(n)}(\mathcal{E}_n^c)\\ 
	&\leq \frac{1}{\Pi_n(\bar{\mathcal{F}}_n)}\cdot P_{f_0^{(n)}} \big[\frac{ \int_{f \notin \bar{\mathcal{F}}_n} p_{f}^{(n)}/p_{f_0}^{(n)}\ \d{\Pi_n(f)} }{\int_{\bar{\mathcal{F}}_n} p_{f}^{(n)}/p_{f_0}^{(n)}\ \d{\bar{\Pi}_n(f)} }\bm{1}_{\mathcal{E}_n}\big]+P_{f_0}^{(n)}(\mathcal{E}_n^c)\\ 
	&\lesssim e^{C_7 m \log (en) } \cdot \Pi_n(\mathcal{F}\setminus \bar{\mathcal{F}}_n)+ e^{-C_5n\bar{\epsilon}_{n,m}^2}
	\end{align*} 
	Furthermore we have, 
	\begin{align*}
	&\Pi_n(\mathcal{F}\setminus \bar{\mathcal{F}}_n)\leq \sum_{k>m} \lambda_{n}(k) (\int_{\abs{x}>R_n} g(x)\ \d{x})^k\\
	&\lesssim \sum_{k>m} e^{-C_6 (k-1) \log(en)-k \log(\int_{\abs{x}>R_n} g(x)\ \d{x})^{-1}}\lesssim e^{-2C_7 m \log (en)},
	\end{align*} 
	where the last inequality follows as $
	\log(\int_{\abs{x}>R_n} g(x)\ \d{x})^{-1} \geq C' \log (en)$
	holds for a large enough constant $C'>0$. Combining the above estimates concludes the proof.
\end{proof}

\section{Proofs of auxiliary lemmas in Appendix \ref{section:proof_main_result}}\label{section:proof_lemma_main_result}

\begin{proof}[Proof of Lemma \ref{lem:global_test_general}]
	Without loss of generality we assume $d_0=0$. Let $
	\mathcal{F}_j:=\left\{f \in \mathcal{F}: j\epsilon<d_n(f,f_0)\leq 2j\epsilon\right\}$ and $\mathcal{G}_j\subset \mathcal{F}_j$ be the collection of functions that form a minimal $c_5 j\epsilon$ covering set of $\mathcal{F}_j$ under the metric $d_n$. Then by assumption $\abs{\mathcal{G}_j}\leq N(j\epsilon)$. Furthermore, for each $g \in \mathcal{G}_j$, it follows by Lemma \ref{lem:local_test_general} that there exists some test $\omega_{n,j,g}$ such that 
	\begin{align*}
	\sup_{f \in \mathcal{F}:d_n(f,g)\leq c_5d_n(g,f_0)}\big[P_{f_0}^{(n)}\omega_{n,j,g}+P_f^{(n)}(1-\omega_{n,j,g})\big]\leq c_6 e^{-c_7nd_n^2(g,f_0)}. 
	\end{align*} 
	Recall that $g \in \mathcal{G}_j \subset \mathcal{F}_j$, then $d_n(g,f_0)>j\epsilon$. Hence the indexing set above contains $
	\{f \in \mathcal{F}: d_n(f,g)\leq c_5j\epsilon\}$.
	Now we see that  
	\begin{align*}
	P_{f_0}^{(n)}\omega_{n,j,g}\leq c_6e^{-c_7nj^2\epsilon^2},\quad
	\sup_{f \in \mathcal{F}: d_n(f,g)\leq c_5j\epsilon} P_{f}^{(n)}(1-\omega_{n,j,g}) \leq c_6 e^{-c_7 nj^2\epsilon^2}.
	\end{align*} 
	Consider the global test $\phi_n:=\sup_{j\geq 1}\max_{g \in \mathcal{G}_j}\omega_{n,j,g}$, then 
	\begin{align*}
	P_{f_0}^{(n)}\phi_n &\leq P_{f_0}^{(n)} \sum_{j\geq 1}\sum_{g \in \mathcal{G}_j} \omega_{n,j,g} \leq c_6\sum_{j\geq 1}N (j\epsilon) e^{-c_7nj^2\epsilon^2} \\
	&\leq c_6 N (\epsilon)\sum_{j\geq 1} e^{-c_7nj^2\epsilon^2} \leq c_6 N (\epsilon)e^{-c_7 n\epsilon^2}\cdot \big(1-e^{-c_7n\epsilon^2 }\big)^{-1}.
	\end{align*} 
	On the other hand, for any $f \in \mathcal{F}$ such that $d_n(f,f_0)\geq \epsilon$, there exists some $j^\ast \geq 1$ and some $g_{j^\ast} \in \mathcal{G}_{j^\ast}$ such that $d_n(f,g_{j^\ast})\leq j^\ast c_5\epsilon$. Hence 
	\begin{align*}
	P_f^{(n)}(1-\phi_n)\leq P_f^{(n)}(1-\omega_{n,j^\ast,g_{j^\ast}})\leq c_6 e^{-c_7 n(j^\ast)^2\epsilon^2}\leq c_6 e^{-c_7n\epsilon^2}.
	\end{align*} 
	The right hand side is independent of individual $f \in \mathcal{F}$ such that $d_n(f,f_0)\geq \epsilon$ and hence the claim follows. 
\end{proof}

\begin{proof}[Proof of Lemma \ref{lem:posterior_denum_control}]
	WLOG we assume $d_0=0$. By Jensen's inequality, the probability in question is bounded by 
	\begin{align*}
	& P_{f_0}^{(n)} \bigg\{ \int \big( \log ({p_{f_0}^{(n)}}/{p_{f}^{(n)}})-P_{f_0}^{(n)}\log ({p_{f_0}^{(n)}}/{p_{f}^{(n)}}) \big)\ \d{\Pi(f)}  \\
	&\qquad\qquad \qquad\qquad \qquad\qquad \geq  \big(C+c_3)n\epsilon^2 -c_3 n \int d_n^2(f_0,f)\ \d{\Pi(f)}\bigg\}\\
	& \leq P_{f_0}^{(n)} \bigg[ \int \big( \log ({p_{f_0}^{(n)}}/{p_{f}^{(n)}})-P_{f_0}^{(n)}\log ({p_{f_0}^{(n)}}/{p_{f}^{(n)}}) \big)\ \d{\Pi(f)} \geq Cn\epsilon^2 \bigg]\\
	& \leq e^{-C\lambda n\epsilon^2}\cdot c_1 P_{f_0}^{(n)}e^{\lambda \int \big( \log ({p_{f_0}^{(n)}}/{p_{f}^{(n)}})-P_{f_0}^{(n)}\log ({p_{f_0}^{(n)}}/{p_{f}^{(n)}}) \big)\ \d{\Pi(f)}}\\ 
	&\leq
	P_{f_0}^{(n)}\int e^{\lambda \big( \log ({p_{f_0}^{(n)}}/{p_{f}^{(n)}})-P_{f_0}^{(n)}\log ({p_{f_0}^{(n)}}/{p_{f}^{(n)}}) \big)}\ \d{\Pi(f)}\leq \int e^{\psi_{\kappa_g nd_n^2(f_0,f),\kappa_\Gamma}(\lambda)} \d{\Pi(f)},
	\end{align*}
	where the last inequality follows from Fubini's theorem and Assumption \ref{assump:laplace_cond_kl}. Now the condition on the prior $\Pi$ entails that 
	\begin{align*}
	P_{f_0}^{(n)}\bigg(\int  ({p_f^{(n)}}/{p_{f_0}^{(n)}})\ \d{\Pi(f)}\leq e^{-(C+c_3) n\epsilon^2}\bigg)\leq c_1 e^{-C\lambda n\epsilon^2+\psi_{\kappa_g n\epsilon^2,\kappa_\Gamma}(\lambda)}.
	\end{align*}
	The claim follows by choosing $\lambda>0$ small enough depending on $C,\kappa$.
\end{proof}

\begin{proof}[Proof of Proposition \ref{prop:underfit_general}]
	By definition we have $\delta_{n,\tilde{m}}\geq d_n(f_0,f_{0,m})$ and $ \delta_{n,\tilde{m}-1}<d_n(f_0,f_{0,m})$. In this case, the global test can be constructed via $
	\tilde{\phi}_n:=\sup_{m' \in \mathcal{I}, m'\geq j\mathfrak{h}\tilde{m}} \phi_{n,m'}$.
	Then analogous to (\ref{ineq:generic_pdim_3}) and (\ref{ineq:generic_pdim_4}), for any random variable $U \in [0,1]$, we have exponential testability: 
	\begin{align*}
	P_{f_{0,m}}^{(n)}U\cdot  \tilde{\phi}_n &\leq  4c_6e^{-(c_7/2\mathfrak{c}^2)nj \mathfrak{h}\delta_{n,\tilde{m}}^2},\\
	\sup_{f \in \mathcal{F}_{j\mathfrak{h}\tilde{m}}: d_n^2(f,f_{0,m})\geq \mathfrak{c}^2 (j\mathfrak{h})^\gamma \delta_{n,\tilde{m}}^2} P_{f}^{(n)}(1-\tilde{\phi}_n) &\leq 2c_6 e^{-(c_7/\mathfrak{c}^2)nj \mathfrak{h}\delta_{n,\tilde{m}}^2}.
	\end{align*}
	Similar to (\ref{ineq:generic_pdim_5}), there exists an event $\tilde{\mathcal{E}}_n$ with
	\begin{align*}
	P_{f_{0,m}}^{(n)}(\tilde{\mathcal{E}}_n^c) \leq c_1 e^{- C' {c_7 n j\mathfrak{h}\delta_{n,\tilde{m}}^2}/{8c_3\mathfrak{c}^2}}
	\end{align*}
	and on the event $\tilde{\mathcal{E}}_n$, 
	\begin{align*}
	&\int \prod_{i=1}^n \frac{p_f}{p_{f_{0,m}}}\ \d{\Pi_n(f)}\\
	&\geq \lambda_{n}(m) e^{-{c_7n j\mathfrak{h}\delta_{n,\tilde{m}}^2}/{4\mathfrak{c}^2}} \Pi_{n,m}(\{f \in \mathcal{F}_m: d_n^2(f,f_{0,m})\leq {c_7 j\mathfrak{h}\delta_{n,\tilde{m}}^2}/{8c_3\mathfrak{c}^2} \}).
	\end{align*} 
	Repeating as in (\ref{ineq:generic_pdim_6}), 
	\begin{align*}
	&P_{f_{0,m}}^{(n)} \Pi_n\big(f \in \mathcal{F}: d_n^2(f,f_{0,m})>\mathfrak{c}^4 (2j\mathfrak{h})^\gamma d_n^2(f_0,f_{0,m}) \big\lvert X^{(n)}\big)(1-\tilde{\phi}_n)\bm{1}_{\tilde{\mathcal{E}}_n} \\ 
	&\leq \frac{ e^{ {c_7n j\mathfrak{h}\delta_{n,\tilde{m}}^2}/{4\mathfrak{c}^2}} }{\lambda_{n}(m) \Pi_{n,m}(\{f \in \mathcal{F}_m: d_n^2(f,f_{0,m})\leq {c_7 j\mathfrak{h}\delta_{n,\tilde{m}}^2}/{8c_3\mathfrak{c}^2} \}) }\\ 
	&\qquad\qquad \times \int_{ f \in \mathcal{F}: d_n^2(f,f_{0,m})> \mathfrak{c}^4 (2j\mathfrak{h})^\gamma d_n^2(f_0,f_{0,m})  } P_f^{(n)}(1-\tilde{\phi}_n)\ \d{\Pi_n(f)}\\ 
	&\leq (\cdots)\times \bigg(\sup_{ f \in \mathcal{F}_{j\mathfrak{h}\tilde{m}}: d_n^2(f,f_{0,m})\geq \mathfrak{c}^2 (j\mathfrak{h})^\gamma \delta_{n,\tilde{m}}^2 }P_f^{(n)}(1-\tilde{\phi}_n) + \Pi_n\big(\mathcal{F}\setminus \mathcal{F}_{j\mathfrak{h} \tilde{m}}\big)\bigg) \\ 
	&\leq C e^{-(c_7/4\mathfrak{c}^2)nj \mathfrak{h}\delta_{n,\tilde{m}}^2}.
	\end{align*} 
	Here the third line is valid since $
	\mathfrak{c}^4 (2j\mathfrak{h})^\gamma d_n^2(f_0,f_{0,m}) > \mathfrak{c}^4 (2j\mathfrak{h})^\gamma \delta_{n,\tilde{m}-1}^2 \geq \mathfrak{c}^2 (j\mathfrak{h})^\gamma \delta_{n,\tilde{m}}^2$
	by the right side of (\ref{cond:suplinearity_delta}), which entails $\delta_{n,\tilde{m}}^2 \leq \mathfrak{c}^2 2^\gamma \delta_{n,\tilde{m}-1}^2$. The fourth line uses exponential testability and assumption (P1), together with the fact that $\delta_{n,\tilde{m}}\geq \delta_{n,m}$. (\ref{ineq:underfit_oracle}) follows from exponential testability, probability estimate for $\mathcal{E}_n^c$.
\end{proof}

\section{Some formal connections with frequentist theory for $M$-estimators}\label{section:MLE}

In this section, we establish some formal structural similarities between the Bayes theory developed in this paper under the local Gaussianity condition Assumption \ref{assump:laplace_cond_kl}, and the frequentist theory for $M$-estimators.

Let us consider the simplest setup where only one big model $\mathcal{F}$ is available, and we consider the sieved MLE $\hat{f}_n$ for illustration of the Gaussian concentration technique. To this end, let $\delta_n>0$ be determined by the entropy condition
\begin{align}\label{cond:MLE_entropy}
\log \mathcal{N}(\delta_n, \mathcal{F}, d_n)\leq \kappa \cdot n\delta_n^2,
\end{align}
where $\kappa>0$ is a small enough constant depending on the constants in Assumption \ref{assump:laplace_cond_kl}. The sieved MLE $\hat{f}_n$ is defined by $
\hat{f}_n \equiv \arg\max_{f \in \mathcal{F}_{\delta_n}} \log p_f^{(n)}(X^{(n)})$, 
where $\mathcal{F}_{\delta_n}$ is a minimal $\delta_n$-net of $\mathcal{F}$ under $d_n$.
\begin{proposition}\label{prop:MLE}
	Suppose the local Gaussianity condition Assumption \ref{assump:laplace_cond_kl} and the entropy condition (\ref{cond:MLE_entropy}) hold. Then the sieved MLE defined above satisfies $
	P_{f_0}^{(n)}\big(d_n^2(\hat{f}_n,f_0)>\delta_n^2\big)\leq \exp(-\kappa' n\delta_n^2)$, 
	where $\kappa'>0$ is a constant depending on the constants in Assumption \ref{assump:laplace_cond_kl}.
\end{proposition}

The entropy condition (\ref{cond:MLE_entropy}) used for the sieved MLE is of global type since the construction of the net $\mathcal{F}_{\delta_n}$ does not allow information on $f_0$. Results of this type in the context of Gaussian regression and density estimation have long been known in the literature; we only refer the readers to \cite{van1996weak,van2000empirical}. Our result here seems to yield some new results for other locally Gaussian experiments considered in Section \ref{section:models}.

The structural similarity of Theorem \ref{thm:general_ms} (when only one model is used) and Proposition \ref{prop:MLE} is obvious: both assertions hold under the same local Gaussianity structure of the experiment and the entropy condition, and the posterior distribution in Theorem \ref{thm:general_ms} and the sieved MLE in Proposition \ref{prop:MLE} both enjoy Gaussian tail behavior. Furthermore, the proofs for both results use (one-sided) Gaussian concentration in an essential way.


\begin{proof}[Proof of Proposition \ref{prop:MLE}]
	Let $S_j\equiv \{f \in \mathcal{F}_{\delta_n}: 2^{j-1}\delta_n \leq d_n(f,f_0)\leq 2^{j}\delta_n\}$. If $\hat{f}_n \in S_j$, then since $\log p_{f_0}^{(n)}/p_{\hat{f}_n}^{(n)}\leq 0$, it follows that 
	\begin{align*}
	\max_{f \in S_j} \big(P_{f_0}^{(n)}\log ({p_{f_0}^{(n)}}/{p_{f}^{(n)}})-\log ({p_{f_0}^{(n)}}/{p_{f}^{(n)}})\big)\geq P_{f_0}^{(n)}\log ({p_{f_0}^{(n)}}/{p_{\hat{f}_n}^{(n)}})\geq c_22^{2j-2} n \delta_n^2.
	\end{align*}
	This implies that 
	\begin{align*}
	& P_{f_0}^{(n)}\big(d_n(\hat{f}_n,f_0)>\delta_n\big)\\
	& \leq \sum_{j=1}^\infty P_{f_0}^{(n)}\bigg(\max_{f \in S_j} \big(P_{f_0}^{(n)}\log ({p_{f_0}^{(n)}}/{p_{f}^{(n)}})-\log ({p_{f_0}^{(n)}}/{p_{f}^{(n)}})\big)\geq c_2 2^{2j-2} n\delta_n^2\bigg)\\
	&\leq  \sum_{j=1}^\infty \sum_{f \in S_j} P_{f_0}^{(n)}\bigg( P_{f_0}^{(n)}\log ({p_{f_0}^{(n)}}/{p_{f}^{(n)}})-\log ({p_{f_0}^{(n)}}/{p_{f}^{(n)}})\geq c_2 2^{2j-2} n\delta_n^2\bigg) \\
	&\leq \sum_{j=1}^\infty N(\delta_n) e^{-C_1 2^{2j} n \delta_n^2}\leq  e^{-C_2 n \delta_n^2+\log N(\delta_n)}\leq e^{-C_3 n \delta_n^2},
	\end{align*}
	as desired.
\end{proof}

\section{More examples}\label{section:more_examples}

This section contains addition examples,  including (i) regression models without boundedness restrictions, (ii) density estimation in location mixtures, (iii) estimation of piecewise constant signals in the Gaussian autoregression model and (iv) subset selection for sparse approximation of regression functions. The main purpose of (i) and (ii) is to demonstrate how the localization principle (cf. Section \ref{section:localization}) can be applied in situations where local Gaussianity may fail over the entire parameter space, but still essentially holds on suitably localized subsets of the parameter space. The purpose of (iii) is to perform some explicit calculations without losing additional logarithmic factors, when the parameter space is non-compact. The purpose of (iv) is to demonstrate how to adapt the machinery in the paper to complicated model structures that are non-nested.

\subsection{Removing boundedness restrictions in Section \ref{section:regression} }

The boundedness assumption in many examples in Section \ref{section:regression} is imposed for simplicity. Below we will remove the boundedness restriction in the binary regression model as a proof of concept. 

Let $n\geq 3$. Consider fitting $X_i\sim_{\mathrm{i.i.d.}} \mathrm{Bern}(\theta_i)$ by piecewise constant model $\Theta\equiv \{\theta \in [0,1]^n\}=\cup_{m=1}^n \Theta_m$, where $\Theta_m\equiv \{\theta \in \Theta \textrm{ has at most }m\textrm{ constant pieces}\}$.  The model selection prior $\Lambda_{n}$ on $m$ is chosen as
\begin{align}
\lambda_n(m)\propto \exp(-c^{\textrm{bin}}m \log (en)).
\end{align}
For the selected model $\Theta_m$, we use the prior $\Pi_{n,m}$ which first randomly selects $m-1$ change points from $\{2,\ldots,n-1\}$, and then assigns a product prior with density $g^{\otimes (m-1)}$ where $g$ is a density on $[0,1]$. 

\begin{proposition}
	Suppose $\theta_0 \in \Theta_m$ and $\theta_0 \in [\eta,1-\eta]^n$ for some $\eta>0$. If $g$ is such that $\int_{x \in [0,t]\cup [1-t,1]} g(x)\ \d{x}\leq e^{-{1/t^C}}$ for some large constant $C>0$ and $t>0$ small. Then there exists $C'>0$ (depending on $\eta$ and the prior) such that $
	P_{\theta_0}^{(n)} \Pi_n\big(\theta \in \Theta: \pnorm{\theta-\theta_0}{2}^2>C' m \log^{C'} n/n\big)\to 0$. 
\end{proposition}

The boundedness restrictions in other Laplace/Poisson models can be removed in a completely similar fashion so we omit these digressions.

\begin{proof}
	Let $\delta_{n,m}^2\equiv cm\log^c n/n$ for some large constant $c>0$. Let the localized parameter spaces be defined by $\bar{\Theta}_n\equiv \{\theta \in \Theta: w_n\leq \theta_1,\ldots,\theta_n\leq 1-w_n\}$, where $w_n\equiv 1/\log n$. By the decomposition (\ref{ineq:localize_bernstein}),
	\begin{align}\label{ineq:binary_1}
	&P_{\theta_0}^{(n)}\Pi_n\big(\theta \in \Theta: \pnorm{\theta-\theta_0}{2}^2>  \delta_{n,m}^2\big\lvert X^{(n)}\big)\\
	&\leq P_{\theta_0}^{(n)}\bar{\Pi}_n\big(\theta \in \bar{\Theta}_n: \pnorm{\theta-\theta_0}{2}^2>  \delta_{n,m}^2\big\lvert X^{(n)}\big)+ P_{\theta_0}^{(n)}\Pi_n(\theta \notin \bar{\Theta}_n|X^{(n)}).\nonumber
	\end{align}
	
	For the first term in (\ref{ineq:binary_1}), we use Theorem \ref{thm:general_ms}. By the proof of Lemma \ref{lem:bernstein_regression}, for any $\theta_0,\theta_1 \in \bar{\Theta}_n$,
	\begin{align*}
	w_n^2\log(1/w_n)\pnorm{\theta_0-\theta_1}{2}^2\lesssim n^{-1}P_{\theta_0}^{(n)}\log (p_{\theta_0}^{(n)}/p_{\theta_1}^{(n)})\lesssim (w_n\log(1/w_n))^{-1} \pnorm{\theta_0-\theta_1}{2}^2.
	\end{align*}
	Similarly we may verify the local Gaussianity condition with constants $\kappa=(\kappa_g,\kappa_\Gamma)$ depending polynomially on $w_n$. So Assumption \ref{assump:laplace_cond_kl} is verified by choosing $\{c_i\}$ and $\kappa$ (or its inverse) on the order of $\mathcal{O}(w_n^{C_1})$ for some $C_1>0$. Assumption \ref{assump:local_ent_general} can be verified immediately using the similar arguments as in Lemma \ref{lem:local_ent_iso}. Assumption \ref{assump:prior_mass_general} follows by similar (and simpler) arguments in Lemma \ref{lem:submodel_mass_iso} and the fact that $\bar{\Pi}_{n,m}(A)\geq \Pi_{n,m}(A)$ for any $A$. Hence, the first term on the RHS of (\ref{ineq:binary_1}) is bounded by
	\begin{align*}
	\exp\big(C_2 \log(1/\omega_n)-n\delta_{n,m}^2 \omega_n^{C_2}\big),
	\end{align*}
	which is $o(1)$ by our choice of $w_n$ and $c>0$ large enough.

	We handle the second term on the right hand side of (\ref{ineq:binary_1}) below. By applying Lemma \ref{lem:posterior_denum_control} to the localized model with $\epsilon^2\equiv \delta_{n,m}^2$, we see that on an event $\mathcal{E}_n$ with $P_{\theta_0}^{(n)}$ probability at least $1-e^{- m \log^{C} n\cdot w_n^{C_3}}=1-o(1)$, 
	\begin{align*}
	& \int_{\bar{\Theta}_n} p_{\theta}^{(n)}/p_{\theta_0}^{(n)}\ \d{\bar{\Pi}_n(\theta)}\geq \lambda_n(m) \int_{\theta \in \bar{\Theta}_n: \pnorm{\theta-\theta_0}{2}^2\leq  \delta_{n,m}^2/c_3 }p_{\theta}^{(n)}/p_{\theta_0}^{(n)}\ \d{\bar{\Pi}_{n,m}(\theta)}\\ 
	&\gtrsim e^{-m\log^{c}n \cdot w_n^{C_4}} \cdot  \bar{\Pi}_{n,m}\big(\{\theta \in \Theta_m \cap \bar{\Theta}_n: \pnorm{\theta-\theta_{0}}{2}^2\leq \delta_{n,m}^2/c_3\}\big)\\ 
	&\gtrsim e^{-m\log^cn \cdot w_n^{C_4}-m\log n/C_5 }\big(\Pi_n(\bar{\Theta}_n)\big)^{-1} \gtrsim e^{-m\log^{C_6} n/C_6 }\big(\Pi_n(\bar{\Theta}_n)\big)^{-1} 
	\end{align*}
	by choosing $c>0$ large enough. 
	Now we have that 
	\begin{align*}
	& P_{\theta_0}^{(n)} \Pi_n\big(\theta \notin \bar{\Theta}_n|X^{(n)}\big) \leq  P_{\theta_0}^{(n)} \Pi_n\big(\theta \notin \bar{\Theta}_n|X^{(n)}\big)\bm{1}_{\mathcal{E}_n}+ P_{\theta_0}^{(n)}(\mathcal{E}_n^c)\\
	&= P_{\theta_0^{(n)}} \bigg[\frac{ \int_{\theta \notin \bar{\Theta}_n} p_{\theta}^{(n)}/p_{\theta_0}^{(n)}\ \d{\Pi_n(\theta)} }{\int_{\Theta} p_{\theta}^{(n)}/p_{\theta_0}^{(n)}\ \d{\Pi_n(\theta)} }\bm{1}_{\mathcal{E}_n}\bigg]+P_{\theta_0}^{(n)}(\mathcal{E}_n^c)\\
	&\leq \frac{1}{\Pi_n(\bar{\Theta}_n)}\cdot P_{\theta_0^{(n)}} \bigg[\frac{ \int_{\theta \notin \bar{\Theta}_n} p_{\theta}^{(n)}/p_{\theta_0}^{(n)}\ \d{\Pi_n(\theta)} }{\int_{\bar{\Theta}_n} p_{\theta}^{(n)}/p_{\theta_0}^{(n)}\ \d{\bar{\Pi}_n(\theta)} }\bm{1}_{\mathcal{E}_n}\bigg]+P_{\theta_0}^{(n)}(\mathcal{E}_n^c)\\
	& \lesssim e^{m\log^{C_6} n/C_6} \cdot \Pi_n(\Theta\setminus \bar{\Theta}_n)+ o(1),
	\end{align*}
	where in the last inequality we used a previous inequality and Fubini's theorem. On the other hand, 
	\begin{align*}
	&\Pi_n(\Theta\setminus \bar{\Theta}_n) \leq \sum_{k\geq 1} \lambda_n(k) k \int_{x \in [0,w_n]\cup [1-w_n,1]} g(x)\ \d{x}\\
	&\lesssim  \int_{x \in [0,w_n]\cup [1-w_n,1]} g(x)\ \d{x}\leq e^{-\log^{C_7}n/C_7}
	\end{align*} 
	for some large $C_7>0$ by the assumption on $g$. 
\end{proof}

\subsection{Density estimation in location mixtures}

Consider estimation of a density $f_0$ on $\R$ from the class of location mixtures $\cup_{m=1}^\infty\mathcal{F}_m$ where $\mathcal{F}_m$ consists densities of the type 
\begin{align*}
h(x;m,\mu,w,\sigma)\equiv\sum_{j=1}^m w_j\psi_\sigma(x-\mu_j),
\end{align*} 
where $\sigma>0$, $w_j\geq 0$, $\sum_{j=1}^m w_j =1$, $\mu_j \in \R$ and $
\psi_\sigma(x)\equiv  e^{-{x^2}/{2\sigma^2}}/\sqrt{2\pi \sigma^2}$. This problem has received considerable attention, see e.g. \cite{ghosal2001entropies,rousseau2010rates,kruijer2010adaptive,scricciolo2016rates,donnet2018posterior} and references therein for some Bayesian developments. The model selection prior $\Lambda_{n}$ on $m$ is chosen as
\begin{align}
\lambda_n(m)\propto \exp(-c^{\textrm{mix}}m \log (en)).
\end{align}
A prior $\Pi_{n,m}$ on the model $\mathcal{F}_m$ is naturally induced by a product prior $\Pi_w\otimes \Pi_\mu\otimes \Pi_\sigma$. For simplicity, we assume that $\Pi_w$ has the standard Dirichlet distribution, $\Pi_\mu, \Pi_\sigma$ have Lebesgue density $g_\mu^{\otimes m}, g_\sigma$ with the following properties: $g_\mu$ has full support on $\R$ such that $-\log g_\mu(x)\asymp \log(x)$ as $x \to \infty$, and $-\log g_\sigma(x)\asymp \log(1/x)$ as $x\to 0$ and $-\log g_\sigma(x)\asymp \log(x)$ as $x \to \infty$.

\begin{proposition}\label{prop:mixture}
	Suppose that $f_0 \in \mathcal{F}_m$, and the priors are specified as above. Then there exist $C>0, \gamma>0$ depending only on the priors such that $
	P_{f_0}^{(n)} \Pi_n\big(f \in \mathcal{F}: h^2(f,f_{0})> C{m\log^\gamma n}/n\big\lvert X^{(n)}\big) \to 0$. 
\end{proposition}

Proposition \ref{prop:mixture} says that the posterior distribution under such hierarchical priors adapts to the finite mixtures at a nearly parametric rate. Although this result does not seem to be explicitly spelled out in the literature, we believe that it can also be derived along the lines, e.g. \cite{kruijer2010adaptive}. Indeed, \cite{kruijer2010adaptive} proved adaptive behavior of the posterior contraction rates with respect to the local smoothness of the density, under similar hierarchical priors. It is clear from the above proposition that adaptation to the smoothness of the density can be accomplished once the quantity $\inf_{f \in \mathcal{F}_m} h^2(f,f_0)$ can be shown to be adaptive to the smoothness of $f_0$. This has been the main focus of \cite{kruijer2010adaptive} (in Kullback-Leibler divergence). The main purpose here, instead of repeating along the lines of \cite{kruijer2010adaptive}, rests in demonstrating how the localization principle can be used in the mixture model. 

It can also be seen immediately from the proof that the Gaussian kernel can be replaced by any kernel of form considered in \cite{kruijer2010adaptive}.

\begin{proof}[Proof of Proposition \ref{prop:mixture}]
	Let $\bar{\mathcal{F}}_n\equiv \{h(\cdot;m,\mu,w,\sigma):  \mu \in [-b_n,b_n]^m, \sigma \in [\underline{\sigma}_n,\bar{\sigma}_n]\}$ where $b_n\asymp (\log n)^{\gamma_1}, \bar{\sigma}_n\asymp\underline{\sigma}_n^{-1}\asymp (\log n)^{\gamma_2}$ for a sufficiently large $\gamma_1>\gamma_2$. For any $f \in \bar{\mathcal{F}}_n$, define $\tilde{f} \equiv f\bm{1}_{[-2b_n,2b_n]}+f_0\bm{1}_{\R\setminus [-2b_n,2b_n]}$, and $f^\ast = \tilde{f}/\int \tilde{f}$. Note that
	\begin{align*}
	\int_{\R}\tilde{f}(x)\ \d{x} = 1-\int_{\R\setminus [-2b_n,2b_n]} (f+f_0)(x)\ \d{x} = 1+\mathcal{O}(e^{-b_n^{2}/(2\bar{\sigma}_n^2)}),
	\end{align*}
	since 
	\begin{align*}
	\int_{\R\setminus [-2b_n,2b_n]} f(x)\ \d{x}&\lesssim\bigg( \sum_{j} w_j\bigg) \int_{2b_n}^\infty  \psi_{\sigma}(x-b_n)\ \d{x}\\
	&\lesssim \int_{b_n/\bar{\sigma}_n}^\infty e^{-x^2/2}\ \d{x}\lesssim e^{-b_n^2/(2\bar{\sigma}_n^2)},
	\end{align*} 
	and for $n$ large
	\begin{align*}
	\int_{\R\setminus [-2b_n,2b_n]} f_0(x)\ \d{x}\lesssim e^{-b_n^2/(2\bar{\sigma}_n^2)}.
	\end{align*} 
	Now define $\bar{\mathcal{F}}^\ast_n$ to be the set containing all $f^\ast$ defined as above from some $f \in \bar{\mathcal{F}}_n$.  Note that for any $f \in \bar{\mathcal{F}}_n$, we have that 
	\begin{align*}
	h^2(f,f_0)\lesssim h^2(f^\ast,f_0)+h^2(f^\ast,f) \lesssim h^2(f^\ast,f_0)+ \mathcal{O}(e^{-b_n^2/(2\bar{\sigma}_n^2)}).
	\end{align*}
	Then for a large enough constant $C>0$, by the decomposition (\ref{ineq:localize_bernstein}), we have for $n$ large, 
	\begin{align*}
	&P_{f_0}^{(n)} \Pi_n\bigg(f \in \mathcal{F}: h^2(f,f_{0})> C\frac{m\log^\gamma n}{n}\bigg\lvert X^{(n)}\bigg)\\ 
	&\leq P_{f_0}^{(n)} \bar{\Pi}_n\bigg(f \in \bar{\mathcal{F}}_n: h^2(f^\ast,f_{0})> C_1\frac{m\log^\gamma n}{n} \bigg\lvert X^{(n)}\bigg)+P_{f_0}^{(n)} \Pi_n\big(f \notin \bar{\mathcal{F}}_n\big\lvert X^{(n)}\big),
	\end{align*}
	which can be bounded by
	\begin{align}\label{ineq:mixture_1}
	&P_{f_0}^{(n)} \bar{\Pi}_n^\ast\big(f^\ast \in \bar{\mathcal{F}}_n^\ast: h^2(f^\ast,f_{0})> C_1{m\log^\gamma n}/{n} \big\lvert X^{(n)}\big)\\
	&\qquad + P_{f_0}^{(n)} \Pi_n^\ast\big(f^\ast \notin \bar{\mathcal{F}}_n^\ast\big\lvert X^{(n)}\big)+ {1}/{n}\nonumber
	\end{align}
	where $\Pi_n^\ast, \bar{\Pi}_n^\ast$ are the natural induced priors from $\Pi_n, \bar{\Pi}_n$. The last inequality follows by noting that 
	\begin{align*}
	P_{f_0}^{(n)}(\max_{1\leq i\leq n} \abs{X_i}>2b_n)\lesssim ne^{-b_n^2/(2\bar{\sigma}_n^2)}\leq 1/(2n)
	\end{align*}
	for $\gamma_1\gg \gamma_2$.
	
	We handle the first term on the right hand side of (\ref{ineq:mixture_1}). To this end, we first verify the local Gaussianity condition Assumption \ref{assump:laplace_cond_kl}. Clearly for any $f_0^\ast,f_1^\ast \in \bar{\mathcal{F}}^\ast_n$,  
	\begin{align*}
	\sup_{x \in \R} \biggabs{\frac{f_0^\ast(x)}{f_1^\ast(x)}} \leq  \sup_{x \in [-2b_n,2b_n]} \biggabs{\frac{f_0(x)}{f_1(x)}}\cdot \frac{1+\mathcal{O}(e^{-b_n^2/(2\bar{\sigma}_n^2)})}{1-\mathcal{O}(e^{-b_n^2/(2\bar{\sigma}_n^2)})}\lesssim \frac{\bar{\sigma}_n}{\underline{\sigma}_n} e^{(3b_n/\underline{\sigma}_n)^2}. 
	\end{align*} 
	By Lemma 8 of \cite{ghosal2007posterior}, 
	\begin{align*}
	h^2(f_0^\ast,f_1^\ast) &\leq \frac{1}{n} P_{f_0^\ast}^{(n)} \log ({P_{f_0^\ast}^{(n)}}/{P_{f_1^\ast}^{(n)}})\lesssim h^2(f_0^\ast,f_1^\ast)\big(1+\log\pnorm{{f_0^\ast}/{f_1^\ast} }{\infty}\big)\\
	&\lesssim h^2(f_0^\ast,f_1^\ast)\big(1+(b_n/\underline{\sigma}_n)^2+\log(\bar{\sigma}_n/\underline{\sigma}_n)\big),\\
	\mathrm{Var}_{f_0^\ast} \big(\log ({f_0^\ast}/{f_1^\ast})\big) & \lesssim h^2(f_0^\ast,f_1^\ast)\big(1+\log\pnorm{{f_0^\ast}/{f_1^\ast} }{\infty}\big)^2\\
	&\lesssim h^2(f_0^\ast,f_1^\ast)\big(1+(b_n/\underline{\sigma}_n)^2+\log(\bar{\sigma}_n/\underline{\sigma}_n)\big)^2.
	\end{align*} 
	By the classical Bernstein inequality, the local Gaussianity condition on $\bar{\mathcal{F}}^\ast_n$ holds with $c_1=c_2=1$, $c_3 \asymp \kappa_\Gamma\asymp \big(1+(b_n/\underline{\sigma}_n)^2+\log(\bar{\sigma}_n/\underline{\sigma}_n)\big)$ and $\kappa_g \asymp \big(1+(b_n/\underline{\sigma}_n)^2+\log(\bar{\sigma}_n/\underline{\sigma}_n)\big)^2$.
	
	Next we verify Assumption \ref{assump:local_ent_general}. Let $\delta_{n,m}^2\equiv C' \frac{m}{n}\log n$ for some large constant $C'>0$. Since the Hellinger distance is bounded by the square root of total variational distance, we have 
	\begin{align*}
	\log \mathcal{N}(c_5\epsilon, \bar{\mathcal{F}}^\ast_n\cap \mathcal{F}_m, h)\leq \log \mathcal{N}(c_5^2\epsilon^2, \bar{\mathcal{F}}^\ast_n\cap \mathcal{F}_m, d_{\mathrm{TV}}).
	\end{align*}
	By Lemma 3 of \cite{kruijer2010adaptive}, for any $f_0^\ast,f_1^\ast \in \bar{\mathcal{F}}_n^\ast\cap \mathcal{F}_m$ that are defined through $f_i = h(\cdot;m, \mu^i,w^j,\sigma^j) (j=0,1)$, 
	\begin{align*}
	d_{\mathrm{TV}}(f_0^\ast,f_1^\ast) &\lesssim \mathcal{O}(e^{-b_n^2/(2\bar{\sigma}_n^2)}) + \pnorm{w^0-w^1}{1}+\pnorm{\psi}{\infty}\sum_{i=1}^{m} \frac{w^0_i\wedge w^1_i}{\sigma^0\wedge \sigma^1} \abs{\mu^0_i-\mu^1_i}+ \frac{\abs{\sigma^0-\sigma^1}}{\sigma^0\wedge \sigma^1}\\
	&\leq C_2\big(e^{-b_n^2/(2\bar{\sigma}_n^2)} + \pnorm{w^0-w^1}{1}+\underline{\sigma}_n^{-1} \pnorm{\mu^0-\mu^1}{1}+ \underline{\sigma}_n^{-1}\abs{\sigma^0-\sigma^1}\big).
	\end{align*} 
	Here $C_2>0$ is an absolute constant. Now for any $1\geq \epsilon^2\geq 4C_2 e^{-b_n^2/(2\bar{\sigma}_n^2)}/c_5^2$, with $\Delta^m$ denoting the unit simplex in  $\R^m$ and using Lemma 5 of \cite{kruijer2010adaptive}, we have 
	\begin{align*}
	&\log \mathcal{N}(c_5^2\epsilon^2, \bar{\mathcal{F}}^\ast_n\cap \mathcal{F}_m, d_{\mathrm{TV}}) \\ 
	&\leq \log \mathcal{N}\bigg(\frac{c_5^2\epsilon^2}{4C_2}, \Delta^{m}, \pnorm{\cdot}{1}\bigg)+ \log \mathcal{N}\bigg(\frac{c_5^2\epsilon^2\underline{\sigma}_n}{4C_2}, [-b_n,b_n]^{m}, \pnorm{\cdot}{1}\bigg) \\
	&\qquad\qquad + \log \mathcal{N}\bigg(\frac{c_5^2\epsilon^2\underline{\sigma}_n}{4C_2}, [\underline{\sigma}_n,\bar{\sigma}_n], \abs{\cdot}\bigg)\\
	& \leq m \log\bigg(\frac{20C_2}{c_5^2\epsilon^2}\bigg)+ \log \bigg(\frac{m!(b_n+1)^{m}(4C_2)^{m}}{(c_5^2\epsilon^2\underline{\sigma}_n)^{m}}\bigg)+ \log \bigg(\frac{4C_2(\bar{\sigma}_n-\underline{\sigma}_n)}{c_5^2\epsilon^2\underline{\sigma}_n}\bigg).
	\end{align*} 
	Using that $c_5\asymp (c_3\kappa_\Gamma)^{-1}\wedge (c_3\kappa_g)^{-1}$ and $\log (m!)\lesssim m\log m$, we have 
	\begin{align*}
	\log \mathcal{N}(c_5\epsilon^2, \bar{\mathcal{F}}^\ast_n\cap \mathcal{F}_m, d_{\mathrm{TV}}) &\lesssim m \bigg(\log m+ \log \bigg(\frac{C_3 b_n\vee (\bar{\sigma}_n-\underline{\sigma}_n)}{c_5^2\epsilon^2 \underline{\sigma}_n}\bigg)\bigg)\\
	& \lesssim m \log n\leq (c_7/2) n\delta_{n,m}^2.
	\end{align*} 
	It is easy to check that $\epsilon^2$ hits the boundary $\delta_{n,m}^2$ by choosing $\gamma>0$ large enough.
	
	We continue to verify Assumption \ref{assump:prior_mass_general}. As before, it suffices to control from below the quantity $
	\Pi_{n,m}\big(\{f \in \bar{\mathcal{F}}_n\cap \mathcal{F}_m, h^2(f,f_{0})\leq \delta_{n,m}^2/(2c_3)\}\big)$. Again by Lemma 3 of \cite{kruijer2010adaptive}, for any $f_1,f_2 \in \bar{\mathcal{F}}_n\cap \mathcal{F}_m$ with $f_i=h(\cdot;m,\mu^i,w^j,\sigma^j)(j=1,2)$, we have 
	\begin{align*}
	h^2(f_1,f_2)\lesssim d_{\textrm{TV}}(f_1,f_2)\leq C_4\big( \pnorm{w^1-w^2}{1}+\underline{\sigma}_n^{-1} \pnorm{\mu^1-\mu^2}{1}+ \underline{\sigma}_n^{-1}\abs{\sigma^1-\sigma^2}\big).
	\end{align*} 
	In view of Lemma 6 of \cite{kruijer2010adaptive}, the above display implies
	\begin{align*}
	&\Pi_{n,m}\big(\{f \in \bar{\mathcal{F}}_n\cap \mathcal{F}_m, h^2(f,f_{0})\leq \delta_{n,m}^2/(2c_3)\}\big)\\
	& \geq \Pi_w \big(\Delta_m(w^{0},\delta_{n,m}^2/(6C_4 c_3))\big)\\
	&\qquad\qquad\times  \prod_{j=1}^m \Pi_\mu\bigg(\abs{\mu_j-\mu_j^{0}}\leq \frac{\delta_{n,m}^2\underline{\sigma}_n}{6C_4 m c_3}\bigg) \Pi_\sigma \bigg(\abs{\sigma-\sigma^{0}}\leq \frac{\delta_{n,m}^2\underline{\sigma}_n}{6C_4 c_3}\bigg)\\
	&\gtrsim e^{-C_5 m \log n}\geq e^{-2n\delta_{n,m}^2}
	\end{align*}
	Now apply Theorem \ref{thm:general_ms}, we see that the first term on the right hand side of (\ref{ineq:mixture_1}) can be bounded by $
	\exp(-C\log^{\gamma'} n)$ for some $\gamma'>0$ if $\gamma>0$ is chosen large enough.
	
	Next we handle the second term on the RHS of (\ref{ineq:mixture_1}). By applying Lemma \ref{lem:posterior_denum_control} to the localized model with $\epsilon^2\equiv \delta_{n,m}^2$ and using the same arguments as before, on an event $\mathcal{E}_n$ with $P_{f_0}^{(n)}$ probability at least $1-e^{-C_6 m\log^{\gamma''} n}$, 
	\begin{align*}
	&\int_{\bar{\mathcal{F}}_n^\ast} p_{f^\ast}^{(n)}/p_{f_0}^{(n)}\ \d{\bar{\Pi}_n^\ast(f^\ast)}\geq \lambda_n(m) \int_{f \in \bar{\mathcal{F}}_n\cap \mathcal{F}_m, h^2(f,f_{0})\leq \delta_{n,m}^2/c_3 }p_{f^\ast}^{(n)}/p_{\theta_0}^{(n)}\ \d{\bar{\Pi}^\ast_{n,m}(f^\ast)}\\
	& \gtrsim e^{-C_7m\log^{C_8} n} \cdot (\Pi_n^\ast(\bar{\mathcal{F}}_n))^{-1} {\Pi}_{n,m}^\ast\big(\{f \in \bar{\mathcal{F}}_n\cap \mathcal{F}_m, h^2(f,f_{0})\leq \delta_{n,m}^2/c_3\}\big)\\
	& \gtrsim e^{-m\log^{C_9} n}(\Pi_n^\ast(\bar{\mathcal{F}}_n))^{-1}.
	\end{align*}
	Similar as above, we have
	\begin{align*}
	&P_{f_0}^{(n)} \Pi_n^\ast\big(f^\ast \notin \bar{\mathcal{F}}_n^\ast\big\lvert X^{(n)}\big) \\
	&\leq \frac{1}{\Pi_n(\bar{\mathcal{F}}_n^\ast)}\cdot P_{f_0^{(n)}} \bigg[\frac{ \int_{f^\ast \notin \bar{\mathcal{F}}_n^\ast} p_{f^\ast}^{(n)}/p_{f_0}^{(n)}\ \d{\Pi_n^\ast(f^\ast)} }{\int_{\bar{\mathcal{F}}_n^\ast} p_{f^\ast}^{(n)}/p_{f_0}^{(n)}\ \d{\bar{\Pi}_n^\ast(f^\ast)} }\bm{1}_{\mathcal{E}_n}\bigg]+P_{f_0}^{(n)}(\mathcal{E}_n^c)\\ 
	&\lesssim e^{  m \log^{C_9} n } \cdot \Pi_n(\mathcal{F}\setminus \bar{\mathcal{F}}_n)+ e^{-C_6m \log^{\gamma''}n}.
	\end{align*}
	Furthermore, for $\gamma_1,\gamma_2$ large enough, 
	\begin{align*}
	&\Pi_n(\mathcal{F}\setminus \bar{\mathcal{F}}_n)\leq \Pi_\sigma \big(\sigma\notin [\underline{\sigma}_n,\bar{\sigma}_n]\big)+ \sum_{m=1}^\infty \lambda_n(m) \Pi_\mu\bigg(\max_{1\leq j\leq m} \abs{\mu_j}>b_n\bigg) \\ 
	&\lesssim e^{-\log^{(C_9+1)} n}+ \sum_{m=1}^\infty e^{-C_{10}(m-1)\log n} m \bigg(\int_{\R\setminus [-b_n,b_n]} g_\mu(x)\bigg)\lesssim e^{-\log^{C_{11}}n}.
	\end{align*}
	Hence $
	P_{f_0}^{(n)} \Pi_n^\ast\big(f^\ast \notin \bar{\mathcal{F}}_n^\ast\big\lvert X^{(n)}\big)=o(1)$ and the proof is complete.
\end{proof}

\subsection{Estimation of piecewise constant signals in the Gaussian autoregression model}
Consider fitting the Gaussian autoregression model (cf. Section \ref{section:gaussian_autoreg}) by the class of piecewise constant functions $\mathcal{F}\equiv \cup_{m=1}^\infty \mathcal{F}_m\equiv  \{f: f=\sum_{j=1}^m a_j \bm{1}_{[t_{j-1},t_j)}, -\infty=t_0<t_1<\ldots<t_{m-1}<t_m=\infty, \abs{a_j}\leq M\}$. Consider the following model selection prior $\Lambda_n$ on the model index $\mathcal{I}\equiv \N$:
\begin{align}\label{eqn:prior_aut}
\lambda_{n}(m)\propto \exp\big(-c\cdot m  \log (en)\big),
\end{align}
where $c>0$ is a constant to be specified later. Similar to the development in Section \ref{section:application_intensity}, we choose the prior $\Pi_{n,m}^t$ on $(t_1,\ldots,t_{m-1})$ with density $\bm{t}=(t_1,\ldots,t_{m-1})\mapsto (m-1)! g_t^{\otimes (m-1)} \bm{1}_{t_1<\ldots<t_{m-1}}(\bm{t})$, and the prior $\Pi_{n,m}^a$ on $(a_1,\ldots,a_m)$ with a product density $g_a^{\otimes m}$. Here we assume that $g_t$ is symmetric and non-increasing on $[0,\infty)$, and $g_a$ is uniform on $[-M,M]$ for simplicity. The difference in the Gaussian autoregression example, compared with the results in Section \ref{section:application_intensity}, is that the metric $d_{r,M}$ is defined on the entire real line $\R$. As commented on page 210 of \cite{ghosal2007convergence}, ``\emph{....The logarithmic factor in the convergence rate appears to be a consequence of the fact that the regression functions are defined on the full real line...}''. Below we perform some explicit computation to address this non-compact issue, with a particular goal of avoiding additional logarithmic factors (compared with the results in Section \ref{section:application_intensity}) in the contraction rates.

\begin{proposition}\label{prop:rate_gaussian_autoreg}
	Suppose that $M>\pnorm{f_0}{\infty}$, and that the prior density $g_t$ satisfies $
	\limsup_{x \to \infty}\frac{1}{x^2} \log (1\vee \frac{1}{g_t(x)})<\infty$. 
	Then there exists some $c>0$ in (\ref{eqn:prior_aut}) such that
	\begin{align*}
	& P_{f_0}^{(n)} \Pi_n\big(f \in \mathcal{F}: d_{r,M}^2(f,f_{0})> C_1(\epsilon_{n,m}^{\mathrm{aut}})^2\big\lvert X^{(n)}\big) \leq C_2 e^{-n (\epsilon_{n,m}^{\mathrm{aut}})^2/C_2}.\nonumber
	\end{align*}
	Here $
	(\epsilon_{n,m}^{\mathrm{aut}})^2\equiv \max\{\inf_{g \in \mathcal{F}_m}d_{r,M}^2(f_0,g),{ m\log n}/{n}\}$, and the constants $C_i(i=1,2)$ depend on $M$.
\end{proposition}

Note that the condition on $g_t$ is quite mild: it essentially requires that the tail of $g_t$ is not lighter than Gaussian. 

\begin{lemma}\label{lem:local_ent_gaussian_autoreg}
	For any $g \in \mathcal{F}_m$, and $\epsilon \in (0,1/e)$, $
	\log \mathcal{N}\big(c_5\epsilon,\{f \in \mathcal{F}_m,d_{r,M}(f,g)\leq 2\epsilon\}, d_{r,M}\big)\lesssim m\log \left(\frac{C_M \log(1/\epsilon)}{\epsilon^4}\right)$. 
\end{lemma}
\begin{proof}
	We only need to consider global entropy $\mathcal{N}\big(c_5\epsilon,\{f \in \mathcal{F}_m\}, d_{r,M}\big)$.

	Let $m\geq 2$. Fix $\epsilon\in (0,1/e)$, let $R_\epsilon= \lceil M+\sqrt{2\log (24M^2)+4\log (1/(c_5\epsilon))}\rceil$ and $\delta^2= \frac{c_5^2\epsilon^2}{4(2M^2+1)R_\epsilon}$. We partition the interval $[-R_\epsilon,R_\epsilon]$ into small intervals $\{ I_{j,\delta}\}_{j=1}^{N_\delta}$ of length $R_\epsilon\delta^2/m$ (ignoring the rounding issue here). For any $f \in \mathcal{F}_m$, let $f\equiv \sum_{j=1}^m a_j\bm{1}_{[t_{j-1},t_j)}$ for some $-\infty=t_0<t_1<\ldots<t_{m-1}<t_m=\infty$. Then $\{t_1,\ldots,t_{m-1}\}\cap [-R_\epsilon,R_\epsilon]$ must be contained in at most $m-1$ intervals amongst $\{I_{j,\delta}\}_{j=1}^{N_\delta}$, namely, $\{\bar{I}_{k,\delta;f}\}_{k=1}^{m_f}$. Furthermore, $[-R_\epsilon,R_\epsilon]\setminus \cup_{k=1}^{m_f} \bar{I}_{k,\delta;f}$ contains at most $m$ intervals. Now define $\bar{f}$ as follows: 
	\begin{align*}
	\bar{f} \equiv M\cdot \bm{1}_{\R \setminus [-R_\epsilon,R_\epsilon]}+\sum_{k=1}^{m_f} M\cdot \bm{1}_{\bar{I}_{k,\delta;f}} + \floor{\frac{f}{\delta}} \delta \cdot \bm{1}_{[-R_\epsilon,R_\epsilon]\setminus \cup_{k=1}^{m_f}\bar{I}_{k,\delta;f}}.
	\end{align*} 
	Then using the well-known fact that $\int_t^\infty \phi(x)\ \d{x}\leq e^{-t^2/2}/(\sqrt{2\pi} t) (t>0)$, we have 
	\begin{align*}
	&\int_{\R} \big(f(x)-\bar{f}(x)\big)^2r_M(x)\ \d{x}\\
	&\leq 4M^2\int_{\R \setminus [-R_\epsilon,R_\epsilon]} r_M(x)\ \d{x}+4M^2 m(R_\epsilon \delta^2/m)+\int_{[-R_\epsilon,R_\epsilon]\setminus \cup_{k=1}^{m_f}\bar{I}_{k,\delta;f}} \delta^2\ \d{x}\\
	&\leq 8M^2\int_{R_\epsilon}^\infty \phi(x-M)\ \d{x} +(4M^2+2)R_\epsilon \delta^2 \leq c_5^2 \epsilon^2
	\end{align*} 
	by our choice of $R_\epsilon$ and $\delta$. On the other hand, there are at most $
	\binom{2m/\delta^2}{m-1}\cdot \big(\frac{2M}{\delta}\big)^m$ many choices of $\bar{f}$, and hence 
	\begin{align*}
	\log \mathcal{N}\big(c_5\epsilon,\{f \in \mathcal{F}_m\}, d_{r,M}\big)\leq \log \bigg[ \binom{2m/\delta^2}{m-1}\cdot \big(\frac{2M}{\delta}\big)^m\bigg]\lesssim m\log \bigg(\frac{C_M \log(1/\epsilon)}{\epsilon^4}\bigg).
	\end{align*} 
	For $m=1$, we define $
	\bar{f} \equiv M\cdot \bm{1}_{\R \setminus [-R_\epsilon,R_\epsilon]}+ \floor{\frac{f}{\delta}} \delta \cdot \bm{1}_{[-R_\epsilon,R_\epsilon]}$, 
	and repeat the above calculation to see that the entropy bound holds.
\end{proof}

Hence we can take $\delta_{n,m}^2=C {m\log n}/{n}$ for some large constant $C>0$. 

\begin{lemma}\label{lem:submodel_mass_gaussian_autoreg}
	Let $M>\pnorm{f_0}{\infty}$. Suppose $g_t$ is such that $
	\limsup_{x \to \infty}\frac{1}{x^2} \log (1\vee \frac{1}{g_t(x)})<\infty$. 
	For $n$ large enough depending on $\pnorm{f_0}{\infty}$ and $M$, (P2) in Assumption \ref{assump:prior_mass_general} holds.
\end{lemma}
\begin{proof}
	The proof uses similar ideas as that of Lemma \ref{lem:submodel_suff_prior_intensity}. Let $f_{0,m}\equiv \sum_{j=1}^m a_j^\ast \bm{1}_{[t^\ast_{j-1},t^\ast_j)}$ for some $t^\ast=(t_1^\ast,\ldots,t_{m-1}^\ast)$ with $-\infty=t_0^\ast<t_1^\ast<\ldots<t_{m-1}^\ast<t_m^\ast=\infty$. Without loss of generality, we may assume that $\min\{t_j^\ast-t_{j-1}^\ast: 2\leq j\leq m-1\}>1/(4nM^2)$ (otherwise we may merge such short intervals to construct a surrogate $\tilde{f}_{0,m}$, and the total difference between $\tilde{f}_{0,m}$ and $f_{0,m}$ in squared $L_2$ metric by doing this does not exceed $m/n$ so that there is no effect in the final oracle inequality). For any $t=(t_1,\ldots,t_{m-1})$ such that $\abs{t_j-t_j^\ast}<1/(8nM^2)$ where $\abs{t_j^\ast}\leq L_n$ with $L_n$ specified later on, and any $a=(a_1,\ldots,a_m)$ such that $\max_j\abs{a_j-a_j^\ast}\leq 1/\sqrt{n}$, let $f\equiv \sum_{j=1}^m a_j\bm{1}_{[t_{j-1},t_j)}$. Then, $\pnorm{f}{\infty}\leq M$ for $n$ large enough depending only through $\pnorm{f_0}{\infty}$ and $M$. Now with $L_n\equiv M+\sqrt{2\log (8M^2)+2\log n}$, 
	\begin{align*}
	&\int_{-\infty}^\infty \big(f(x)-f_{0,m}(x)\big)^2r_M(x)\ \d{x}\\
	&\leq 8M^2\int_{L_n}^\infty \phi(x-M)\ \d{x} + \int_{-L_n}^{L_n}\big(f(x)-f_{0,m}(x)\big)^2r_M(x)\ \d{x} \\
	&\leq \frac{1}{n}+\big(\frac{1}{n}+4M^2\cdot m\cdot \frac{1}{8nM^2}\big)\leq \frac{3m}{n} \leq \delta_{n,m}^2/c_3
	\end{align*} 
	by choosing the constant $C=C_M>0$ in the definition of $\delta_{n,m}^2$ large enough. This implies that 
	\begin{align*}
	&\Pi_{n,m} \big(\{f \in \mathcal{F}_m:  d_{r,M}^2(f,f_{0,m})\leq \delta_{n,m}^2/c_3\}\big)\\ 
	&\geq g_t(L_n)^{m-1}(16nM^2)^{-(m-1)}\bigg(\frac{2}{\sqrt{n}}\bigg)^{m}\\
	&\geq e^{-m \big(\log g_t(L_n)^{-1}+\log (16nM^2)+\log \big(\sqrt{n}/2\big)\big) }\geq e^{-2n\delta_{n,m}^2}
	\end{align*} 
	by the assumption on $g_t$ and again choosing the constant $C=C_M>0$ in the definition of $\delta_{n,m}^2$ large enough.
\end{proof}

\begin{proof}[Proof of Proposition \ref{prop:rate_gaussian_autoreg}]
	Proposition \ref{prop:rate_gaussian_autoreg} follows from Corollary \ref{cor:rate_auto} combined with Lemmas \ref{lem:local_ent_gaussian_autoreg} and \ref{lem:submodel_mass_gaussian_autoreg}.
\end{proof}


\subsection{Subset selection for sparse approximation of functions}

Consider Gaussian regression with random design $Y_i = f_0(X_i)+\epsilon_i(1\leq i\leq n)$. We assume that $X_i$'s are i.i.d. uniformly distributed on $[0,1]$ and are independent of $\epsilon_i$'s for simplicity of discussion. Let $\{\phi_k\}_{k=1}^\infty$ be an orthonormal basis of $L_2([0,1])$. Let $\bm{N}\equiv \{N_1,N_2,\ldots\}\subset \N$. For any $\bm{\gamma}\equiv (\gamma_0,\gamma_1,\ldots)$, let the $\bm{\gamma}$-sparse approximation space $\mathcal{S}(\bm{\gamma},\bm{N})\equiv \{f\in L_2([0,1]): \min_{\ell_j\leq N_j, 1\leq j\leq k}\min_{(a_{\ell_1},\ldots,a_{\ell_k})} \pnorm{f-\sum_{j=1}^k a_{\ell_j} \phi_{\ell_j}}{L_2([0,1])}\leq \gamma_k, k =0,1,\ldots\}$. For any $\bm{\gamma}$ and $k \in \N$, let $\bm{\gamma}^{(k)}\equiv (\gamma_0,\gamma_1,\ldots,\gamma_{k-1},0,0,\ldots)$, and $\mathcal{F}_k\equiv \mathcal{S}(\bm{\gamma}^{(k)},\bm{N})$. Clearly $\mathcal{F}_1\subset \mathcal{F}_2\subset \cdots$. We use the model selection prior:
\begin{align}\label{eqn:prior_ss}
\lambda_{n}(k)\propto \exp\big(-c\cdot k \log(e n)\big).
\end{align}
For each model $\mathcal{F}_k$, we use the prior $\Pi_{n,k}$ which first picks randomly a subset $I \subset \{1,\ldots,N_k\}$ with cardinality $k$, then puts a product prior $g^{\otimes \abs{I}}$ on the coefficients $(a_j)_{j \in I}$. We assume for simplicity that $g$ is symmetric and non-decreasing on $[0,\infty)$. Note that in Section 6.3 of \cite{yang1999model}, the model index corresponds to $(k,I)$ in our notation.

\begin{proposition}
	Let $f_0 \in \mathcal{S}(\bm{\gamma},\bm{N})$ be such that $\sup_k \abs{\int f_0 \phi_k}<\infty$. Suppose the priors are specified as above and $g$ satisfies $g\big(\sup_k \abs{\int f_0 \phi_k}+1\big)>0$. Then if $\log N_k\lesssim \log k$, with $\epsilon_{n,k}^2\equiv \gamma_k + k \log (N_k\vee n)/n$, for $n$ large,
	\begin{align*}
	P_{f_0}^{(n)}\Pi_n\big(f: L_2^2(f,f_0)>C_1 \epsilon_{n,k}^2\big\lvert X^{(n)},Y^{(n)}\big)\leq C_2 e^{-n\epsilon_{n,k}^2/C_2}.
	\end{align*}
	The constants $C_i(i=1,2)$ do not depend on $k$.
\end{proposition}

\begin{proof}
	We only sketch the proof. For the entropy condition, we claim that for any $g \in \mathcal{F}_k$, 
	\begin{align*}
	\log \mathcal{N}(c_5\epsilon, \{f \in \mathcal{F}_k: L_2(f,g)\leq 2\epsilon\},L_2)\leq C_{c_5} k \log (e N_k).
	\end{align*}
	To see this, the entropy can be bounded by 
	\begin{align*}
	\log \bigg[\binom{N_k}{k}\max_{I\subset \{1,\ldots,N_k\},\abs{I}=k} \mathcal{N}(c_5\epsilon, \{f \in \mathcal{F}_{k,I}:L_2(f,g)\leq 2\epsilon\},L_2)\bigg], 
	\end{align*}
	where $\mathcal{F}_{k,I}\equiv \{f =\sum_{\ell_j \in I} a_{\ell_j} \phi_{\ell_j}\}$. Now we may use the standard entropy bound for Euclidean balls to conclude.  The sufficient mass condition can be checked along similar lines as many examples before, by using $f_{0,k}\in \mathcal{F}_k$ as the best linear approximation amongst $\{\sum_{\ell_j \in I} a_{\ell_j} \phi_{\ell_j}: I \subset \{1,\ldots,N_k\},\abs{I}=k\}$, and $\delta_{n,k}^2\equiv C k \log (N_k\vee n)/n$ for a large enough constant $C>0$.
\end{proof}

It is straightforward from here to compute a more concrete contraction rate by specifying concrete orders of $\bm{\gamma},\bm{N}$. Details are omitted.

The above proposition holds for a pre-specified $\bm{N}$. Let us now consider `adaptation' problem with respect to $\bm{N}$. We will consider this in the framework of Corollary 2 of \cite{yang1999model}. Let $\bm{N}^{(1)} \equiv (N_1^{(1)},N_2^{(1)},\ldots)$ and $\bm{N}^{(2)}\equiv (N_1^{(2)},N_2^{(2)},\ldots)$ where $N_k^{(2)}\geq N_k^{(1)}$ and $\log N_k^{(i)} \lesssim \log k$. In this case, we  may formulate formally two models: $\tilde{\mathcal{F}}_1\equiv \mathcal{S}(\bm{\gamma},\bm{N}^{(1)})$ and $\tilde{\mathcal{F}}_2\equiv \mathcal{S}(\bm{\gamma},\bm{N}^{(1)}) \cup \mathcal{S}(\bm{\gamma},\bm{N}^{(2)})$, and we put a uniform prior on the index $\{1,2\}$. The prior $\tilde{\Pi}_i$ on $\tilde{\mathcal{F}}_i$ is given by $\sum_k \lambda_n(k) \Pi_{n,k}(\bm{N}^{(i)})$ as specified above in (\ref{eqn:prior_ss}) and satisfies the conditions in the proceeding proposition (so the prior on $\tilde{\mathcal{F}}_2$ only charges mass on $\mathcal{S}(\bm{\gamma},\bm{N}^{(2)})$).  Let $\gamma_k = k^{-\alpha}$ for $\alpha>0$. 

\begin{proposition}
	Consider the above setup. Let $f_0 \in L_2([0,1])$ be such that $\sup_k \abs{\int f_0 \phi_k}<\infty$. Then with $\epsilon_{n,\alpha}^2\equiv (\log n/n)^{2\alpha/(2\alpha+1)}$, for $f \in \mathcal{S}(\bm{\gamma},\bm{N}^{(1)}) \cup \mathcal{S}(\bm{\gamma},\bm{N}^{(2)})$, 
	\begin{align*}
	P_{f_0}^{(n)}\Pi_n\big(f: L_2^2(f,f_0)>C_1 \epsilon_{n,\alpha}^2 \big\lvert X^{(n)},Y^{(n)}\big)\leq C_2 e^{-n\epsilon_{n,\alpha}^2 /C_2}
	\end{align*}
	holds for $n$ large enough.
\end{proposition}
\begin{proof}
	Let $f_{0,i}$ be the best linear approximation amongst $\{\sum_{\ell_j \in I} a_{\ell_j} \phi_{\ell_j}: I \subset \{1,\ldots,N_{k_n}^{(i)}\},\abs{I}=k_n\}$, so $\pnorm{f_0-f_{0,i}}{L_2([0,1])}^2\leq \gamma_{k_n}^2 (i=1,2)$, where $k_n=(n/\log n)^{1/(1+2\alpha)}$. In particular, write $f_{0,i}=\sum_{\ell_j^{(i)} \in I^{(i)}} a_{\ell_j^{(i)}}\phi_{\ell_j^{(i)}}$. Using the result on page 1586 of \cite{yang1999information}, $\log \mathcal{N}(c_5\epsilon, \tilde{\mathcal{F}}_i, L_2)\lesssim \epsilon^{-1/\alpha}\log(1/\epsilon)$. So we may take $\delta_{n,i}^2\equiv C(n/\log n)^{-2\alpha/(2\alpha+1)}$ for $i=1,2$ and a large constant $C>0$. To verify the sufficient mass condition, note that 
	\begin{align*}
	&\tilde{\Pi}_i(\{f \in \tilde{\mathcal{F}}_i: L_2^2(f,f_{0,i})\leq \delta_{n,i}^2/c_3\})\\
	&\geq \lambda_n(k_n)\cdot\binom{N_{k_n}}{k_n}^{-1}\big(\delta_{n,i}/\sqrt{c_3}\big)^{k_n}g\big(\sup_k \abs{\int f_0 \phi_k}+1\big)^{k_n}\geq e^{-2n\delta_{n,k_n}^2}
	\end{align*}
	by choosing $C>0$ large enough. 
\end{proof}

The proposition shows that under the specified prior, it is indeed possible to achieve adaptive rate over $\mathcal{S}(\bm{\gamma},\bm{N}^{(1)}) \cup \mathcal{S}(\bm{\gamma},\bm{N}^{(2)})$. It is straightforward to extend this result to multiple lists of models so we omit the details.

\section*{Acknowledgements}
The author is indebted to Chao Gao, Johannes Schmidt-Hieber, and the anonymous referees for numerous comments and suggestions that lead to substantial improvements of an earlier version of the paper.  The author would also like to thank Jon Wellner for constant support and continuous encouragement as this work developed.

\bibliographystyle{amsalpha}
\bibliography{mybib}

\providecommand{\bysame}{\leavevmode\hbox to3em{\hrulefill}\thinspace}
\providecommand{\MR}{\relax\ifhmode\unskip\space\fi MR }
\providecommand{\MRhref}[2]{%
  \href{http://www.ams.org/mathscinet-getitem?mr=#1}{#2}
}
\providecommand{\href}[2]{#2}
\begin{thebibliography}{CSHvdV15}

\bibitem[ACCR14]{alquier2014bayesian}
Pierre Alquier, Vincent Cottet, Nicolas Chopin, and Judith Rousseau,
  \emph{Bayesian matrix completion: prior specification}, arXiv preprint
  arXiv:1406.1440 (2014).

\bibitem[AGR13]{arbel2013bayesian}
Julyan Arbel, Ghislaine Gayraud, and Judith Rousseau, \emph{Bayesian optimal
  adaptive estimation using a sieve prior}, Scand. J. Stat. \textbf{40} (2013),
  no.~3, 549--570. \MR{3091697}

\bibitem[BBM99]{barron1999risk}
Andrew Barron, Lucien Birg{\'e}, and Pascal Massart, \emph{Risk bounds for
  model selection via penalization}, Probab. Theory Related Fields \textbf{113}
  (1999), no.~3, 301--413. \MR{1679028 (2000k:62049)}

\bibitem[BC91]{barron1991minimum}
Andrew~R. Barron and Thomas~M. Cover, \emph{Minimum complexity density
  estimation}, IEEE Trans. Inform. Theory \textbf{37} (1991), no.~4,
  1034--1054. \MR{1111806}

\bibitem[Bel17]{belitser2017coverage}
Eduard Belitser, \emph{On coverage and local radial rates of credible sets},
  Ann. Statist. \textbf{45} (2017), no.~3, 1124--1151. \MR{3662450}

\bibitem[BG03]{belitser2003adaptive}
Eduard Belitser and Subhashis Ghosal, \emph{Adaptive {B}ayesian inference on
  the mean of an infinite-dimensional normal distribution}, Ann. Statist.
  \textbf{31} (2003), no.~2, 536--559, Dedicated to the memory of Herbert E.
  Robbins. \MR{1983541}

\bibitem[BG14]{banerjee2014posterior}
Sayantan Banerjee and Subhashis Ghosal, \emph{Posterior convergence rates for
  estimating large precision matrices using graphical models}, Electron. J.
  Stat. \textbf{8} (2014), no.~2, 2111--2137. \MR{3273620}

\bibitem[BG15]{banerjee2015bayesian}
\bysame, \emph{Bayesian structure learning in graphical models}, J.
  Multivariate Anal. \textbf{136} (2015), 147--162. \MR{3321485}

\bibitem[BLM13]{boucheron2013concentration}
St\'ephane Boucheron, G\'abor Lugosi, and Pascal Massart, \emph{Concentration
  inequalities: {A} nonasymptotic theory of independence}, Oxford University
  Press, Oxford, 2013. \MR{3185193}

\bibitem[BM93]{birge1993rates}
Lucien Birg{\'e} and Pascal Massart, \emph{Rates of convergence for minimum
  contrast estimators}, Probab. Theory Related Fields \textbf{97} (1993),
  no.~1-2, 113--150. \MR{1240719 (94m:62095)}

\bibitem[BvdG11]{buhlmann2011statistics}
Peter B{\"u}hlmann and Sara van~de Geer, \emph{Statistics for high-dimensional
  data}, Springer Series in Statistics, Springer, Heidelberg, 2011, Methods,
  theory and applications. \MR{2807761}

\bibitem[Cas14]{castillo2014bayesian}
Isma{\"e}l Castillo, \emph{On {B}ayesian supremum norm contraction rates}, Ann.
  Statist. \textbf{42} (2014), no.~5, 2058--2091. \MR{3262477}

\bibitem[CGR04]{choudhuri2004bayesian}
Nidhan Choudhuri, Subhashis Ghosal, and Anindya Roy, \emph{Bayesian estimation
  of the spectral density of a time series}, J. Amer. Statist. Assoc.
  \textbf{99} (2004), no.~468, 1050--1059. \MR{2109494}

\bibitem[CGS15]{chatterjee2015risk}
Sabyasachi Chatterjee, Adityanand Guntuboyina, and Bodhisattva Sen, \emph{On
  risk bounds in isotonic and other shape restricted regression problems}, Ann.
  Statist. \textbf{43} (2015), no.~4, 1774--1800. \MR{3357878}

\bibitem[CP11]{candes2011tight}
Emmanuel~J. Cand{\`e}s and Yaniv Plan, \emph{Tight oracle inequalities for
  low-rank matrix recovery from a minimal number of noisy random measurements},
  IEEE Trans. Inform. Theory \textbf{57} (2011), no.~4, 2342--2359.
  \MR{2809094}

\bibitem[CSHvdV15]{castillo2015bayesian}
Isma{\"e}l Castillo, Johannes Schmidt-Hieber, and Aad van~der Vaart,
  \emph{Bayesian linear regression with sparse priors}, Ann. Statist.
  \textbf{43} (2015), no.~5, 1986--2018. \MR{3375874}

\bibitem[CT05]{candes2005decoding}
Emmanuel~J. Cand{\`e}s and Terence Tao, \emph{Decoding by linear programming},
  IEEE Trans. Inform. Theory \textbf{51} (2005), no.~12, 4203--4215.
  \MR{2243152}

\bibitem[CvdV12]{castillo2012needles}
Isma{\"e}l Castillo and Aad van~der Vaart, \emph{Needles and straw in a
  haystack: posterior concentration for possibly sparse sequences}, Ann.
  Statist. \textbf{40} (2012), no.~4, 2069--2101. \MR{3059077}

\bibitem[dJvZ10]{deJonge2010adaptive}
R.~de~Jonge and J.~H. van Zanten, \emph{Adaptive nonparametric {B}ayesian
  inference using location-scale mixture priors}, Ann. Statist. \textbf{38}
  (2010), no.~6, 3300--3320. \MR{2766853}

\bibitem[DRRS18]{donnet2018posterior}
Sophie Donnet, Vincent Rivoirard, Judith Rousseau, and Catia Scricciolo,
  \emph{Posterior concentration rates for empirical {B}ayes procedures with
  applications to {D}irichlet process mixtures}, Bernoulli \textbf{24} (2018),
  no.~1, 231--256. \MR{3706755}

\bibitem[GGvdV00]{ghosal2000convergence}
Subhashis Ghosal, Jayanta~K. Ghosh, and Aad van~der Vaart, \emph{Convergence
  rates of posterior distributions}, Ann. Statist. \textbf{28} (2000), no.~2,
  500--531. \MR{1790007}

\bibitem[GLvdV08]{ghosal2008nonparametric}
Subhashis Ghosal, J{\"u}ri Lember, and Aad van~der Vaart, \emph{Nonparametric
  {B}ayesian model selection and averaging}, Electron. J. Stat. \textbf{2}
  (2008), 63--89. \MR{2386086}

\bibitem[GS13]{guntuboyina2013covering}
Adityanand Guntuboyina and Bodhisattva Sen, \emph{Covering numbers for convex
  functions}, IEEE Trans. Inform. Theory \textbf{59} (2013), no.~4, 1957--1965.
  \MR{3043776}

\bibitem[Gun12]{guntuboyina2012optimal}
Adityanand Guntuboyina, \emph{Optimal rates of convergence for convex set
  estimation from support functions}, Ann. Statist. \textbf{40} (2012), no.~1,
  385--411. \MR{3014311}

\bibitem[GvdV01]{ghosal2001entropies}
Subhashis Ghosal and Aad~W. van~der Vaart, \emph{Entropies and rates of
  convergence for maximum likelihood and {B}ayes estimation for mixtures of
  normal densities}, Ann. Statist. \textbf{29} (2001), no.~5, 1233--1263.
  \MR{1873329}

\bibitem[GvdV07a]{ghosal2007convergence}
Subhashis Ghosal and Aad van~der Vaart, \emph{Convergence rates of posterior
  distributions for non-i.i.d. observations}, Ann. Statist. \textbf{35} (2007),
  no.~1, 192--223. \MR{2332274}

\bibitem[GvdV07b]{ghosal2007posterior}
\bysame, \emph{Posterior convergence rates of {D}irichlet mixtures at smooth
  densities}, Ann. Statist. \textbf{35} (2007), no.~2, 697--723. \MR{2336864}

\bibitem[GvdV17]{ghosal2017fundamentals}
\bysame, \emph{Fundamentals of nonparametric {B}ayesian inference}, Cambridge
  Series in Statistical and Probabilistic Mathematics, vol.~44, Cambridge
  University Press, Cambridge, 2017. \MR{3587782}

\bibitem[GvdVZ15]{gao2015general}
Chao Gao, Aad van~der Vaart, and Harrison~H Zhou, \emph{A general framework for
  bayes structured linear models}, arXiv preprint arXiv:1506.02174 (2015).

\bibitem[GZ15]{gao2015ratepca}
Chao Gao and Harrison~H. Zhou, \emph{Rate-optimal posterior contraction for
  sparse {PCA}}, Ann. Statist. \textbf{43} (2015), no.~2, 785--818.
  \MR{3325710}

\bibitem[GZ16]{gao2016rate}
\bysame, \emph{Rate exact {B}ayesian adaptation with modified block priors},
  Ann. Statist. \textbf{44} (2016), no.~1, 318--345. \MR{3449770}

\bibitem[HD11]{hannah2011bayesian}
Lauren~A Hannah and David~B Dunson, \emph{Bayesian nonparametric multivariate
  convex regression}, arXiv preprint arXiv:1109.0322 (2011).

\bibitem[HH03]{holmes2003generalized}
CC~Holmes and NA~Heard, \emph{Generalized monotonic regression using random
  change points}, Statistics in Medicine \textbf{22} (2003), no.~4, 623--638.

\bibitem[HRSH15]{hoffmann2015adaptive}
Marc Hoffmann, Judith Rousseau, and Johannes Schmidt-Hieber, \emph{On adaptive
  posterior concentration rates}, Ann. Statist. \textbf{43} (2015), no.~5,
  2259--2295. \MR{3396985}

\bibitem[HW16]{han2016multivariate}
Qiyang Han and Jon~A. Wellner, \emph{Multivariate convex regression: global
  risk bounds and adaptation}, arXiv preprint arXiv:1601.06844 (2016).

\bibitem[KRvdV10]{kruijer2010adaptive}
Willem Kruijer, Judith Rousseau, and Aad van~der Vaart, \emph{Adaptive
  {B}ayesian density estimation with location-scale mixtures}, Electron. J.
  Stat. \textbf{4} (2010), 1225--1257. \MR{2735885}

\bibitem[KvdV06]{kleijn2006misspecification}
B.~J.~K. Kleijn and Aad van~der Vaart, \emph{Misspecification in
  infinite-dimensional {B}ayesian statistics}, Ann. Statist. \textbf{34}
  (2006), no.~2, 837--877. \MR{2283395}

\bibitem[LD14]{lin2014bayesian}
Lizhen Lin and David~B. Dunson, \emph{Bayesian monotone regression using
  {G}aussian process projection}, Biometrika \textbf{101} (2014), no.~2,
  303--317. \MR{3215349}

\bibitem[LG17]{li2017bayesian}
Meng Li and Subhashis Ghosal, \emph{Bayesian detection of image boundaries},
  Ann. Statist. \textbf{45} (2017), no.~5, 2190--2217. \MR{3718166}

\bibitem[LvdV07]{lember2007universal}
J\"uri Lember and Aad van~der Vaart, \emph{On universal {B}ayesian adaptation},
  Statist. Decisions \textbf{25} (2007), no.~2, 127--152. \MR{2388859}

\bibitem[MA15]{alquier2015bayesian}
The~Tien Mai and Pierre Alquier, \emph{A {B}ayesian approach for noisy matrix
  completion: optimal rate under general sampling distribution}, Electron. J.
  Stat. \textbf{9} (2015), no.~1, 823--841. \MR{3331862}

\bibitem[Mas07]{massart2007concentration}
Pascal Massart, \emph{Concentration inequalities and model selection}, Lecture
  Notes in Mathematics, vol. 1896, Springer, Berlin, 2007, Lectures from the
  33rd Summer School on Probability Theory held in Saint-Flour, July 6--23,
  2003, With a foreword by Jean Picard. \MR{2319879 (2010a:62008)}

\bibitem[MR13]{meister2013asymptotic}
Alexander Meister and Markus Rei\ss, \emph{Asymptotic equivalence for
  nonparametric regression with non-regular errors}, Probab. Theory Related
  Fields \textbf{155} (2013), no.~1-2, 201--229. \MR{3010397}

\bibitem[MRS20]{mariucci2017bayesian}
Ester Mariucci, Kolyan Ray, and Botond Szab\'{o}, \emph{A {B}ayesian
  nonparametric approach to log-concave density estimation}, Bernoulli
  \textbf{26} (2020), no.~2, 1070--1097. \MR{4058361}

\bibitem[ND04]{neelon2004bayesian}
Brian Neelon and David~B. Dunson, \emph{Bayesian isotonic regression and trend
  analysis}, Biometrics \textbf{60} (2004), no.~2, 398--406. \MR{2066274}

\bibitem[PBPD14]{pati2014posterior}
Debdeep Pati, Anirban Bhattacharya, Natesh~S. Pillai, and David Dunson,
  \emph{Posterior contraction in sparse {B}ayesian factor models for massive
  covariance matrices}, Ann. Statist. \textbf{42} (2014), no.~3, 1102--1130.
  \MR{3210997}

\bibitem[Pol90]{pollard1990empirical}
David Pollard, \emph{Empirical processes: theory and applications}, NSF-CBMS
  Regional Conference Series in Probability and Statistics, 2, Institute of
  Mathematical Statistics, Hayward, CA; American Statistical Association,
  Alexandria, VA, 1990. \MR{1089429 (93e:60046)}

\bibitem[RCL12]{rousseau2012bayesian}
Judith Rousseau, Nicolas Chopin, and Brunero Liseo, \emph{Bayesian
  nonparametric estimation of the spectral density of a long or intermediate
  memory {G}aussian process}, Ann. Statist. \textbf{40} (2012), no.~2,
  964--995. \MR{2985940}

\bibitem[RFP10]{recht2010guaranteed}
Benjamin Recht, Maryam Fazel, and Pablo~A. Parrilo, \emph{Guaranteed
  minimum-rank solutions of linear matrix equations via nuclear norm
  minimization}, SIAM Rev. \textbf{52} (2010), no.~3, 471--501. \MR{2680543}

\bibitem[Rou10]{rousseau2010rates}
Judith Rousseau, \emph{Rates of convergence for the posterior distributions of
  mixtures of betas and adaptive nonparametric estimation of the density}, Ann.
  Statist. \textbf{38} (2010), no.~1, 146--180. \MR{2589319}

\bibitem[RS16]{rousseau2016asymptotic}
Judith Rousseau and Botond Szabo, \emph{Asymptotic frequentist coverage
  properties of bayesian credible sets for sieve priors in general settings},
  arXiv preprint arXiv:1609.05067 (2016).

\bibitem[RS17]{rousseau2017asymptotic}
\bysame, \emph{Asymptotic behaviour of the empirical {B}ayes posteriors
  associated to maximum marginal likelihood estimator}, Ann. Statist.
  \textbf{45} (2017), no.~2, 833--865. \MR{3650402}

\bibitem[RSH17]{reiss2017nonparametric}
Markus Reiss and Johannes Schmidt-Hieber, \emph{Nonparametric bayesian analysis
  for support boundary recovery}, arXiv preprint arXiv:1703.08358 (2017).

\bibitem[RT11]{rohde2011estimation}
Angelika Rohde and Alexandre~B. Tsybakov, \emph{Estimation of high-dimensional
  low-rank matrices}, Ann. Statist. \textbf{39} (2011), no.~2, 887--930.
  \MR{2816342}

\bibitem[Sal14]{salomond2014adaptive}
Jean-Bernard Salomond, \emph{Adaptive {B}ayes test for monotonicity}, The
  contribution of young researchers to {B}ayesian statistics, Springer Proc.
  Math. Stat., vol.~63, Springer, Cham, 2014, pp.~29--33. \MR{3133254}

\bibitem[Scr06]{scricciolo2006convergence}
Catia Scricciolo, \emph{Convergence rates for {B}ayesian density estimation of
  infinite-dimensional exponential families}, Ann. Statist. \textbf{34} (2006),
  no.~6, 2897--2920. \MR{2329472}

\bibitem[Scr16]{scricciolo2016rates}
\bysame, \emph{Rates for {B}ayesian estimation of location-scale mixtures of
  super-smooth densities}, Topics in theoretical and applied statistics, Stud.
  Theor. Appl. Stat. Sel. Papers Stat. Soc., Springer, Cham, 2016, pp.~49--57.
  \MR{3838069}

\bibitem[SSW09]{shively2009bayesian}
Thomas~S. Shively, Thomas~W. Sager, and Stephen~G. Walker, \emph{A {B}ayesian
  approach to non-parametric monotone function estimation}, J. R. Stat. Soc.
  Ser. B Stat. Methodol. \textbf{71} (2009), no.~1, 159--175. \MR{2655528}

\bibitem[SW01]{shen2001rates}
Xiaotong Shen and Larry Wasserman, \emph{Rates of convergence of posterior
  distributions}, Ann. Statist. \textbf{29} (2001), no.~3, 687--714.
  \MR{1865337}

\bibitem[Tsy14]{tsybakov2014aggregation}
Alexandre~B Tsybakov, \emph{Aggregation and minimax optimality in
  high-dimensional estimation}, Proceedings of the International Congress of
  Mathematicians, 2014, pp.~225--246.

\bibitem[vdG00]{van2000empirical}
Sara van~de Geer, \emph{Applications of {E}mpirical {P}rocess {T}heory},
  Cambridge Series in Statistical and Probabilistic Mathematics, vol.~6,
  Cambridge University Press, Cambridge, 2000. \MR{1739079 (2001h:62002)}

\bibitem[vdVvZ08]{van2008rates}
Aad van~der Vaart and J.~H. van Zanten, \emph{Rates of contraction of posterior
  distributions based on {G}aussian process priors}, Ann. Statist. \textbf{36}
  (2008), no.~3, 1435--1463. \MR{2418663}

\bibitem[vdVvZ09]{van2009adaptive}
\bysame, \emph{Adaptive {B}ayesian estimation using a {G}aussian random field
  with inverse gamma bandwidth}, Ann. Statist. \textbf{37} (2009), no.~5B,
  2655--2675. \MR{2541442}

\bibitem[vdVW96]{van1996weak}
Aad van~der Vaart and Jon~A. Wellner, \emph{Weak {C}onvergence and {E}mpirical
  {P}rocesses}, Springer Series in Statistics, Springer-Verlag, New York, 1996.
  \MR{1385671 (97g:60035)}

\bibitem[Yan99]{yang1999model}
Yuhong Yang, \emph{Model selection for nonparametric regression}, Statist.
  Sinica \textbf{9} (1999), no.~2, 475--499. \MR{1707850}

\bibitem[YB98]{yang1998asymptotic}
Yuhong Yang and Andrew~R. Barron, \emph{An asymptotic property of model
  selection criteria}, IEEE Trans. Inform. Theory \textbf{44} (1998), no.~1,
  95--116. \MR{1486651}

\bibitem[YB99]{yang1999information}
Yuhong Yang and Andrew Barron, \emph{Information-theoretic determination of
  minimax rates of convergence}, Ann. Statist. \textbf{27} (1999), no.~5,
  1564--1599. \MR{1742500}

\bibitem[YG16]{yoo2016supremum}
William~Weimin Yoo and Subhashis Ghosal, \emph{Supremum norm posterior
  contraction and credible sets for nonparametric multivariate regression},
  Ann. Statist. \textbf{44} (2016), no.~3, 1069--1102. \MR{3485954}

\bibitem[YLC19]{yu2016minimax}
Zhuqing Yu, Michael Levine, and Guang Cheng, \emph{Minimax optimal estimation
  in partially linear additive models under high dimension}, Bernoulli
  \textbf{25} (2019), no.~2, 1289--1325. \MR{3920373}

\bibitem[YP17]{yang2017bayesian}
Yun Yang and Debdeep Pati, \emph{Bayesian model selection consistency and
  oracle inequality with intractable marginal likelihood}, arXiv preprint
  arXiv:1701.00311 (2017).

\bibitem[YZ16]{yuan2015minimax}
Ming Yuan and Ding-Xuan Zhou, \emph{Minimax optimal rates of estimation in high
  dimensional additive models}, Ann. Statist. \textbf{44} (2016), no.~6,
  2564--2593. \MR{3576554}

\end{thebibliography}

\end{document}